\documentclass[10pt,a4paper,oneside]{amsart}

\usepackage{cancel}
    
\usepackage{amsmath} 
\usepackage{amssymb} 
\usepackage{mathrsfs}
\usepackage{amsthm} 
\usepackage{stmaryrd} 
\usepackage[english]{babel} 
\usepackage[font=small,justification=centering]{caption} 
\usepackage[nodayofweek]{datetime}
\usepackage[shortlabels]{enumitem} 
\usepackage[T1]{fontenc} 
\usepackage[utf8]{inputenc} 
\usepackage{ifthen} 
\usepackage{mathabx} 
\usepackage{mathtools} 
\usepackage{MnSymbol}
\usepackage[dvipsnames]{xcolor} 
\usepackage[pdftex,  colorlinks=true, pagebackref=true]{hyperref} 
    \hypersetup{urlcolor=RoyalBlue, linkcolor=RoyalBlue,  citecolor=black}
\usepackage{setspace} 
\usepackage{xcolor}
\usepackage{tikz-cd} 
\usepackage{xfrac} 
\usepackage[capitalize]{cleveref} 
\usepackage{csquotes}

\makeatletter
\@namedef{subjclassname@1991}{Mathematical subject classification 1991}
\@namedef{subjclassname@2000}{Mathematical subject classification 2000}
\@namedef{subjclassname@2010}{Mathematical subject classification 2010}
\@namedef{subjclassname@2020}{Mathematical subject classification 2020}
\makeatother

\newtheorem{thm}{Theorem}[section]
\newtheorem{lemma}[thm]{Lemma}
\newtheorem{corollary}[thm]{Corollary}
\newtheorem{prop}[thm]{Proposition}

\newtheorem{claim}[thm]{Claim}

\newtheorem{thmx}{Theorem}

\newtheorem{propx}[thmx]{Proposition}

\theoremstyle{definition}
\newtheorem{defn}[thm]{Definition}
\newtheorem{remark}[thm]{Remark}

\theoremstyle{plain}

    \newtheoremstyle{TheoremNum}
        {8.0pt plus 2.0pt minus 4.0pt}{8.0pt plus 2.0pt minus 4.0pt} 
        {\itshape} 
        {-0.15cm} 
        {\bfseries} 
        {.} 
        { }  
        {\thmname{#1}\thmnote{ \bfseries #3}}
    \theoremstyle{TheoremNum}
    \newtheorem{duplicate}{}

\newcommand*{\claimproofname}{My proof}

\renewcommand*{\backref}[1]{}
\renewcommand*{\backrefalt}[4]
{
    \ifcase #1
        No citation in the text.
    \or
        Cited on Page #2.
    \else
        Cited on Pages #2.
    \fi
}

\newcounter{commentcounter}

\DeclareMathOperator*{\aster}{\raisebox{-0.3ex}{\textnormal{\LARGE{\textasteriskcentered}}}}

\DeclareMathOperator{\Aut}{\mathrm{Aut}}
\DeclareMathOperator{\Out}{\mathrm{Out}}

\DeclareMathOperator{\Fix}{\mathrm{Fix}}

\DeclareMathOperator{\Isom}{\mathrm{Isom}}

\newcommand{\calc}{{\mathcal{A}}}
\newcommand{\cald}{{\mathcal{D}}}
\newcommand{\calg}{{\mathcal{G}}}
\newcommand{\calh}{{\mathcal{H}}}
\newcommand{\calk}{{\mathcal{K}}}
\newcommand{\calx}{{\mathcal{X}}}

\newcommand{\FF}{\mathbb{F}}

\newcommand{\GL}{\mathrm{GL}}

\newcommand{\bass}{\mathrm{Bass}}

\newcommand{\Z}{\mathbb{Z}}
\newcommand{\ZZ}{\mathbb{Z}}
\newcommand{\NN}{\mathbb{N}}
\newcommand{\onto}{\twoheadrightarrow}
\newcommand{\imm}{\looparrowright}

\DeclareMathOperator{\ad}{ad}
\DeclareMathOperator{\cd}{cd}
\DeclareMathOperator{\lcm}{lcm}
\DeclareMathOperator{\rank}{rank}

\DeclareMathOperator{\wt}{\mathrm{wt}}

\newcommand{\findex}{\leqslant_{\mathrm{fi}}}
\newcommand{\nfindex}{\trianglelefteqslant_{\mathrm {fi}}}

\usepackage{tikz}
\usetikzlibrary{arrows,quotes}
\tikzstyle{blackNode}=[fill=black, draw=black, shape=circle]
\tikzset{
math to/.tip={Glyph[glyph math command=rightarrow]},
math onto/.tip={Glyph[glyph math command=twoheadrightarrow]},
loop/.tip={Glyph[glyph math command=looparrowleft, swap]},
loop'/.tip={Glyph[glyph math command=looparrowleft]},
 weird/.tip={Glyph[glyph math command=Rrightarrow, glyph length=1.5ex]},
  pi/.tip={Glyph[glyph math command=pi, glyph length=1.5ex, glyph axis=0pt]},
}

\title[Free-by-cyclic groups are conjugacy separable]{Free-by-cyclic groups are\\ conjugacy separable}

\author{Fran\c{c}ois Dahmani}
\address[Fran\c{c}ois Dahmani]{Institut Fourier, Laboratoire de Mathématiques, Université Grenoble Alpes, CS 40700, 38058 Grenoble cedex 9, France}
\email{francois.dahmani@univ-grenoble-alpes.fr}

\author{Sam Hughes}
\address[Sam~Hughes]{Rheinische Friedrich-Wilhelms-Universit\"at Bonn, Mathematical Institute, Endenicher Allee 60, 53115 Bonn, Germany}
\email{sam.hughes.maths@gmail.com; hughes@math.uni-bonn.de}

\author{Monika Kudlinska}
\address[Monika Kudlinska]{DPMMS, Centre for Mathematical Sciences, Wilberforce Road, Cambridge,
CB3 0WB, UK, and Emmanuel College, St Andrew's Street, Cambridge CB2 3AP, UK}
\email{m.kudlinska@dpmms.cam.ac.uk}

\author{Nicholas Touikan}
\address[Nicholas Touikan]{Department of Mathematics \& Statistics, University of New Brunswick (Fredericton)}
\email{nicholas.touikan@unb.ca}

\date{\today}

\begin{document}
\begin{abstract}
We show that all finitely generated free-by-cyclic groups are conjugacy separable: if a finitely generated group $G$ surjects onto $\mathbb{Z}$  with free kernel, then for every pair of non-conjugate elements $g,h\in G$, there exists a finite quotient $\alpha:G\twoheadrightarrow Q$ such that $\alpha(g)$ is not conjugate to $\alpha(h)$. This resolves Question 19.41 of the Kourovka Notebook.

We apply this to prove that the outer automorphism group of a finitely generated free-by-cyclic group is residually finite.  Along the way we prove that if the monodromy of a \{finitely generated free\}-by-cyclic group is polynomially growing, then the double cosets of a cyclic subgroup are separable.  

Our approach combines vertex fillings in graph-of-groups decompositions, and Dehn fillings in relatively hyperbolic groups, according to the different geometric regimes  in free-by-cyclic groups.
\end{abstract}
\maketitle


\section{Introduction}

\subsection{Conjugacy separability}

A group $G$ is \emph{conjugacy separable} if the conjugacy class of every element is closed in the profinite topology on $G$. More explicitly, for any pair of non-conjugate elements $g,h \in G$, there exists a finite quotient of $G$ such that the image of $g$ is not conjugate to the image of $h$.

Separation properties through finite quotients remain a central topic in the study of finitely generated groups. Among these properties, conjugacy separability is substantially stronger than residual finiteness and provides significantly more algebraic control, though it is notoriously delicate to establish.

The historical motivation for studying conjugacy separability largely stems from its connection to the conjugacy problem. Mostowski \cite{Mostowski1966} established that any finitely presented, conjugacy separable group has a solvable conjugacy problem. He also noted that Blackburn \cite{Blackburn1965} had already demonstrated the conjugacy separability of finitely generated nilpotent groups. This result was subsequently extended to polycyclic-by-finite groups independently by Remeslennikov \cite{Remeslennikov1969} and Formanek \cite{Formanek1976}.

The scope of conjugacy separable groups was dramatically expanded by Stebe \cite{Stebe1970}, who established the property for free groups. This was later generalised to virtually free groups by Dyer \cite{Dyer1979} using Stallings' structure theorem.

Stebe \cite{Stebe1972} also proved that surface groups are conjugacy separable, paving the way for Grossman's \cite{Grossman1974} important application of conjugacy separability to the residual finiteness of mapping class groups and other outer automorphism groups. Techniques from low-dimensional topology, already implicit in Stebe's work, facilitated further extensions to all Fuchsian groups (Fine and Rosenberg \cite{FineRosenberg1990}), virtual surface groups (Martino \cite{Martino2007}), limit groups (Chagas and Zalesskii \cite{chagas_limit_2007}), and compact 3-manifold groups (Hamilton et al. \cite{HamiltonWiltonZalesskii}). In the realm of geometric group theory, Minasyan \cite{Minasyan2012} established conjugacy separability for finitely generated subgroups of right-angled Artin groups. Building on this, Minasyan and Zalesskii \cite{MinasyanZalesskii2016} utilised cohomological arguments to treat virtually compact special hyperbolic groups, while Ferov \cite{Ferov2016} proved that conjugacy separability is closed under graph products.

Several steps of this progression involve passing to a finite index subgroup, or finite extensions. We observe that this is a delicate part, only possible in specific cases. 

Indeed, there exists a non-conjugacy separable group $G$ with a conjugacy separable index two subgroup \cite{Goryaga1986}. More surprisingly, Martino--Minasyan construct non-conjugacy separable subgroups of finite index in a conjugacy separable group, which are finitely presented and have solvable conjugacy problem \cite{MartinoMinsyan2012}.  See \cite{Minasyan2017} for further pathologies. 

\subsection{Free-by-cyclic groups, and main results}

A group is \emph{free-by-cyclic} if it admits a homomorphism onto $\mathbb{Z}$ with free kernel.  The class of such groups is surprisingly large, including all finite rank free groups, compact surface groups, and fundamental groups of many 3-manifolds. 

Much recent progress has been made in finding \emph{virtual} free-by-cyclic structure in groups. For example, it is a result of Hagen and Wise \cite{hagen_special_2010} that special groups with elementary hierarchies, such as hyperbolic limit groups, are virtually free-by-cyclic; Kielak and Linton showed that all one-relator groups with torsion are virtually free-by-cyclic \cite{KielakLinton2024}, and this was extended by Fisher to all finitely generated RFRS groups of cohomological dimension 2 with vanishing second $L^2$-Betti number \cite{Fisher2026}. Moreover, it was shown in \cite{KielakKrophollerWilkes2022} that a random group of deficiency $d > 1$ with respect to the few-relator model is virtually free-by-cyclic with asymptotic probability 1 (and it is free-by-cyclic with positive asymptotic probability).

\smallskip

The question of whether free-by-cyclic groups are conjugacy separable has been recorded as an open problem, most notably in the Kourovka Notebook \cite[Question~19.41]{Kourovka2018} (see also \cite[Question~1.4]{HughesKudlinska2023}). It arises naturally in the context of the conjugacy problem for $\Out(F_n)$ \cite{DahmaniTouikan2025}, and is moreover important for the study of the outer automorphism groups of free-by-cyclic groups and the structural automorphisms of the free groups defining them.  

The purpose of this article is to give a complete answer to this question by establishing the following theorem:

\begin{thmx}\label{main}
Every finitely generated free-by-cyclic group is conjugacy separable. 
\end{thmx}

We emphasise that a finitely generated free-by-cyclic group need not have a finitely generated free kernel. Nevertheless, the specific case of $G = F \rtimes \Z$ with $F$ a finitely generated free group is of significant independent interest. Indeed, the existing questions in the literature primarily focus on this setting. This restricted case constitutes the bulk of our analysis; it is established in Corollary \ref{coro;CS_fgf_bc} and serves as a fundamental mechanism for proving the general case of \Cref{main} in \Cref{sec;final}.

\smallskip

Free-by-cyclic groups are ubiquitous in geometric group theory, and have attracted recent interest across different topics.  However, the study of their profinite topology has received little attention and has almost entirely focused on subgroup separability and rigidity \cite{Kudlinska2024,HughesKudlinska2023,BridsonPiwek2025, AndrewHillenLymanPfaff2025}.  

For a given free-by-cyclic group $G = F\rtimes \mathbb{Z}$, a lift of the generator of $G/F$ to $G$ acts by conjugation on $F$. Any two lifts induce automorphisms that differ by an inner automorphism. We call the corresponding outer class the \emph{monodromy} of the splitting. 

The class of free-by-cyclic groups exhibits a remarkably diverse range of algebraic and geometric phenomena, often governed by the dynamics of the defining monodromy map. Such monodromies, in turn, admit rich dynamical properties, well studied by several authors \cite{bestvina_tits_2000, BestvinaFeighnHandel2005, Levitt2009}.  

As such, it is often the case that a property is only enjoyed by a subset of free-by-cyclic groups which exhibit a specific type of dynamical behaviour in their monodromies, whilst failing spectacularly for the others. For example, by the recent work of Linton \cite{Linton2025}, if $G = F \rtimes_{\Phi} \mathbb{Z}$ is finitely generated, $F$ is non-finitely generated free, and $\Phi$ is \emph{fully irreducible} (see \cref{sec:auto-free-group-background}), then $G$ is subgroup separable; that is, every finitely generated subgroup of $G$ can be separated from a complementary element in a finite quotient of $G$. Conversely, if $F$ is finitely generated and $\Phi$ grows polynomially, then $G$ is subgroup separable only in the degenerated case, that is, when $\Phi$ has finite order in $\Out(F)$ \cite{Kudlinska2024}.

Accordingly, our argument takes several steps of different flavour. First we prove that free-by-cyclic groups with unipotent polynomially growing monodromies are conjugacy separable (see Theorem \ref{thm:unipotent-conj-sep}). Even in this first step, the case of linear unipotent growth contrasts with the higher degree unipotent growth. Then we  use a recent result of G. Bartlett, that every free-by-cyclic group with finite order monodromy is conjugacy separable \cite{Bartlett2025},   to prove that free-by-cyclic groups with polynomially growing monodromy (non-necessarily unipotent) are conjugacy separable (see \Cref{thm:poly_conj_sep}).   The final step uses a relatively hyperbolic Dehn filling argument and cubulation, in order to cover all cases with monodromies of exponential growth.      

Whilst the solvability of the conjugacy problem was already known in the case of finitely generated free kernel by Bogopolski et al. \cite{BMMV2006} 
and Linton in the general case \cite[Corollary~8.2]{Linton2025}, we remark that our \Cref{main} gives a  uniform effective solution to  the conjugacy problem  over the class of finitely generated  free-by-cyclic groups.

As an application of our results we obtain the following theorem (proved in \Cref{sec;final}). 

\begin{thmx} \label{cor:rf-outer-automorphism}
    Let $G$ be finitely generated free-by-cyclic. Then $\Out(G)$ is residually finite. 
\end{thmx}

\smallskip

\subsection{Proof strategy and further results}
Here we outline the proof of \Cref{main}.

An outer automorphism $\Phi \in \Out(F)$ of a finite rank free group $F$ is \emph{polynomially growing} if  for every element of $F$, the lengths of cyclic reductions of its images by the iterated monodromy of $\Phi$ are bounded above by a polynomial (see \cref{sec:auto-free-group-background} for the precise definition).  We call a free-by-cyclic group \emph{polynomially growing} if it has a polynomially growing monodromy.  An outer automorphism $\Phi \in \Out(F)$ of a finite rank free group $F$ is \emph{unipotent} if it induces a unipotent element of $\GL(H_1(F;\Z))$.  Every polynomially growing element has a unipotent power.

Most of the work in this paper is in understanding the unipotent polynomially growing case. We will now explain how to reduce to this case. 

By \cite{Ghosh2023, DahmaniLi2022}, every \{finitely generated free\}-by-cyclic group either has polynomial-growth monodromy or is hyperbolic relative to a finite collection of polynomially growing free-by-cyclic subgroups.  We use a relatively hyperbolic group Dehn filling argument (see \cite{Osin2007,GrovesManning2008}) to lift the conjugacy separability from the polynomially growing parabolics.  A key step in our proof is that the Dehn filling quotients we construct are hyperbolic hyperbolic-by-cyclic groups and hence virtually compact special \cite{DahmaniKrishnaMutanguha2025} and conjugacy separable by \cite{MinasyanZalesskii2016}.

In the case of free-by-cyclic groups that have a non-finitely-generated free kernel, we instead use Linton's structure of relative hyperbolicity with respect to maximal \{finitely generated free\}-by-cyclic subgroups \cite{Linton2025}, and a controlled embedding in a larger \{finitely generated free\}-by-cyclic group, from \cite{Linton_2026}, to reduce to the case of finitely generated kernel.

\smallskip

We now turn our attention to the polynomially growing case.  It is by now a well-known fact that free-by-cyclic groups with unipotent and polynomially growing monodromies admit acylindrical graphs-of-groups splittings over abelian subgroups \cite{Macura2002,BestvinaFeighnHandel2005,Hagen2019,AndrewHughesKudlinska2024,AndrewMartino22, dahmani_unipotent_2024, KudlinskaValiunas2025}. Furthermore the vertex groups of these splittings are themselves free-by-cyclic groups with unipotent monodromies whose polynomial growth is of strictly lower degree, thus giving rise to a hierarchical decomposition. Our aim is to exploit such splittings and induction on the degree of polynomial growth to construct finite quotients that separate conjugacy classes.

In the base case of linearly growing monodromy, by \cite[Proposition~5.2.2]{AndrewMartino22} the corresponding free-by-cyclic group $G$ splits as a graph of groups with vertex groups of the form $F_v \times \Z$ where $F_v$ is a finite rank free group, and $\Z^2$-edge groups. Such groups are called piecewise trivial suspensions (see \cref{sec:PTS}), and it is interesting to note that fundamental groups of graph manifolds are virtually piecewise trivial suspensions. Wilton and Zalesskii's combination theorem (see \cref{thm:combination-Wilon--Zalesskii}) was formulated to show that the fundamental groups of graph manifolds are conjugacy separable. Since the argument requires (profinite) 2-acylindricity, it is not applicable to free-by-cyclic groups with linearly growing monodromy: in general the action on the Bass--Serre tree is 4-acylindrical. 

We therefore resort to a more direct approach. We construct a virtually free quotient of $G$ for a given pair of non-conjugate elements so that the images remain non-conjugate. Since virtually free groups are conjugacy separable, we may pass to a further finite quotient where the images are non-conjugate. 

The virtually free quotients arise through a \emph{vertex filling} procedure, which works by replacing each vertex group in the splitting of $G$ by an appropriate finite quotient, resulting in a graph of \emph{finite} groups with the same underlying graph as that of the original splitting of $G$, and such that $G$ maps onto the fundamental group of the new graph of groups.

Standard arguments reduce to the case of non-conjugate elements of $G$ which both act loxodromically on the Bass--Serre tree of the splittings with the same translation length and same sequence of double cosets representatives in their \emph{short-position representatives} (see \cref{sec:short-position} for the definition). In this case, we must work harder to construct the required virtually free fillings.

 A key tool in constructing virtually free fillings in this case is the property of \emph{strong command} for independent elements of vertex groups. Two infinite-order elements $g_1, g_2 \in G$ are \emph{independent} if the conjugacy class of the cyclic subgroup $\langle g_1 \rangle$ intersects the conjugacy class of $\langle g_2 \rangle$ in exactly the trivial subgroup. The group $G$ \emph{strongly commands} independent elements $g_1$ and $g_2$ if there is a finite quotient of $G$ such that the images of $g_1$ and $g_2$ have prescribed orders and their cyclic subgroups have trivial intersection. Bridson--Wilton showed that a (virtually) free group strongly commands any tuple of independent elements \cite[Theorem 4.3]{BridsonWilton2015}, and thus the same holds true for groups of the form $\mathbb{F}_v \times \Z$. Being able to ensure that independent elements have images in finite groups whose orders have prescribed divisors and that generate trivially intersecting subgroups is crucial to the combinatorial arguments in the linear growth case. It is also interesting to note that while strong command is typically used to construct finite index subgroups of graphs of groups, which we do, we will  also use this property to construct vertex fillings.

 \smallskip

In the case of superlinear polynomial growth, the corresponding free-by-cyclic group $G$ splits as a graph of groups with infinite cyclic edge groups and vertex groups which are free-by-cyclic with unipotent and polynomially growing monodromy of strictly lower degree growth. The action on the Bass--Serre tree is 2-acylindrical. Here, we may apply the Combination Theorem of Wilton--Zalesskii \cite{WiltonZalesskii2010} (see Theorem \ref{thm:combination-Wilon--Zalesskii}) and argue by induction.

One of the hypotheses in the theorem of Wilton--Zalesskii is that edge groups are double coset separable in the vertex groups. Thus, we show the following:

\begin{thmx}\label{main2}\label{prop:double-coset-unipotent}
    Let $G$ be a \{finitely generated free\}-by-cyclic group with polynomially growing monodromy. Then, for any cyclic subgroups $H,K \leq G$ and $g \in G$, the double coset $HgK \subseteq G$ is separable.
\end{thmx}

As before, we prove \cref{main2} by constructing virtually free quotients via vertex fillings. This time, we need good control over the orders and the pairwise intersections of cyclic subgroups in the image of the vertex groups. In the case that the vertex groups split as products $\mathbb{F}_v \times \Z$, we can use strong command of the fibre $F_v$ as in the previous arguments.  

We are not able to prove command for all free-by-cyclic groups with unipotent and polynomially growing monodromies, however we show the following, which might be of independent interest:

\begin{propx}\label{prop}
    Let $G$ be a \{finitely generated free\}-by-cyclic group with unipotent and polynomially growing monodromy. Let $S\subseteq G$ be a collection of non-trivial elements and let $g,h \in G$ be such that $\langle g \rangle \cap \langle h \rangle = 1$ or $[g,h] \neq 1$. There exists $N = N(S)$ such that for every prime $p>N$ 
    there is a $p$-periodic quotient of $G$ such that the image of each element of $S$ is non-trivial and the cyclic subgroups generated by $g$ and $h$ have trivial intersection. 
\end{propx}

The key tool for proving \cref{prop} is the observation that any \{finitely generated free\}-by-cyclic group with unipotent and polynomially growing monodromy is residually torsion-free nilpotent, and thus residually $p$-finite for every prime $p$. We note that it is known that free-by-cyclic groups are \emph{virtually} residually $p$-finite for every prime $p$ by \cite[Corollary~4.32]{AschenbrennerFriedl2013}.

\cref{prop} differs from strong command in that we are no longer able to prescribe arbitrary divisors to the orders of the images of element in our independent set. This proposition is therefore not suitable for the linear growth case, but it is good enough for the inductive step. This fits into a typical pattern, starting with \cite{Macura2002}, where the linear growth and the superlinear growth cases requires substantially different techniques, the latter case being  less pathological.

\subsection{Structure of the paper}
In \Cref{sec:prelims} we give the necessary background on group actions on trees, graphs of groups, profinite topologies on groups, automorphisms of free groups, and free-by-cyclic groups.  

In \Cref{sec:resTFN} we prove that unipotent polynomially growing free-by-cyclic groups are residually torsion-free nilpotent (\Cref{thm:unipotent_res_nilp}) and deduce \cref{prop}. 

In \Cref{sec:virFreeFill} we construct virtually free vertex fillings for cyclic splittings of free-by-cyclic groups. We also prove general results about conjugacy distinguished elements and cyclic subgroups in graphs of groups that admit virtually free vertex fillings.

In \Cref{sec:PTS} we specialise our study of vertex fillings to the case of unipotent linear monodromy.  This section uses many results from \cite{dahmani_unipotent_2024}.   We conclude this section by proving \Cref{main} in the special case of a free-by-cyclic group with unipotent and linearly growing monodromy (\Cref{thm:conjugacy-separable-linear-UPG}).  

In \Cref{sec:dbleZsep} we prove \Cref{main2}. 

In \Cref{sec:ConjSep} we prove that polynomially growing free-by-cyclic groups are conjugacy separable (\Cref{thm:poly_conj_sep}) by combining Theorems~\ref{thm:conjugacy-separable-linear-UPG} and \ref{main2} with work of Wilton--Zalesskii \cite{WiltonZalesskii2010} and Chagas--Zalesskii \cite{ChagasZalesskii2010}.  

Finally, in \Cref{sec:general-case}, we prove \Cref{main} using \Cref{thm:poly_conj_sep} and relatively hyperbolic Dehn filling \cite{Osin2007,GrovesManning2008}, cubulability of hyperbolic hyperbolic-by-cyclic groups \cite{DahmaniKrishnaMutanguha2025}, and conjugacy separability of hyperbolic virtual special groups \cite{MinasyanZalesskii2016}. The relatively hyperbolic structures used are from \cite{Ghosh2023, DahmaniLi2022}, and \cite{Linton2025}. We finish by deducing \Cref{cor:rf-outer-automorphism} from the work of Grossman \cite{Grossman1974}.   

\tableofcontents

\section*{Acknowledgments}
FD is supported by ANR GoFR - ANR-22-CE40-0004. SH was supported by a Humboldt Research Fellowship at Universit\"at Bonn and by the Deutsche Forschungsgemeinschaft (DFG, German Research Foundation) under Germany's Excellence Strategy - EXC-2047/1 - 390685813.  SH would like to thank the Isaac Newton Institute for Mathematical Sciences, Cambridge, for support and hospitality during the programme Operators, Graphs, and Groups where work on this paper was undertaken. This work was supported by EPSRC grant no EP/Z000580/1.  NT is supported by an NSERC Discovery Grant.  MK thanks Motiejus Valiunas for an insightful conversation about torsion-free nilpotent groups. The authors thank Naomi Andrew for many helpful discussions about free-by-cyclic groups, Yassine Guerch for pointing out \cite[Theorem~1.3]{LevittMinasyan2014}, and Ashot Minasyan for many helpful conversations and suggesting the use of \cite[Theorem~2.4]{ChagasZalesskii2010}. The authors would like to thank the referee of a previous version, who indirectly stimulated a significant improvement of the main result.

\section{Preliminaries}\label{sec:prelims}

We begin by establishing conventions that will be used throughout the paper. For $g,h \in G$, we write $\ad_g(h)= g^h=h^{-1}gh$.

\subsection{Actions on trees}

Let $T$ be a simplicial tree and define a metric $d_T$ on $T$ such that each edge is isometrically identified with the unit interval and the distance between two points in $T$ is the length of the shortest path between them. Let $\Isom(T)$ be the group of isometries of $(T, d_T)$.

For any element $g \in \mathrm{Isom}(T)$, the \emph{translation length} of $g$ is
\[ \ell_T(g) := \inf\{d_T(x, g\cdot x) \mid x \in X\}.\]
It is classical that if $\ell_T(g) = 0$ then $g$ fixes a point in $T$, and otherwise there exists a unique line in $T$ called the \emph{axis} of $g$ on which $g$ acts as translation by $\ell_T(g)$. If $g \in \mathrm{Isom}(T)$ fixes a point then it is called \emph{elliptic} and otherwise it is \emph{hyperbolic.}

The action of $G \leqslant \mathrm{Isom}(T)$ on $T$ is said to be \emph{$\kappa$-acylindrical} if the pointwise stabiliser of any edge path of length $\kappa+1$ is trivial. We say that the action of $G$ on $T$ is \emph{acylindrical} if it is $\kappa$-acylindrical for some non-negative integer $\kappa$.


\subsection{Graphs of groups, Bass--Serre theory}

A \emph{(combinatorial) graph} $X$ consists of a tuple of sets $(V(X), E(X))$ where $V(X)$ is called the \emph{vertex set}, and $E(X)$ the \emph{edge set}, together with a pair of maps $\iota \colon E(X) \to V(X)$ and $\tau \colon E(X) \to V(X)$, and a fixed-point-free involution $\bar{\ }\  \colon E(X) \to E(X)$, such that for every edge $e \in E(X)$ we have that $\iota(e) = \tau(\bar{e}).$

A \emph{graph of groups} $\calg$ is a triple $(X, \mathcal{G}_{\bullet}, \iota_{\bullet})$ where $X$ is a graph, $\mathcal{\mathcal{G}}_{\bullet}$ encodes the assignment of a group $\mathcal{G}_v$ to every vertex $v \in V(X)$ and a group $\mathcal{G}_e$ to every edge $ e\in E(X)$ so that $\mathcal{G}_e = \mathcal{G}_{\bar{e}}$, and $\iota_{\bullet}$ determines monomorphisms $\iota_e \colon \mathcal{G}_e \hookrightarrow \mathcal{G}_{\iota(e)}$ for all edges $e \in E(X)$. We also define the monomorphism $\tau_e \colon \calg_e \to \calg_{\tau(e)}$, to be $\tau_e = \iota_{\bar{e}}$ for each $e \in E(X)$. Throughout the paper we will assume that the underlying graph $X$ is finite.

Let $\mathcal G$ be a graph of groups and denoting by $F_{E(X)}$ the free group with basis $E(X)$. The \emph{Bass group} $\bass(\calg)$ is defined to be
\[ \bass(\calg)  = \left.  \left(\left(\bigast_{v \in V(X)} \calg_v \right) * F_{E(X)}  \right) \middle/ \left.\llangle e\bar{e} = 1,\, \bar{e}\iota_e(g)e = \tau_e(g) \rrangle  \right.\right..\] 

For a vertex $v \in V(X)$, the \emph{fundamental group of $\calg$ based at v}, denoted by $\pi_1(\calg,v)$, is the subgroup of $\bass(\calg)$ given by the elements of the form
\[
a_0e_1a_1e_2\cdots e_n a_n
\]
where $e_i \in E(X)$, $a_i \in \calg_{\tau(e_i)} = \calg_{\iota(e_{i+1})}$  (when the indices occur) and the edge path $e_1\cdots e_n$ is a closed loop based at $v$ in the graph $X$. 

For any vertex $v\in V(X)$, the fundamental group $\pi_1(\calg,v)$ sits as a free factor inside $\bass(\calg)$ and for any two vertices $v,w \in V(X)$, the subgroups $\pi_1(\calg,w)$ and $\pi_1(\calg,v)$ are conjugate in $\bass(\calg)$. We call an element $g \in \bass(\calg)$ a \emph{$\calg$-loop} if it is an element of $\pi_1(\calg,v)$ for some $v\in V(X)$. We will sometimes omit the basepoint and simply write $\pi_1(\calg)$ when there is no risk of confusion.

For any two graphs of groups $\calg = (X_{\calg}, \calg_{\bullet}, \iota_{\bullet})$ and $\calh = (X_{\calh}, \calh_{\bullet}, \kappa_{\bullet})$, a \emph{morphism} $f \colon \calg \to \calh$ consists of the tuple \[(f_X, \{f_v\}_{v \in V(X)}, \{f_e\}_{e\in E(X)}, \{\gamma_e\}_{e \in E(X)})\] where $f \colon X_{\calg} \to X_{\calh}$ is a morphism of graphs, each $f_v \colon \calg_v \to \calh_{f(v)}$ for $v \in V(X_{\calg})$ and $f_e \colon \calg_e \to \calh_{f(e)}$ for $e \in E(X_{\calg})$ is a homomorphism, and $f_e = f_{\bar{e}}$, and $\gamma_e \in \calg_{f(e)}$ for every $e \in E(X_{\calg})$. We also require that for every $e \in E(X_{\calg})$, \[f_{\iota(e)}\circ \iota_e = \ad_{\gamma_e} \circ \iota_{f(e)} \circ f_e.\]

A morphism of graphs of groups $f \colon \calg \to \calh$ induces a homomorphism of the corresponding fundamental groups $f_{\ast} \colon \pi_1(\calg, v) \to \pi_1(\calh, f(v))$ for any $v \in V(X_{\calg})$ \cite[Proposition~2.4]{Bass1993}.

Let $X_\calg$ be a graph and fix $v_0\in V(X_\calg)$.  The universal cover of a graph of groups $\calg=(X_\calg,\calg_\bullet,\iota_\bullet)$ is the graph $T=T_\calg$ with vertices 
\[V(T)=\coprod_{v\in V(X_\calg)} \pi_1(\calg,v_0)/\calg_v,\]
 edges
\[E(T)=\coprod_{e\in E(X_\calg)}\pi_1(\calg,v_0)/\iota_e(\calg_e),\]
and the adjacency map $\iota$ given by inclusions of cosets.  The graph $T$ comes equipped with an action of $\pi_1(\calg,v_0)$ with cell stabilisers conjugates of the groups in $\calg_\bullet$.

\subsection{Profinite topology}

\begin{defn}
Let $G$ be a discrete group and let $\{N_i\}_{i\in J}$ be the family of finite index normal subgroups of $G$. The \emph{profinite completion $\widehat{G}$ of $G$} is the inverse limit of the inverse system $(G/N_i)_{i \in J}$, \[ \widehat{G} := \varprojlim_{i\in J} G/N_i.\]
There is a natural map $\iota \colon G \to \widehat{G}$ that sends each element $g$ to the tuple $(gN_i)_{i \in J}$. The \emph{profinite topology} is the coarsest topology on $G$ so that the map $\iota \colon G \to \widehat{G}$ is continuous. Equivalently, it is the topology generated by a basis of open subsets consisting of cosets of finite index normal subgroups of $G$. 
\end{defn}

A subset $X \subseteq G$ is \emph{separable} if it is closed in the profinite topology. Equivalently, for any element $g \in G \setminus X$, there exists a finite quotient $\pi \colon G \to Q$ such that $\pi(g) \not\in \pi(X)$.

A subgroup $H \leq G$ is \emph{fully separable}, if every finite index subgroup $H' \leq _f H$ is separable in $G$. We say that $G$ \emph{induces the full profinite topology on $H$} if the closure of $H$ in the profinite completion of $G$ is isomorphic to $\widehat{H}$.

We will often use the following lemma:

\begin{lemma}[{Reid \cite[Lemma~4.6]{Reid2015}}]\label{lemma:Reid}
    Let $G$ be a finitely generated group and $H \leq G$ a finitely generated subgroup. Then $H$ is fully separable in $G$ if and only if $G$ induces the full profinite topology on $H$.
\end{lemma}

A group $G$ is \emph{conjugacy separable} if the conjugacy class of any element is separable in $G$. We say that an element $g \in G$ is \emph{conjugacy distinguished} if the conjugacy class of $g$ is separable.

\begin{thm}[{Stebe \cite{Stebe1970}, Dyer \cite{Dyer1979}}]\label{thm:conjugacy-v-free}
    Let $G$ be a finitely generated virtually free group. Then $G$ is conjugacy separable.
\end{thm}

A group $G$ is said to have the \emph{unique roots property} if for any two elements $a,b \in G$ such that $a^n = b^n$ for some positive integer $n$, it follows that $a = b$. A subgroup $H \leq G$ of a torsion-free group is \emph{root-closed} if for any $g \in G$ such that $g^n \in H$ for some positive integer $n$, it follows that $g \in H$.

\begin{lemma}[{Cotton-Barratt--Wilton \cite[Lemma~3.1]{cotton-barratt_conjugacy_2012}}]\label{lemma:Cotton-Barrett--Wilton-lemma}
    Let $G$ be a finitely generated group with the unique roots property. If $G$ contains a conjugacy separable finite index subgroup then $G$ is conjugacy separable.
\end{lemma}

\begin{lemma}\label{lemma:unique-roots-property}
    Let $\calg = (X, \calg_{\bullet}, \iota_{\bullet})$ be an acylindrical graph of groups such that the vertex groups have the unique roots property and the edge groups are root-closed. Then, $\pi_1(\calg)$ has the unique roots property. 
\end{lemma}

\begin{proof}
    Let $G = \pi_1(\calg)$ and let $T$ be the Bass--Serre tree corresponding to the splitting. Note that each vertex group has the unique roots property and thus is torsion free. Hence, $G$ is torsion free.
    
    Arguing as in \cite[Lemma~3.3]{cotton-barratt_conjugacy_2012}, if $a^n = b^n$ for some positive integer $n$, then $a$ and $b$ are both hyperbolic or elliptic. If $a$ and $b$ are both hyperbolic then it must be the case that $a$ and $b$ have the same axis and $\ell_T(a) = \ell_T(b)$. Thus, the element $ab^{-1}$ fixes an infinite line in $T$. Since the action of $G$ on $T$ is acylindrical, it follows that $ab^{-1} = e_G$ and $a = b$.

       Suppose that both $a$ and $b$ are elliptic. If $\Fix(a) \cap \Fix(b) \neq \emptyset$, then $a,b \in \calg_v$ for some vertex $v$ and $a^n = b^n$ in $\calg_v$. Hence, by the unique roots property of $\calg_v,$ it follows that $a = b$.

       Suppose now that $\Fix(a) \cap \Fix(b) = \emptyset$. Let $\gamma$ be the shortest path in $T$ joining $\Fix(a)$ to $\Fix(b)$. Then $a^n=b^n$ must fix the path $\gamma$ pointwise. Thus, $a^n$ and $b^n$ fix every edge in $\gamma$. Using root closure of edge groups, it follows that $a$ and $b$ fix every edge in $\gamma$. Hence, $a$ and $b$ fix a common vertex and we may argue as before.
\end{proof}

The group $G$ is said to be \emph{double coset separable}, if for any finitely generated subgroups $H,K \leq G$ and $g \in G$, the double coset $HgK \leq G$ is separable in $G$. 

We will use the following lemma often throughout the text.

\begin{lemma}[Niblo {\cite[Proposition~2.2]{niblo_separability_1992}}]\label{lemma:double-coset-separability-finite index}
    Let $G$ be a group and $H, K \leq G$ subgroups. Let $G' \leq_f G$ be a finite index subgroup and let $H'=H\cap G'$ and $K'=K\cap G'$. Then the double coset $HK$ is separable in $G$ if and only if $H'K'$ is separable in $G'$.
\end{lemma}

\begin{thm}[{Gitik--Rips \cite{GitikRips95}}]\label{thm:free-double-coset}
If $G$ is a finitely generated free group, then $G$ is double coset separable.
\end{thm}

Combining the previous two results we obtain that double coset separability also holds for virtually free groups. 

\begin{thm}[{Minasyan \cite[Theorem~1.1]{minasyan_double_2023}}]\label{thm-Minasyan}
    Let $G$ be a residually finite group and let $H, K \leq G$ be subgroups. Suppose that for every finite index subgroup $H' \leq_f H$, we have that $H' K$ is separable in $G$. Then the intersection $\overline{H} \cap \overline{K}$ of the closures of $H$ and $K$ in the profinite completion of $G$ is equal to $\overline{H \cap K}$.
\end{thm}

A finitely generated subgroup $H \leq G$ is \emph{conjugacy distinguished}, if for any $g \in G$ which is not conjugate into $H$, there exists a finite quotient $\pi \colon G \to Q$ such that $\pi(g)$ is not conjugate into $\pi(H)$.

\begin{thm}[{Ribes--Zalesskii \cite[Theorem A]{RibesZalesskii2016}}]\label{virtually free-conj-dist}
Every finitely generated subgroup of a finitely generated virtually free group is conjugacy distinguished. 
\end{thm}

\begin{defn}[{Wilton--Zalesskii \cite{WiltonZalesskii2010}}]
    Let $\calg = (X, \calg_{\bullet}, \iota_{\bullet})$ be a graph of groups. We say that the profinite topology on $\pi_1(\calg)$ is \emph{efficient} if $\pi_1(\calg)$ is residually finite, every vertex and edge group is closed in the profinite topology on $\pi_1(\calg)$ and $\pi_1(\calg)$ induces the full profinite topology on each vertex and edge group.
\end{defn}

Let $\calg = (X, \calg_{\bullet}, \iota_{\bullet})$ be a graph of groups and suppose that the profinite topology on $G= \pi_1(\calg)$ is efficient. We write $\widehat{\calg}$ to denote the graph of groups with underlying graph $X$, such that the vertex and edge groups are profinite completions of the corresponding groups in $\calg$, and the edge inclusions $\widehat{\calg}_e \hookrightarrow \widehat{\calg}_{i(e)}$ are the natural maps induced by the edge inclusions in $\calg$. Since the topology on $G$ is efficient, the profinite completion $\widehat{G}$ of $G$ is isomorphic to the profinite fundamental group of $\widehat{\calg}$ . Moreover, there is a simply-connected profinite graph, which we denote by $S(\widehat{G})$ and call the \emph{profinite Bass--Serre tree}, which admits an action of $\widehat{G}$. The Bass--Serre tree corresponding to the splitting $G = \pi_1(\calg)$, which we denote here by $S(G)$, embeds as a dense subset of $S(\widehat{G})$. See \cite{ZalesskiiMelnikov2988,Ribes2017} for further details about actions of profinite groups on profinite trees.

\begin{defn}
    Let $\calg$ be a graph of groups such that the profinite topology on $\pi_1(\calg)$ is efficient. Let $\widehat{\calg}$ be the corresponding graph of profinite groups and $S(\widehat{G})$ the profinite Bass--Serre tree. We say that $\widehat{\calg}$ is \emph{profinitely $\kappa$-acylindrical} if every edge path of length $\kappa +1$ in $S(\widehat{G})$ has trivial pointwise stabiliser for the action of $\widehat{G}$.
\end{defn}

\begin{lemma}\label{lemma:profinitely-acylindrical-splitting}
    Let $\calg$ be a 2-acylindrical graph of groups with efficient profinite topology. Suppose that for every vertex $v \in V(X)$ and incident edges $e$ and $f$, the following conditions are satisfied:
    \begin{enumerate}
         \item the intersection $\calg_{e} \cap \calg_{f}$ is either trivial or $\calg_{e} = \calg_{f}$, and
        \item the intersection of the closures of ${\calg}_{e}$ and ${\calg}_{f}$ in the profinite completion of $\calg_v$ is given by $\widebar{\calg}_{e} \cap \widebar{\calg}_{f} = \widehat{\calg_{e} \cap \calg_{f}}$. 
    \end{enumerate}
    Then, $\widehat{\calg}$ is profinitely 2-acylindrical.
\end{lemma}

\begin{proof}
    Our argument is modelled on the proof of \cite[Lemma 5.5]{WiltonZalesskii2010}.
    
    Let $\rho$ be a path of length 3 in $S(\widehat{G})$ that consists of the concatenation of the edges $e_1, e_2$ and $e_3$. We will show that the pointwise stabiliser of $\rho$ in $\widehat{G}$ is trivial. After translating $e_2$ by an element of $\widehat{G}$, we may assume that $e_2$ is an element of $S(G)$. Let $u$ be the initial point of $e_2$ and the final point of $e_1$. Since $e_2 \in S(G)$, we must have that $u$ is also contained in $S(G)$. By assumption (1) the intersection of $\calg_{e_1}$ and $\calg_{e_2}$ is either trivial or $\calg_{e_1} = \calg_{e_2}$, so it follows from assumption (2) that $\widebar{\calg}_{e_1} \cap \widebar{\calg}_{e_2}$ is either trivial or is equal to $\widehat{\calg}_{e_1} = \widehat{\calg}_{e_2}$. Since $\calg_{u}$ induces the full profinite topology on $\calg_{e_1}$ and $\calg_{e_2}$, we have that $\widehat{\calg}_{e_1} \cap \widehat{\calg}_{e_2}$ is either trivial, in the case that $\calg_{e_1} \cap \calg_{e_2} = 1$, or is equal to $\widehat{\calg}_{e_1} = \widehat{\calg}_{e_2}$, otherwise. Similarly, we have that  $\widehat{\calg}_{e_2} \cap \widehat{\calg}_{e_3}$ is trivial if $\calg_{e_2} \cap \calg_{e_4} = 1$ or is equal to $\widehat{\calg}_{e_2} = \widehat{\calg}_{e_3}$.  
    
    If $\calg_{e_1} \cap \calg_{e_2}$ is non-trivial, then we must have that $\calg_{e_2}\cap \calg_{e_3} = 1$ by 2-acylindricity of $\calg$ and by the first assumption. Then, $\widehat{\calg}_{e_2} \cap \widehat{\calg}_{e_3}$ is trivial by the argument above and thus $\widehat{\calg}_{e_1} \cap \widehat{\calg}_{e_2} \cap \widehat{\calg}_{e_3} = 1$.  On the other hand, if $\calg_{e_1} \cap \calg_{e_2}$ is trivial then $\widehat{\calg}_{e_1} \cap \widehat{\calg}_{e_2}$ is trivial and thus $\widehat{\calg}_{e_1} \cap \widehat{\calg}_{e_2} \cap \widehat{\calg}_{e_3} = 1$.
\end{proof}

\subsection{Automorphisms of finitely generated free groups} \label{sec:auto-free-group-background}

Let $F$ be a free group of finite rank and $X$ a free basis of $F$. An outer automorphism $\Phi \in \Out(F)$ acts on the set of conjugacy classes of elements in $F$. Given a conjugacy class $\bar{g}$ of an element $g\in F$, we write $| \bar{g} |_X$ to denote the word length of the shortest representative of the conjugacy class $\bar{g}$. We say that the conjugacy class \emph{$\bar{g}$ grows polynomially of degree $d$ under the iteration of $\Phi$,} if there exist constants $A,B > 0$ such that for all $k \geq 1$ \[Ak^d - A\leq |\Phi^k(\bar{g})|_X \leq Bk^d .\]
We say that the conjugacy class \emph{$\bar{g}$ grows exponentially under the iteration of $\Phi$,} if there exists a constant $A> 0$, $C > 1$ such that for all $k \geq 1$ \[ |\Phi^k(\bar{g})|_X \geq A\cdot C^k  .\]

For any two finite free generating sets $S$ and $S'$ of $F$, the word metrics with respect to $S$ and $S'$ are bi-Lipschitz equivalent. It follows that the growth of a conjugacy class under $\Phi$ does not depend on the specific choice of free basis for $F$.

We say the outer automorphism $\Phi \in \Out(F)$ \emph{grows polynomially of degree $d$} if every conjugacy class in $F$ grows polynomially of degree $\leq d$ and there exists an element $g\in F$ whose conjugacy class grows polynomially of degree exactly $d$. The outer automorphism $\Phi \in \Out(F)$ \emph{grows exponentially} if there exists a conjugacy class in $F$ that grows exponentially. 

We say $\Phi$ is \emph{fully irreducible} if there does not exist a non-trivial proper free factor of $F$ whose conjugacy class if preserved by a non-trivial power of $\Phi$. Note that this definition makes sense even when $F$ has infinite rank.

An outer automorphism $\Phi \in \Out(F)$ where $F$ is free of finite rank $n$ is \emph{unipotent} if $\Phi$ induces a unipotent element of $\Out(F_{\mathrm{ab}}) \cong \GL_n(\Z)$.  We will typically abbreviate \emph{unipotent and polynomially growing} to \emph{UPG}. An outer automorphism $\Phi$ is \emph{neat} if for every $x \in F$ and every representative automorphism $\phi$ of $\Phi$, if $\phi^k(x) = x$ for some $k \in \Z \setminus 0$, then $\phi(x) = x$.

\begin{lemma}
    Let $F$ be finite rank free. If $\Phi \in \Out(F)$ is UPG then it is neat. 
\end{lemma}

\begin{proof}
    Apply the argument from the proof of \cite[Theorem~4.4]{BestvinaFujiwaraWigglesworth2023} to the topmost splitting coming from an improved relative train track representative of a UPG element $\Phi \in \Out(F)$ \cite[Theorem~5.1.8]{bestvina_tits_2000}.
\end{proof}

\subsection{Free-by-cyclic groups}

A group $G$ is \emph{free-by-cyclic} if there exists an epimorphism $\chi \colon G \to \Z$ with $\ker \chi = F$, where $F$ is (not necessarily finitely generated) free. In this case, $G$ splits as a semidirect product, $G \cong \FF \rtimes_{\phi} \Z$ for $\FF \cong F$ and $\phi \in \Aut(\FF)$. We will often abuse notation and write $G = F \rtimes_{\phi} \langle t \rangle$ where $t \in \chi^{-1}(1)$. Note that for any two elements of $\chi^{-1}(1)$, the corresponding automorphisms of $F$ induced by the conjugation action represent the same element of $\Out(F)$.  We call $F \trianglelefteq G$ the \emph{fibre} and $\Phi := [\phi] \in \Out(F)$ the \emph{monodromy} corresponding to the pair $(G, \chi)$. We will also often suppress $\chi$ from the notation.

\begin{lemma}\label{lemma:unique-roots-UPG}
  Let $F$ be finite rank free and $\phi \in \Aut(F)$ be a representative of a UPG outer automorphism $\Phi$. Then $G = F \rtimes_{\phi} \langle t \rangle$ has the unique roots property. 
\end{lemma}

The proof is essentially \cite[Lemma 3.3]{KudlinskaValiunas2025}. 

\begin{proof}
    Let $n > 0$ and let $g,h \in G$ be such that $g^n = h^n$.  Then there is some $k \in \Z$ such that $g = ut^k$ and $h = vt^k$ for $u,v \in F$. Let $\psi \in \Aut(F)$ be the automorphism induced by the conjugation action by $g$. Let $w = vu^{-1} \in F$. Then $h = wg$ and the equation $g^n = h^n$ evaluates to 
    \[ g^n = g^n  \psi^{n-1}(w) \dots \psi^2(w) \psi(w) w.\] 
    Hence, $\psi^{n-1}(w) \dots \psi^2(w) \psi(w) w = 1$ and thus $\psi^n(w) = w$.
    Note that $\psi$ is in the outer automorphism class of $\Phi^k$, and since $\Phi$ is UPG it follows that $\Phi^k$ is also UPG. Hence $\Phi^k$ is neat and so we must have that $\psi(w) = w$ and thus $w^n = 1$. Hence by the unique roots property of the free groups we have that $w = 1$ and so $g = h$. \end{proof}

    \begin{lemma}\label{lemma:UPG-root-closed-subgroups}
       Let $F$ be finite rank free and let $\phi \in \Aut(F)$ be a representative of a UPG outer automorphism. Let $G = F \rtimes_{\phi} \langle t \rangle$ and let $H \leq G$ be a subgroup of the form $H = \langle vt \rangle$ or $H = K \oplus \langle vt \rangle$ for some $v \in F$, where $\Z \cong K \leq F$ is not generated by a proper power. Then $H$ is root-closed. 
    \end{lemma}

\begin{proof}
Suppose that there exists some $x \in G \setminus 1$ such that $x^m \in H$ for some $m \in \NN$.

    We begin by considering the case where $H$ is a cyclic subgroup. By replacing the automorphism $\phi$ with $\phi \circ \mathrm{ad}_v$, we may assume that $H = \langle t \rangle$. Let $x = ut^k$ for some $u \in F$ and $k \in \Z$. Since $x$ is non-trivial and $H$ is not a subgroup of $F$, it must be the case that $k \neq 0$ and 
    \[ (ut^k)^m = t^{km}.\]
   Then, by \cref{lemma:unique-roots-UPG} it follows that $ut^k = t^k$ and thus $u = 1$. Hence $x \in H$. 

    Suppose now that $H = K \oplus \langle t \rangle$ where $\Z \cong K \leq F$ is not generated by a proper power. Let us first assume that $x \in F$. Then $x^m \in H \cap F = K$. Let $k \in K$ be a generator of $K$. Then $x^m$ is a power of $k$. It follows that $x \in C_F(k) \cong \Z$ and thus $k$ and $x$ are powers of a common element $z \in F$. However, since $k$ is not a proper power, it must be the case that $k = z^{\pm 1}$ and thus $x \in K \leq H$. 

    Suppose now that $x \not \in F$. Then $x = ut^k$ for some $u \in F$ and $k \in \Z \setminus 0$. Since $x^m \in H$, we have that
    \[ w:= u \phi^k(u) \ldots \phi^{(m-1)k}(u) \in K.\]
    Then $\phi^k(w) = u^{-1} w \phi^{mk}(u)$, and also $\phi^k(w) = w$ since $t$ centralises $K$. Hence 
    \[ t^{-mk} u t^{mk} = \phi^{mk}(u) = w^{-1}uw.\]
    Now, one checks that $t^{mk}w^{-1} = (t^ku^{-1})^m$ and thus
    \[  (t^k u^{-1})^{-m} u (t^ku^{-1})^m = u.\]
    Hence, setting $\psi \in \Aut(F)$ to be the automorphism induced by the conjugation action of $t^ku^{-1}$ on $F$, we get that $\psi^m(u) = u$. Then, by neatness we must have that $\psi(u) = u$, and thus $\phi^k(u) = u$. Hence, $u^m \in K$ and by the argument above we must have that $u \in K$. Hence $ut^k \in H$.
\end{proof}

Note that if $G$ is \{finitely generated free\}-by-cyclic then $\cd_{\Z}(G) = 2$ and thus if $H \leq G$ is abelian then $H$ is isomorphic to one of $1, \Z,$ or $\Z^2$.  

\begin{lemma}\label{lem:abelian-subgroups-finite index}
    Let $G$ be \{finitely generated free\}-by-cyclic. Let $H, K \leq G$ be non-trivial abelian subgroups of $G$. Then there exists a finite index subgroup $G' \leq G$ such that $G' = F' \rtimes \langle s \rangle$ where $F'$ is finitely generated free and for each $J \in \{H,K\}$, the following holds. 
    \begin{enumerate}
        \item If $J \cong \Z$ then $J \cap G' = \langle vs \rangle$ for some $v \in F'$, or $J \cap G' =  \langle u \rangle$ where $u \in F'$ is primitive.
        \item If $J \cong \Z^2$ then $J \cap G' = \langle u, vs \rangle$ for some $v \in F'$, where $u \in F'$ is primitive.
    \end{enumerate} 
    Moreover, we can pick the fibre $F'$ to be a subgroup of any given fibre of $F$, and the monodromy to be a power of the corresponding monodromy in $G$. \end{lemma}

\begin{proof}
    Let $G = F \rtimes \langle t \rangle$ be a \{finitely generated free\}-by-cyclic splitting and let $\chi \colon G \to \Z$ be the character corresponding to the splitting. For each $J \in \{H,K\}$, let $x_J \in F$ denote the (possibly trivial) generator of $J \cap F$. By \cite[Lemma~6.5]{dahmani_unipotent_2024}, there exists a finite index normal subgroup $F' \trianglelefteq_f F$ such that if $x_J \neq 1$ then $\langle x_J \rangle \cap F'$ is a free factor of $F'$. 
    
    Now for each $J \in \{H,K\}$, let $\mathcal{O}_J = | \Z / \chi(J) |$ and define 
    \[\mathcal{O} := \lcm\{\mathcal{O}_H, \mathcal{O}_K, [F \colon F']!\}.\]
    Let $s = t^{\mathcal{O}}$. Then, $(F')^{s} = F'$ and we can set $G': = \langle F', t^{s} \rangle  \cong F' \rtimes \langle s \rangle  \leq_f G$. It follows that $J \cap F'$ is trivial or primitive, and if $\chi(J) \neq 0$, then there exists some $v \in F'$ such that $vs \in J$.
\end{proof}

We end this section by recording the graph-of-groups splittings for \{finitely generated free\}-by-cyclic groups with polynomially growing monodromies that will be used throughout the paper.  

\begin{thm}[see {\cite[\S 3.1]{dahmani_unipotent_2024}}]\label{thm:unipotent-linear-splitting}
    Let $F$ be finite rank free. Let $\phi \in \Aut(F)$ be a representative of a unipotent and linearly growing outer automorphism and let $G = F \rtimes_{\phi} \langle t \rangle$. Then $G \cong \pi_1(\calg, v_0)$ where $\calg$ is a graph of groups with a bipartite underlying graph $(X, V_0(X), V_1(X))$ that satisfies the following properties. 
    \begin{enumerate}
    \item For any vertex $b \in V_0(X)$, the group $\calg_b$ is a maximal subgroup of $G$ of the form $F_b \oplus \langle t_b\rangle$ where $F_b \leq F$ is a maximal cyclic subgroup and $t_b \in Ft$. We call the vertices in $V_0(X)$ \emph{black vertices}.
    \item For any vertex $w \in V_1(X)$, the group $\calg_w$ is maximal subgroup of the form $F_w \oplus \langle t_w \rangle$, where $F_w$ is a finitely generated non-abelian subgroup of $F$ and $t_w \in Ft$. We call the vertices in $V_0(X)$ \emph{white vertices}. The subgroup $F_w$ is called the \emph{local} fibre of $\calg_w$ and $t_w$ the \emph{central element}.
    \item Edge groups are isomorphic to maximal $\ZZ^2$ subgroups of $G$ and map surjectively onto vertex groups in $V_0(X)$.
    \item The action of $G$ on the Bass--Serre tree corresponding to the splitting $G \cong \pi_1(\calg, v_0)$ is 4-acylindrical.
\end{enumerate}
\end{thm}

The following proposition (modulo the action being $2$-acylindrical) is well known to experts and can be found in \cite{Macura2002}, \cite[Theorem 4.22]{BestvinaFeighnHandel2005}, \cite{Hagen2019}, and \cite[Proposition~2.5]{AndrewHughesKudlinska2024}.  A proof that the action is $2$-acylindrical can be found in \cite[Lemma~5.2]{KudlinskaValiunas2025}.

\begin{prop}\label{prop:std-splitting-higher-growth}
   Let $F$ be finite rank free. Let $\phi \in \Aut(F)$ be a representative of a unipotent and linearly growing outer automorphism and let $G = F \rtimes_{\phi} \langle t \rangle$. Then $G$ admits a 2-acylindrical splitting $G \cong \pi_1(\calg)$. The vertex groups $\calg_v$ are of the from $\calg_v = F_v \rtimes_{\phi_v} \langle t_v \rangle$ for $F_v \leq F$ finitely generated, $t_v \in Ft$, and $\phi_v \colon F_v \to F_v$ an automorphism that represents a unipotent and polynomially growing outer automorphism of strictly lower degree growth. Moreover, each edge group is of the form $\calg_e = \langle t_e \rangle$ where $t_e \in F t$.
\end{prop}

\begin{defn}[Standard splitting] \label{defn:std_splitting}
We will call the graph of groups splittings in \cref{thm:unipotent-linear-splitting} and \cref{prop:std-splitting-higher-growth} the \emph{standard splitting} for $(G, \chi)$. 
\end{defn}

\section{Residual nilpotency and periodic quotients}\label{sec:resTFN}

A group $G$ is \emph{residually (torsion-free) nilpotent} if for every non-trivial element $g\in G$ there exists a homomorphism $\alpha_g\colon G\to N$ such that $N$ is (torsion-free) nilpotent and $\alpha_g(g)$ is non-trivial. In \cite[Proposition 7.8]{bardakov_residually_2021} it is shown that that all free-by-cyclic are virtually residually nilpotent (see also \cite{AschenbrennerFriedl2013}). In this section we will show that \{finitely generated free\}-by-cyclic groups with unipotent and polynomially
growing monodromies are residually nilpotent.

For this section let $F$ denote a finite rank free group and let $\phi$ be a polynomially growing automorphism in $\Aut(F)$. We denote the commutator of two elements by $[x,y]=x^{-1}y^{-1}xy$ and the commutator of two subgroups by $$[A,B]=\langle  [a,b]: a\in A,b\in B\rangle.$$  We set $\gamma_1F = F$ and inductively define the terms of the \emph{lower central series} $\gamma_{n+1} = [\gamma_nF,F]$. We say an element $g\in F$ has \emph{weight $n$}, denoted $\wt(g)=n$ if and only if $g \in \gamma_nF$.

\begin{lemma}\label{lem:comutator_facts}
Let $a,b,c$ be elements of a group. The following conclusions hold:
    \begin{enumerate}
        \item\label{com:prod} $[ab,c] = [b,[a,c]][a,c][b,c]$;
        \item\label{com:inv} $[a,b]^{-1} = [b,a]$; 
        \item\label{com:prod2} $[c,ab] = [c,b][c,a][[a,c],b]$;
        \item\label{com:wt-sum} $\wt([a,b])\geq \wt(a)+\wt(b)$;
        \item\label{prod:wt} $\wt(ab) = \min(\wt(a),\wt(b))$.
    \end{enumerate}
\end{lemma}

Repeatedly applying this lemma, and using the fact that elements of weight $n$ commute modulo $\gamma_{n+1} F$ we have,

\begin{corollary}\label{cor:commutator_factor}
    Let $\wt(a_i) \geq n'$ and $\wt(b_i)\geq n''$ for $i=0,\ldots,m$ then\[
        [a_0a_1\cdots a_m,b_0b_1\cdots b_m] = \prod_{i=1}^m\prod_{j=1}^m [a_i,bj]
    \] modulo $\gamma_{n+1} F$, where $n=n'+n''$.
\end{corollary}

The following is a classical result of Magnus \cite{Magnus1935}. 

\begin{thm}[Magnus \cite{Magnus1935}]\label{thm:res_nilp}
    If $F$ is a finitely generated free group then $F$ is residually torsion-free nilpotent, \[
    \bigcap_{i=1}^\infty \gamma_iF = \{1\}.
    \]
    In particular for any finite set $S\subset F$ there is some $c(S)$ such that the set $S$ is mapped injectively via the canonical quotient $F/\gamma_{c(S)}F$.
\end{thm}

Since $\gamma_nF$ is characteristic in $F$ the automorphism $\phi$ descends to an automorphism $\bar\phi_n$ of $F/\gamma_nF$. We also have that $\gamma_n F \leq F\rtimes_\phi \mathbb Z$ is a normal subgroup and that\[
(F\rtimes_\phi \mathbb Z)/\gamma_n F \cong (F/\gamma_nF)\rtimes_{\bar\phi_n}\mathbb Z
\] is polycyclic. Wolf's Theorem \cite{wolf_growth_1968} asserts that every polycyclic group is either virtually nilpotent or has exponential growth. We shall use the following which is actually a result of the proof of \cite[Proposition 14.28]{drutu_geometric_2018}.

\begin{prop}[{see proof of \cite[Proposition 14.28]{drutu_geometric_2018}}]\label{prop:virt_nilp}
    The semidirect product $(F/\gamma_nF)\rtimes_{\bar\phi_n}\langle t\rangle$ is nilpotent if all the eigenvalues of the natural induced linear map\[
    \phi_n:\gamma_nF/\gamma_{n+1}F \to \gamma_nF/\gamma_{n+1}F
    \] are precisely 1, i.e. each $\phi_n$ is unipotent.
\end{prop}

Let $\phi\in\Aut(F)$ be unipotent. We say that an ordered basis $X= (x_1,\ldots,x_r)$ is \emph{$\phi$-ordered} if $\phi(x_i) = x_iW_i(x_{i+1},\ldots,x_r)$ where $W_i(x_{i+1},\ldots,x_r)$ denotes a word in $\{x_{i+1},\ldots,x_r\}^{\pm 1}$. We remark that $F$ may not have a $\phi$-ordered basis even though $\phi$ is unipotent. We address this issue in the proof of \Cref{thm:unipotent_res_nilp}.

\begin{defn}~{\cite[\S 11.1]{hall_theory_1959}}\label{defn:basic-comm}
  Let $X = (x_1,\ldots,x_r)$ be an ordered basis of $F$. The \emph{basic commutators of $F$} form an ordered subset $(\calc,\leq) \subset F$ consisting of all possible elements in $F$ that satisfy the following properties:
  \begin{enumerate}
  \item \label{it:bc1} Either $c\in X$ or $c=[c',c'']$ where $c',c'' \in \calc$.
  \item \label{it:bc2} The order $\leq$ satisfies the following properties:
    \begin{enumerate}
    \item If $c_1,c_2 \in \calc$ and $\wt(c_1)>\wt(c_2)$ then $c_1 > c_2$.
    \item For elements in $X$ we have $x_i\leq x_j \Leftrightarrow i\leq j$.
    \end{enumerate}
    
  \item \label{it:bc3} If $c = [c',c'']\in \calc$ then we must have
    \begin{enumerate}
    \item $c' > c''$, and
    \item\label{it:order-weird} $c'' \geq (c')''$, where $c'=[(c')',(c')'']$.
    \end{enumerate}
  \item\label{it:order-extra} If $\wt([c_1',c_1'']) = \wt([c_2',c_2'')\geq 2$ then \[
      [c_1',c_1''] \geq [c_2',c_2''] \Leftrightarrow
      \begin{cases}
        c_1'' > c_2'',\textrm{~or}\\
        c_1'' =  c_2'' \textrm{~and~} c_1' \geq c_2''
      \end{cases}
    \]
  \end{enumerate}
\end{defn}

The  ``anti-lexicographic ordering'' Property \eqref{it:order-extra} in Definition \ref{defn:basic-comm} is not standard. Usually, we are free to order the basic commutators any way we like within a weight class, but this specific ordering will be crucial to the results of this section. This terminology is also abusive since while the elements of $X$ cannot be commutators, they are still  \emph{basic} commutators.

The \emph{collection process} is a rewriting process that takes a given word  $w = x_{i_1}\cdots x_{i_l} \in F$ and iteratively rewrites it as a product $w=c_1^{n_1}c_2^{n_2}\cdots$ with $c_i < c_{i+1}$ and $n_j \in \mathbb Z$, by iteratively taking the $\leq$-minimal basic commutator that is ``out of position'' and migrating it to the left into position. Since $yx=xy[x,y]$, doing so inserts commutators, but if at each step we only move $\leq$-minimal ``out of position'' commutators then all new commutators will be basic. If we work modulo $\gamma_n F$ then this process will terminate since we can ignore high weight commutators.

Let $\calc_n = \{c \in\calc: \wt(c)=n\}$ and denote by $(\calc_n,\leq)$ the set $\calc_n$ ordered by the basic commutator ordering. The following result, in particular, motivates the use of the term \emph{basic}.

\begin{thm}[{Basis Theorem \cite[Theorem 11.2.4]{hall_theory_1959}}]\label{thm:bc-basis}
  The set $\calc_n$ of basic commutators of weight $n$ maps bijectively to a basis of the free abelian group $\gamma_nF/\gamma_{n+1}F$ via the map \[
    c \mapsto c\gamma_{n+1}F.
    \]
\end{thm}


\begin{lemma}[{see \cite[\S11.1]{hall_theory_1959}}]\label{lem:collect-migrate}
  Let $v,u \in \calc$ and suppose $[v,u] \in \calc$. Let $v_0=v$ and $v_{i+1} = [v_i,u], i=0,1,2,3,\ldots$. Let $w_1 =[v,u]$ and $w_{t+1} = [w_t,v]$. Then all $v_i,w_t \in \calc$ and we have:
  \begin{align*}
    vu  &= uv[v,u] &\\
    vu^{-1}  &= u^{-1}v v_2 v_4 \cdots v_3^{-1}v_1^{-1} &= u^{-1}v W_{v,u^{-1}} [v,u]^{-1}\\
    v^{-1}u & = uv^{-1}w_2w_4\cdots w_3^{-1}w_1^{-1} &= uv^{-1} W_{v^{-1},u}[v,u]^{-1}\\
    v^{-1}u^{-1} &=u^{-1}v_1v_3\cdots v_4^{-1}v_2^{-1}v^{-1}&= u^{-1}[v,u]W_{v^{-1},u^{-1}}v^{-1},\\
  \end{align*}
  modulo $\gamma_n F$ for all $n>0$. Also $\wt(W_{u^{\pm 1},v^{\pm 1}}) > \wt([v,u]).$
\end{lemma}

The factors $W_{v^{\pm_1},u^{\pm_1}}$ will be called $W$-factors.
We will now extend the ``antilexicographic'' order $\leq$ on basic commutators to commutators of the form $[c',c'']$ where $c',c'' \in \calc^{\pm 1}$ naturally as follows: firstly $c \leq c^{-1}$ and $c^{-1} \leq c$ and secondly if $\wt(c_1)=\wt(c_2)$ then $c_1=[c'_1,c''_1] \leq c_2=[c'_2,c''_2]$ if and only if either $c''_1< c''_2$ or  $c''_1 = c''_2$ and $c'_1\leq c'_2$.

The following three lemmas show how to rewrite ``badly formed'' commutators into products of $\leq$-larger basic commutators and their inverses.

\begin{lemma}\label{lem:nilp-base-case}
    Let $c_0=[a_0,b_0] \in \calc$, let $y_i,y_j \in \calc$ and let $c_{ij}=[y_i^{\epsilon_i},y_j^\epsilon{j}]$, where $\epsilon_1,\epsilon_2 \in \{-1,1\}$, with $c_0 < c_{ij}$ in the extended ordering. Suppose $\wt(c_0)=\wt(c_{ij})=n$ and suppose $\wt(a_0)=\wt(b_0)=\wt(y_i)=\wt(y_j)=n/2$. If $c_{ij} \not \in \calc$ then there is basic commutator $b_{ij}\in\calc$ with $c_0 < b_{ij}$ such that\[
    c_{ij} = b_{ij}^{\epsilon'} \mod \gamma_{n+1}F,
    \] for some $\epsilon' \in\{-1,1\}$.
\end{lemma}
\begin{proof}
    There are 4 cases to consider. In all cases, due to our hypotheses on weights we always have $y_j>y_i''$, where $y_i=[y_i',y_i'']$. In particuar $[y_i,y_j]$ is a basic commutator precisely when $y_i>y_j$. In all cases, even if $[y_i,y_j]$ is not a basic commutator, we still have $y_j\geq b_0$ and $y_i\geq a_0$. In particular, even if $y_j>y_i$, we still have $y_i\geq a_0 > b_0$.

    \textbf{Case 1:} $\epsilon_i,\epsilon_j=1$. In this case if $[y_i,y_j]$ is already a basic commutator there is nothing to show. If $[y_i,y_j]$ is not a basic commutator then $y_i<y_j$. Thus $[y_i,y_j] = [y_j,y_i]^{-1}$ is the inverse of basic commutator and as explained above $y_i>b_0$ so $[a_0,b_0]< [y_j,y_i]$.

  \textbf{Case 2} $\epsilon_i=1,\epsilon_j = -1$. We have $[y_i,y_j^{-1}]$. We will use the collection process to express this as a product of basic commutators. Consider first the case where $y_i > y_j$ (so that $[y_i,y_j]$ is a basic commutator). We will start by migrating $y_j$ symbols left and will be repeatedly using Lemma \ref{lem:collect-migrate}:
  \begin{equation*}
    [y_i,y_j^{-1}] = y_i^{-1}y_jy_iy_j^{-1} = y_i^{-1}\cancel{y_jy_j^{-1}}y_i\cancel{W_{y_i,y_j^{-1}}} [y_i,y_j]^{-1} = [y_i,y_j]^{-1}.
  \end{equation*}
  Note that here we can cancel the $W$-factors since they will have weight $n+1$ or more. Which, as seen in Case 1, is the inverse of a basic commutator that is greater than $c_0$. The next possibility is $b_0 < y_i < y_j$, this time we will start by migrating $y_i$ and simply ignore $W$-factors and immediately cancel all commutators of weight more than $n$:
  \begin{equation*}
    [y_i,y_j^{-1}]= y_i^{-1}y_jy_iy_j^{-1} = \cancel{y_i^{-1}y_i}y_j[y_j,y_i]y_j^{-1}=\cancel{y_jy_j^{-1}}[y_j,y_i]\cancel{[[y_j,y_i],y_j^{-1}]} = [y_j,y_i].
  \end{equation*}
  Again, in this case, $c_0 = [a_0,b_0]<[y_j,y_i]$.

  \textbf{Case 3} $\epsilon_i=1,\epsilon_j = -1$. Calculations completely analogous to those in Case 2 will rewrite the commutator as a basic commutator or the inverse of a basic commutator that is greater than $c_0$.

  \textbf{Case 4} $\epsilon_i= \epsilon_j = -1$. We consider first the case where $y_i>y_j$. And proceed as before
  \begin{multline*}
    [y_i^{-1},y_j^{-1}] = y_iy_jy_i^{-1}y_j^{-1}\\ = y_i\cancel{y_jy_j^{-1}}[y_i,y_j]\cancel{W_{y_i^{-1},y_j^{-1}}}y_i^{-1} = \cancel{y_iy_{i}^{-1}}\cancel{[[y_i,y_j],y_i^{-1}]}=[y_i,y_j].
  \end{multline*}
  Which as we've seen before will be greater than $c_0$. The case $y_j>y_i>x_j$ is handled similarly.
\end{proof}

\begin{lemma}\label{lem:flip-signs}
  Let $a,b,c \in \calc$ with  $a<b$, $c<a$ and $[c,a]>b$ so that $[[c,a],b]$ is  a basic commutator, say of weight $n$. Then we have
  \begin{enumerate}
  \item $[[c,a]^{-1},b] = [[c,a],b]^{-1}$
  \item $[[c,a],b^{-1}] = [[c,a],b]^{-1}$
  \item $[[c,a]^{-1},b^{-1}] = [[c,a],b]$,
  \end{enumerate}
  modulo $\gamma_{n+1}F$.
\end{lemma}
\begin{proof}
  Consider the first equation. Noting the weight of the $W$-factors we have:
  \begin{multline*}
    [[c,a]^{-1},b] = [c,a]b^{-1}[c,a]^{-1}b 
    = [c,a]\cancel{b^{-1}b}[c,a]^{-1}W_{[c,a]^{-1},b}[[c,a],b]^{-1} \\= [[c,a],b]^{-1} \left(\underbrace{W_{[c,a]^{-1},b}[W_{[c,a]^{-1},b},[[c,a],b]^{-1} ]}_{\wt > n} \right)
  \end{multline*}
  The other equations follow similarly.
\end{proof}

\begin{lemma}\label{lem:explosions}
  Let $a,b,c,[c,b] \in \calc$ with $[c,b]>a$ but $b>a$ so that $[[c,b],a]$ is not a basic commutator then if $\epsilon_1,\epsilon_2 \in \{-1,1\}$,  $q = [[c,b]^{\epsilon_1},a^{\epsilon_2}]$ can be rewritten as a product \[
    q=[[c,b]^{\epsilon_1},a^{\epsilon_2}]=b_1\cdots b_s \mod \gamma_{n+1}F
    \] of (possibly repeated) basic commutators of weight $n$, where $\wt(q)=n$, such that $q < b_i, i=1,\ldots, s$. 
\end{lemma}
\begin{proof}
  The hypotheses imply that $c>b>a$ and that $[[c,a],b]$ is a basic commutator. For this proof, we will be using the following terminology: we will say that basic commutators $q_1<q_2$ are \emph{commutable} if $[q_2,q_1]\in \calc$. If $q_1<q_2$ are not commutable then we will say \emph{$q_1$ explodes $q_2$} and if $q_2=[q_2',q_2'']$ we will say that \emph{$q_2',q_2''$ are the debris} of the explosion. Explosions occur in the collection process when we have a subword $[q_2',q_2'']q_1$ with $q_1<[q_2',q_2'']$ and $q_2'' > q_1$. Then the basic commutator $[q_2',q_2'']$ must be replaced by the product $q_2'^{-1}q_2''^{-1}q_2'q_2''$ of basic commutators and their inverses. The following is immediate from definitions, but central to the argument of the proof.
  
  \textbf{Fact:} \emph{If $q_1$ explodes $q_2$ then $q_1<q_2',q_2''$ where $q_2',q_2''$ is the debris of the explosion.}
  
  Let us first consider the case where $\epsilon_1=\epsilon_2=1$ and where $[c,a],[b,a] \in \calc$. We turn the non-basic commutator into a product of basic commutators by expanding it and applying the collection process\[
    [[c,b],a] = [c,b]^{-1}a^{-1} [c,b]a
  \]
  We see here that $[c,b]$ and $a$ are not commutable, which means we  must explode $[c,b]$:\[
[c,b]^{-1}a^{-1} [c,b]a = [c,b]^{-1}a^{-1}c^{-1}b^{-1} c b a.\\
    \]
  The smallest basic commutator is $a$ so we migrate it to the left until it cancels with its inverse. We underline it to aid in keeping track of the process
  \begin{eqnarray*}
     [c,b]^{-1}a^{-1}c^{-1}b^{-1} c b \underline{a}&&\\
    &=& [c,b]^{-1}a^{-1}c^{-1} b^{-1} \underline a c [c,a] b [b,a]  \\
    &=& [c,b]^{-1}a^{-1}c^{-1} \underline a b^{-1} W_{b^{-1},a} [b,a]^{-1}c [c,a] b [b,a]\\
    &=& [c,b]^{-1}\cancel{a^{-1} \underline a} c^{-1} W_{c^{-1},a}[c,a]^{-1}b^{-1} W_{b^{-1},a} [b,a]^{-1}c [c,a] b [b,a]\\
    &=&[c,b]^{-1}c^{-1} W_{c^{-1},a}[c,a]^{-1}b^{-1} W_{b^{-1},a} [b,a]^{-1}c [c,a] b [b,a].\\
   \end{eqnarray*}
  Where from Lemma \ref{lem:collect-migrate} we have
  \begin{eqnarray*}
    W_{b^{-1},a} & =& [[b,a],b]\cdot [[[[b,a],b],b],b] \cdots  [[[b,a],b],b]^{-1}\\
    W_{c^{-1},a} &=& [[c,a],c]\left(\cancel{\cdots  [[[c,a],c],c]^{-1}}\right).
  \end{eqnarray*}
  where we cancel off the terms that clearly have weight greater than $n$. At this point in the collection process we have cancelled out all $a$ symbols and we have a product of basic commutators that are all strictly greater than $a$.

  Consider now the general case
  \[   
    [c,b]^{\epsilon_1}a^{\epsilon_2} [c,b]^{-\epsilon_1}a^{-\epsilon_2},  
  \]
  where we no longer assume that $a$ is commutable with $b$ or $c$. In all cases, the collection process will migrate the rightmost $a^{\pm 1}$ to the left until it cancels with the other $a^{\mp 1}$ symbol. Throughout the collection process basic commutators are created when $a$ is commutable with its leftmost neighbour.  Note that all created $W$-factors will have weight greater than $a$, or exploded, but the debris will remain strictly greater than $a$. It follows that once the $a$ symbols cancel out, we will be left with a product of basic commutators that are all strictly greater than $a$.
  
  Now the collection process will continue to rewrite the product. In doing so many basic commutators will be created, exploded, or cancelled out. Because of our fact about explosion debris, all basic commutators that will occur for the remainder of the collection process will remain strictly greater than $a$.
  
  In the end no basic commutators of weight less than $n$ will remain,  furthermore in the final product of basic commutators, every basic comutator factor will be of the form $b_i = [b_i',b_i'']$ where $b_i',b_i''$ are commutators that are strictly greater than $a$.  In particular $b_i''>a$ so $b_i > [[c,b],a]$ as required.
\end{proof}

\begin{lemma}\label{lem:incr}
  Let $\phi \in \Aut(F)$ be unipotent and let $X$ be a $\phi$-ordered basis of $F$ that in turn induces the order $\leq$ on $\calc$. If $c$ is a basic commutator with $\wt(c)=n$ then\[
    \phi(c) = c c_1^{n_1}\cdots c_m^{n_m} \mod \gamma_{n+1}F
  \] where $c<c_1<\cdots c_m$ are basic commutators with the same weight as $c$, $n_i \in \mathbb Z$. In particular the induced the matrix representation of the induced linear map\[
  \phi_n:\gamma_n F/\gamma_{n+1}F \to \gamma_n F/\gamma_{n+1}F
  \] with respect to the ordered basis $(\calc_n,\leq)$ is lower unitriangular.
\end{lemma}
\begin{proof}
  We proceed by induction on $n=\wt(c)$. The base case is $n=1$ where $\phi_1$ is the induced automorphism of the abelianization of $F$. $C_1 = \{x_1,\ldots,x_r\}$ and $x_i< x_j \Leftrightarrow i <j$. By definition of $\phi$-ordered we immediately get\[
  \phi(x_i) = x_ix_{i+1}^{n_{i+1}}\cdots x_r^{n_r} \mod \gamma_2F
  \] for some $n_{i+1},\ldots,n_r \in\ZZ.$ The lower unitriangularity of the matrix representation follows immediately.
  
    Suppose now that the result was true for all weights up to $n$. Let $c=[c',c'']$ be such that $\wt(c)=n+1$. By induction hypothesis and Corollary \ref{cor:commutator_factor} we get
  \begin{multline*}
    \phi(c) = [\phi(c'),\phi(c'')]\\ = [a_0 a_1^{n_1}\cdots a_p^{n_p} R,  b_0 b_1^{m_1} \cdots b_q^{m_q} S] = [a_0,b_0]\left(\prod_{(i,j)\neq(0,0)}\left[a_i^{|n_i|/n_i},b_j^{|m_j|/m_j}\right]^{n_i m_j}\right) T  
  \end{multline*}
  
  where $a_0=c',b_0=c''$, $\wt(a_i)=\wt(a_0)$ and $\wt(b_j)=\wt(b_0)$ for all $i,j$, and where $\wt(a_0)<\wt(R)$ and  $\wt(b_0)<\wt(S)$. By induction hypothesis, the sequences of basic commutators $a_0,a_1,a_2,\ldots$ and $b_0,b_1,b_2,\ldots$ are strictly increasing with respect to $\leq$ and $\wt(T) > n+1$. Now for $(i,j) \neq (0,0)$ each of the commutators\[
    \left[a_i^{|n_i|/n_i},b_j^{|m_j|/m_j}\right]
  \] is greater than $c=[c',c'']=[a_0,b_0]$ with respect to the extended ordering $\leq$ for $i,j>0$, but some may not be a basic commutator due to a combination of the the signs $|n_i|/n_i,|m_j|/m_j$ possibly being negative or $[a_i,b_j]$ not being basic. 

  If $[a_i,b_j]$ is not basic because $b_j>a_i$, then $\wt(a_i)=\wt(b_j)$ and \Cref{lem:nilp-base-case} lets us rewrite it as $[b_j,a_i]^{\pm1}$ which is a basic commutator mod $\gamma_{n+2}F$.

  Otherwise we have $\wt(a_i)>\wt(b_j)$ and \Cref{lem:flip-signs} or \Cref{lem:explosions} allows us to rewrite the commutator as a product of $\leq$-strictly greater basic commutators. Finally noting that all commutators of weight $n+1$ commute modulo $\gamma_{n+2}F$ the first part of the result follows.

  By Theorem \ref{thm:bc-basis} $(\calc_n,\leq)$ gives a basis of $\gamma_{n+1}F/\gamma_{n+2}F$ and $\phi_{n+1}$ is easily seen be lower unitriangular. The result now follows by induction.
\end{proof}

\begin{thm}\label{thm:unipotent_res_nilp}
  If $\phi \in \Aut(F)$ is unipotent then \[
    \left(F \rtimes_{\phi}\mathbb Z\right)/\gamma_n F
    \] is torsion-free nilpotent.
\end{thm}
\begin{proof}
    If $\phi$ is unipotent then by \cite[Theorem 5.1.5]{bestvina_tits_2000} there is a connected directed graph $\Gamma$ such that $E(\Gamma) = \{e_1,\ldots,e_r\}$ and a homotopy equivalence $\Phi:\Gamma\to\Gamma$ that, for $i>1$, maps $e_i$ to a concatenation $e_i\mu_i$ where $\mu_i$ is a (possibly empty) loop that is itself a concatenation of edges (possibly traversed with or against orientation) that lie in $\{e_1,\ldots,e_{i-1}\}$. We can pick any vertex $v \in \Gamma$ identify $F = \pi_1(\Gamma,v)$ and $\Phi$ will be a representative for the class $[\phi]$ in $\Out(F)$. Since $\Phi(v)=v$ we may assume without loss of generality that, under the $\pi_1$-functor, we have $\Phi_\sharp = \phi$.

    Consider the quotient map $q:\Gamma \to \Gamma_\bullet$ obtained by identifying all the vertices of $\Gamma$ to obtain a bouquet of circles with the single vertex $\bullet$. The map $q$ is $\pi_1$-injective and maps $F$ to a free factor of $F*F_r = \pi_1(\Gamma_\bullet,\bullet)$. Now, as a bouquet of circles, we can also view $\pi_1(\Gamma_{\bullet},\bullet) = F(E(\Gamma))$, the free group on the (abstract) set $E(\Gamma)$. Let $X = E(\Gamma)$ and reverse its order so that $e_i<e_j$ if and only if $i> j$. 
    
    Now $\Phi$ naturally descends to a homotopy equivalence $\Phi_\bullet$ of $\Gamma_\bullet$ and since it preserves the unique basepoint it induces the automorphism $\phi_\bullet \in \Aut(F(X))$ given by\[
    e_i \mapsto \mu_i,
    \] where $\mu_i$ can be interpreted as a string in $X^{\pm_1}$. It follows that $(X,\leq)$ is a $\phi_\bullet$-ordered basis for $\pi_1(\Gamma_\bullet)=F(X)$. 
    
    Now the basic commutators of weight $n$ map to an ordered basis of the free abelian group\[
    \gamma_nF(X)/\gamma_{n+1}F(X).
    \] By Lemma \ref{lem:incr} and Proposition \ref{prop:virt_nilp}, for all $N>0$ the quotient\[
       \left(F(X)\rtimes_{\phi_\bullet}\mathbb Z\right)/\gamma_NF(X) \cong \left(F(X)/\gamma_NF(X)\right)\rtimes_{\overline{\phi}_{\bullet_N}}\mathbb Z
    \]
    is torsion-free nilpotent.
    
    We note that although $F$ is a free factor of $F(X)$ this particular free factorization is not obtained from a partition of $X$, besides none these free factorizations will be $\phi_\bullet$-invariant

    By our construction we do have that the image of $F$, which we identify with $F$ is $\phi_\bullet$ invariant and that $[\phi_\bullet|_F]\in \Out(F)$ as an outer automorphism is represented by $\Phi$. Thus, without loss of generality we may assume that $\phi_\bullet|F=\phi$ and in fact that we have an embedding
    \begin{eqnarray*}
        F\rtimes_\phi\langle s\rangle  & \hookrightarrow & F(X)\rtimes_{\phi_\bullet}\langle t\rangle\\
        fs^n &\mapsto & ft^n,
    \end{eqnarray*}
    where we identify $f\in F$ with its image in $F(X)$. We further note that since $F$ is a free factor of $F(X)$ we have \[
        \gamma_nF(X)\cap F = \gamma_n F
    \] for all terms of the lower central series. It therefore follows that we have a natural embedding $$\left(F\rtimes_\phi \mathbb Z\right)/\gamma_NF \hookrightarrow \left(F(X)\rtimes_{\phi_\bullet}\mathbb Z\right)/\gamma_NF(X).$$ 

    The result now follows since torsion-free nilpotency is inherited by subgroups.
\end{proof}

\begin{thm}[{\cite[Theorem 17.2.5]{kargapolov_fundamentals_1979}}]\label{thm:nilp-emb}
    Let $G$ be a finitely generated torsion-free nilpotent group, then there exists an integer $n=n(G)$ such that $G$ embeds in $\mathrm{UT}_n(\mathbb Z)$, the group of $n\times n$ upper unitriangular integral matrices.
\end{thm}

The following proposition is a simple exercise. 

\begin{prop}\label{prop:p-quotients}
    For any prime $p>n$ the group $\mathrm{UT}_n(\mathbb Z/p\mathbb Z)$ is a group with exponent $p$, i.e. every nontrivial element has order $p$.
\end{prop}

\begin{corollary}\label{cor:periodic-quotients}
    Let $\phi\in\Aut(F)$ be unipotent. For any finite set $S\subset F\rtimes_\phi\mathbb Z$ there exists some $N(S) \in\mathbb Z_{\geq 0}$ such that for any prime $p>N(S)$ there is a finite $p$-periodic quotient $Q_p^S$ of $F\rtimes_\phi\mathbb Z$ in which $S$ is mapped injectively. 

    Furthermore if $g,h\in F\rtimes_\phi\mathbb Z$ and the commutator $[g,h]$ does not vanish in $Q_p^S$ then the images of the cyclic groups $\langle g\rangle, \langle h\rangle$ will have order precisely $p$ and will have trivial intersection.

\end{corollary}
\begin{proof}
    Let $S =\{f_1t^{n_1},\ldots f_mt^{n_m}\}$. By Theorem \ref{thm:res_nilp} there is a sufficiently large $N$ so that if $f_i$ is non-trivial then the image of $f_i$ survives in $F/\gamma_N F$. Each element of $S$ will therefore be mapped non-trivially to \[
        G_N:=\left(F\rtimes_\phi\mathbb Z\right)/\gamma_N F \cong  \left(F/\gamma_N F\right)\rtimes_{\overline\phi_N}\mathbb Z,
    \] which by Theorem \ref{thm:unipotent_res_nilp} is torsion-free nilpotent. By Theorem \ref{thm:nilp-emb} $N$ embeds into $\mathrm{UT}_d(\mathbb Z)$ for some $d=d(N)$. Looking at the matrix images of the elements of $S$ in $\mathrm{UT}_d(\mathbb Z)$, we see that if we pick a prime $p$ greater than $N_1$, the maximal absolute value of the coefficients of the matrices occuring in the image of $S$, then $S$ will be mapped injectively via\[
    F \rtimes_\phi\mathbb Z \twoheadrightarrow G_N \hookrightarrow \mathrm{UT}_d(\mathbb Z) \twoheadrightarrow \mathrm{UT}_d(\mathbb Z/p\mathbb Z).        
    \] If $p$ is chosen to be be greater than $\max(N_1,d)=N(S)$ then the image $Q_p^S$of $F \rtimes_\phi\mathbb Z$ will be $p$-periodic. The first part of the proof follows.

    Suppose now that $[g,h]$ did not vanish in the quotient to $Q_p^S$. Then neither $g$ nor $h$ vanished so, by $p$-periodicity, their images both generate subgroups isomorphic to $\mathbb Z/p\mathbb Z$. Suppose towards a contradiction the images $\langle g\rangle, \langle h\rangle$ had non-trivial intersection. Then by they  structure of $\mathbb Z/p\mathbb Z$, the images must coincide and the images of both $g$ and $h$ will generate this intersection. This means that the image of $g$ will be a power of the image of $h$, forcing their images to commute, contradicting the assumption that $[g,h]$ had non-trivial image. It follows that the images of $\langle g\rangle, \langle h\rangle$ must have trivial intersection and the proof is complete.
\end{proof}

We will also need the following:

\begin{lemma}\label{lem:periodic-quotients-commuting}
    Let $G = F \rtimes_{\phi} \langle t \rangle$ where $\phi \in \Aut(F)$ is unipotent and $F$ is finitely generated free with $\mathrm{rank}(F) > 1$. Let $g,h \in G \setminus 1$ be two commuting elements such that $\langle g \rangle \cap \langle h \rangle = 1$. Let $S \subseteq G$ be a finite subset of elements that contains $g$ and $h$. Then there exists some $N(S) \in \Z_{\geq 0}$ such that for every prime $p > N(S)$ there is a finite $p$-periodic quotient $Q_p^S$ such that each element of $S$ has non-trivial image and the images of $\langle g \rangle$ and $\langle h \rangle$ have trivial intersection.
\end{lemma}

\begin{proof}
    We begin with the following claim:

    \begin{claim}\label{product-subgroup-claim}
       The elements $g,h$ are contained in a subgroup $H = F_H \oplus \langle s \rangle\leq G$, where $F_H \leq F$ is not abelian and $s \in F\cdot t$. 
    \end{claim}

    \begin{proof}
        We argue by induction on the degree of growth of $\phi$. If $\deg(\phi) = 0$ then $\phi = \mathrm{id}$. Hence we can take $H$ to be the entire group $G$.
        Now suppose that $\deg(\phi) > 0$. Let $G \cong \pi_1(\calg)$ be the standard splitting as in \Cref{defn:std_splitting}. Recall that the standard splittings are acylindrical.
        
        Since $g$ and $h$ are elements of infinite order that commute, by acylindricity it must be the case that $g$ and $h$ are elliptic and $\Fix(g) \cap \Fix(h) \neq \emptyset$. Pick a vertex $v$ in the intersection. Then, $g,h \in \calg_v$ and $\calg_v = F_v \rtimes \langle t_v \rangle$ for some $F_v \leq F$ and $t_v \in F\cdot t$. The conjugation action of $s$ on $F_v$ induces a unipotent and polynomially growing automorphism of degree $d < \deg(\phi)$. If $\rank(F_v) >1$ then we may conclude the required result by induction.

        Hence, suppose that $\rank(F_v) = 1$ and assume first that $\deg(\phi) = 1$. It must be the case that $\calg_v$ is a black vertex of the standard splitting (as defined in \cref{thm:unipotent-linear-splitting}), and thus for any incident edge $e$, the map $\iota_e \colon \calg_e \to \calg_v$ is onto. Hence, $\calg_v$ can be identified with a subgroup of $\calg_w$ where $w = \tau(e)$. Then, $\calg_w$ is a white vertex and thus is of the form $\calg_w = F_w \oplus \langle t_w \rangle$ where $F_v \leq F$ is finitely generated of rank greater than 1. Again, the result follows.

        Suppose now that $\rank(F_v) = 1$ and $\deg(\phi) > 1$. Suppose that there exists a non-loop edge $e$ incident at $v$ and let $w$ be the other endpoint of $e$. We may write $\calg_v = \langle x_v \rangle \oplus \langle t_v \rangle$ for $x_v \in F$, $t_v \in F\cdot t$, and $\calg_w = F_w \rtimes_{\phi_w} \langle t_w \rangle$ for $F_w \leq F$, $t_w \in F\cdot t$, and let \[K := \langle \calg_v, \calg_w \rangle \cong \calg_v \ast_{ x_v^l t_v = t_w} \calg_w,\]
        for some $l \in \Z$. Note that $\phi_w$ is unipotent and polynomially growing and $\deg(\phi_w) < \deg(\phi)$.
        Thus $K$ admits a presentation 
        \[ K \cong \left\langle F_w, x_v, t_w \middle\mid a^{t_w} = \phi_w(a) \forall a \in F_w, x_v^{t_w} = x_v \right\rangle \cong F' \rtimes_{\tilde{\phi}} \langle s \rangle\] where $F' \leq F$ is finitely generated, $s \in F\cdot t$ and $\tilde{\phi}$ is a unipotent automorphism of degree $ d< \mathrm{deg}(\phi)$. Again, the result follows by induction.

       Suppose now that there is some loop edge $e$ at $v$. Then, the stable letter $a$ is an element of the fibre $F$ and we have that 
       \[ K:= \langle \calg_v, a \rangle \cong \calg_v \ast_{{x_v}^{l}t_v \sim {x_v}^{k}t_v, a}\]
       for some $l, k \in \Z$. One checks that 
       \[  K  \cong \langle x_v, a, s \mid x_v^{s} = x_v, a^s = ax_v^{k - l} \rangle,\] where $s = x_v^{l}t_v \in F \cdot t$. Thus, either $K = F' \oplus \langle s \rangle$ where $F' = \langle x_v, a \rangle \leq F$, or $K = F' \rtimes_{\tilde{\phi}} \langle s\rangle$ where $\tilde{\phi}$ is unipotent and linearly growing. In the former case, the result follows immediately. For the latter case, since $\tilde{\phi}$ grows linearly and $\mathrm{rank}(F') > 1$, we may argue by induction.

       The final case to consider is $\mathrm{rank}(F_v) = 0$. Then, $\calg_v$ is infinite cyclic and $g,h \in \calg_v$, contradicting the assumption that $\langle g \rangle \cap \langle h \rangle = 1$.
    \end{proof}

    Let $H = F_H \oplus \langle s \rangle$ be as in \cref{product-subgroup-claim}. Since $g$ and $h$ commute, it must be the case that $g = u^as^b$ and $h = u^cs^d$ for some $u \in F_H$ not a proper power and integers $a,b,c,d \in \Z$ such that
    \[\det \begin{pmatrix} a & b \\ c & d\end{pmatrix} \neq 0.\] Since $\mathrm{rank}(F_H) > 1$, there exists some $v \in F_H$ which does not commute with $u$. 

    Let $S' = S \cup \{u, v, s,[u,v]\}$ and let $N(S') > 0$ be the integer so that for each prime $p > N(S')$, there is a $p$-periodic quotient $\pi \colon G \to Q_p^{S'}$ so that the image of each element of $S'$ is non-trivial, which exists by \cref{cor:periodic-quotients}.
    
    Fix $p > N(S')$ and suppose, for contradiction, that $\langle \pi(u) \rangle \cap \langle \pi(s) \rangle \neq 1$. Since $Q_p^{S'}$ is $p$-periodic, we must have that $\langle \pi(u) \rangle = \langle \pi(s) \rangle$ and thus $\pi(u)$ commutes with $\pi(v)$. However, we constructed $Q_p^{S'}$ so that $\pi([u,v]) \neq 1$. 
    
    It follows that $\langle \pi(u), \pi(s) \rangle = \langle \pi(u) \rangle \oplus \langle \pi(s) \rangle \cong \Z/p\Z \times \Z /p \Z$. Now if we choose $M(S) = \max\{N(S'), \det(A) + 1\}$, then for any prime $p > M(S)$ it follows that 
     \[\det \begin{pmatrix} a & b \\ c & d\end{pmatrix} \not\equiv 0 \mod p.\]
   Hence, $\pi(u^a s^b) \not \in \langle \pi(u^c s^d) \rangle.$ It follows that $\langle \pi(g) \rangle \cap \langle \pi(h) \rangle = 1$.
\end{proof}

Combining \cref{cor:periodic-quotients} and \cref{lem:periodic-quotients-commuting} we obtain \Cref{prop} from the introduction.

\begin{duplicate}[\Cref{prop}]
    Let $G$ be a \{finitely generated free\}-by-cyclic group with unipotent and polynomially growing monodromy. Let $S\subseteq G$ be a collection of non-trivial elements and let $g,h \in G$ be such that $\langle g \rangle \cap \langle h \rangle = 1$ or $[g,h] \neq 1$. There exists $N = N(S)$ such that for every prime $p>N$ 
    there is a $p$-periodic quotient of $G$ such that the image of each element of $S$ is non-trivial and the cyclic subgroups generated by $g$ and $h$ have trivial intersection. 
\end{duplicate}

\section{Virtually free vertex fillings}\label{sec:virFreeFill}

The aim of this section is to introduce and construct virtually free vertex fillings, in particular for \{finitely generated free\}-by-cyclic groups with unipotent polynomial monodromy. 

\subsection{Constructing fillings}

\begin{defn}\label{defn:vertex-filling}
    Let $\calg = (X_{\calg}, \calg_{\bullet}, \iota_{\bullet})$ and $\calh = (X_{\calh}, \calh_{\bullet}, \kappa_{\bullet})$ be graphs of groups. A morphism $f \colon \calg \to \calh$ is a \emph{vertex filling} if the corresponding graph morphism $f \colon X_{\calg} \to X_{\calh}$ is an isomorphism, for each vertex $v \in V(X)$ the homomorphism $f_v \colon \calg_{v} \twoheadrightarrow {\calh}_{f(v)}$ is onto, and $\gamma_e = 1$ for each $e \in E(X_{\calg})$. A vertex filling induces a quotient of the corresponding fundamental groups
    \[ \Theta_f \colon \pi_1(\calg, \ast) \twoheadrightarrow \pi_1(\calh, f(\ast)).\]
\end{defn}

Let $f \colon \calg \to \calh$ be a vertex filling. Let $T_{\calg}$ and $T_{\calh}$ be the Bass--Serre trees corresponding to $\calg$ and $\calh$, respectively, and let $N = \ker \Phi$. Then $T_{\calh} = N \backslash T_{\calg}$ and the quotient 
\[\pi \colon T_{\calg} \to T_{\calh} = N \backslash T_{\calg} \]
is a $\Phi$-equivariant morphism, i.e. for every $g \in \pi_1(\calg, \ast)$ and $x \in T_{\calg}$, $\pi(g\cdot x) = \Phi(g) \cdot \pi(x)$.

\begin{defn}Let $\calg = (X, \calg_{\bullet}, \iota_{\bullet})$ be a graph of groups. For every $v \in V(X)$, let $N_v \leq_{f} \calg_v$ be a finite index subgroup. A \emph{virtually free vertex filling associated to the collection $\{N_v\}_{v\in V(X)}$} consists of a collections of finite index normal subgroups $N_v' \trianglelefteq_f \calg_v$ such that $N_v' \leq N_v$, and a vertex filling $f \colon \calg \to \calh$ such that for every $v \in V(X)$, the vertex group $\mathcal{H}_{f(v)} = \mathcal{G}_v / N_v'$, and $f_v \colon \calg_v \to \calh_{f(v)}$ is the natural quotient. 

If $S \subseteq V(X)$ is a subset of vertices, and $N_v \leq_f \calg_v$ a collection of finite index subgroups for all $ v\in S$, then a virtually free vertex filling corresponding to $\{N_v\}_{v\in S}$ is the virtually free vertex filling corresponding to $\{N_v\}_{v\in S} \cup \{\mathcal{G}_v\}_{v \in V(X) \setminus S}$.

We say that $\calg$ admits \emph{arbitrarily deep virtually free vertex fillings} if for every collection $\{N_v\}_{v\in V(X)}$ where $N_v \leq_f \calg_v$ is a finite index subgroup, there exists a virtually free vertex filling associated to it. 
\end{defn}

\begin{lemma} \label{lem:orbits-in-virtually free-fillings}
    Let $f \colon \calg \to \calh$ be a vertex filling with a corresponding morphism of trees $\pi \colon T_{\calg} \to T_{\calh}$. Fix a basepoint $\ast \in V(X)$ and let $G  = \pi_1(\calg, \ast)$ and $H = \pi_1(\calh, f(\ast))$.  Let $\rho$ and $\rho'$ be a vertex or an edge in $T_{\calg}$. Then $\rho$ and $\rho'$ are in the same $G$-orbit if and only if $\pi(\rho)$ and $\pi(\rho')$ are in the same $H$-orbit.
\end{lemma}

\begin{proof}
    By the definition of a vertex filling we have a commutative square
    \[\begin{tikzcd}
        T_\calg \arrow[r,"\pi"] \arrow[d] & T_\calh \arrow[d] \\
        X_\calg \arrow[r,"f"] & X_\calh,
    \end{tikzcd}\]
where $f$ is a graph isomorphism and the vertical arrows are the covering maps.  The claim follows immediately from commutativity since $\rho$ and $\rho'$ are in the same orbit if and only if they have the same image in $X_\calg$ and similarly for the images of $\pi(\rho)$ and $\pi(\rho')$ in $X_\calh$.
\end{proof}

\begin{lemma}\label{lemma:edge separating-fillings}
    Let $\calg = (X, \calg_{\bullet}, \iota_{\bullet})$ be a graph of groups that admits arbitrarily deep virtually free vertex fillings. Suppose that all the vertex groups are residually finite and all the edge groups are separable in the vertex groups. Then for every vertex $v \in V(X)$, every collection of (possibly repeated) images $H_{e_1},\ldots,H_{e_s} \leq \calg_v$ of edge group monomorphism into $\calg_v$, every finite set of elements $g_1,\ldots,g_s \in \calg_v$
    that respectively do not lie in the subgroups $H_{e_1},\ldots,H_{e_s} \leq \calg_v$, and any finite set of elements $h_1,\ldots,h_m \in \calg_v$, there exists a vertex filling of $\calg$ such that the image of each vertex group $\bar\calg_v$ is finite, and for $i=1,\ldots,s$, the image of $g_i$ in $\bar\calg_v$ does not lie in the image of $H_{e_i}$ and none of the elements $h_j$ are mapped to the identity. 
\end{lemma}

\begin{proof}
    Using the hypotheses that the edge groups are separable in the vertex groups and that the vertex groups are residually finite, for each vertex $v \in V(X)$ we can find a finite index normal subgroup $N_v \trianglelefteq_f \calg_v$ such that the image of each $g_i$ is not mapped into the image of the edge group $H_{e_i}$ and each $h_j$ is mapped to a non-trivial element in the quotient $\calg_v / N_v$. Then, we may pass to a virtually free vertex filling associated to the collection $\{N_v\}_{v\in V(X)}$ to obtain the required vertex filling. 
\end{proof}

If $\calg$ satisfies the conclusion of \cref{lemma:edge separating-fillings} then we say that $\calg$ admits \emph{arbitrarily deep edge separating virtually free fillings.}

\begin{lemma}\label{lemma:prescribed-orders-fillings}
    Let $G = H \rtimes_{\phi} \langle t \rangle$ where $H$ is finitely generated and let $\chi \colon G \to \Z$ be the map associated to the splitting. Let $N \leq_f G$ be a subgroup of finite index and fix a positive integer $k$. Then, there exists a positive integer $L$ such that for every $i \in k \Z_{>0} $, there is a finite index normal subgroup $N_i \trianglelefteq_f G$ such that $N_i\leq N$ and every element $g \in G$ with $\chi(g) = k$ has order $iL / \gcd(iL, k)$ in the quotient $G / N_i$
\end{lemma}

\begin{proof}

Fix a positive integer $k$ and a subgroup $N \leq_f G$ of finite index. After possibly passing to a further subgroup of finite index, which we also call $N$ by an abuse of notation, we may assume that $\chi(N) = n\ZZ$ for some positive integer $n$ and $t^n \in N$.  Then,  $N = (N \cap H) \rtimes \langle t^n \rangle$. Let $K := C_H(N \cap H) \trianglelefteq_f H$ be the characteristic core of the subgroup $N\cap H$ in $H$. Let $M = [H : K]!$ and $L = Mn$. Pick $i \in k \Z_{>0}$ and define $N_i := \langle K, t^{iL} \rangle \cong K \rtimes \langle t^{iL} \rangle$. Note that $N_i$ is a finite index normal subgroup of $G$ and $N_i  \leq N$.

Now to prove the claim about orders of elements in the quotient. For every $x \in H$ and positive integers $k,m > 0$, let
\[\Phi_{k,m}(x) := x\phi^k(x)\phi^{2k}(x)\cdots \phi^{(m-1)k}(x).\]
Let $g \in G$ be such that $\chi(g) = k > 0$. Then, $g = xt^k$ for some $x \in H$.
Let $G_i = G/ N_i$, and denote the natural quotient by \[ \begin{split}\widebar{\cdot}\ \colon G &\to G_i \\ z &\mapsto \bar{z}. \end{split}\]
Note that $G_i \cong  H / K \rtimes \Z / (iL) \Z$. Let $l = iL / \gcd(iL, k)$. Let $m > 0$ and write $m = ql + r$ for some $0 \leq r < l$ and $q \in \Z$. Then, 
\[ \bar{g}^m = \overline{\Phi_{k,l}(x)}^q \cdot \overline{\Phi_{k,r}(x)} \cdot \overline{t}^r.\]
It follows that the order of $\bar{g}$ is given by \[ \mathrm{ord}_{G_i}(\bar{g}) = \mathrm{lcm}\{ \mathrm{ord}_{G_i}(\overline{\Phi_{k, l}(x)}), l \}.\]
However, the order of $\overline{\Phi_{k, l}(x)}$ in $G_i$ is the order of $\overline{\Phi_{k, l}(x)}$ in $H/K$, and thus $\mathrm{ord}_{G_i}(\overline{\Phi_{k, l}(x)})$ divides $[H:K]$. By definition, $[H:K]$ divides $L$, and since $i$ is a multiple of $k$, $L$ divides $l$. It follows that $\mathrm{ord}_{G_i}(\overline{\Phi_{k, l}(x)})$ divides $l$ and thus $\mathrm{ord}_{G_i}(\widebar{g}) = l$. 
\end{proof}

\begin{corollary}\label{cor:cyclic-graph-fillings}
      Let $\calg = (X, \calg_{\bullet}, \iota_{\bullet})$ be a graph of groups with infinite cyclic edge groups. Suppose that there exists an epimorphism $\chi \colon \pi_1(\calg) \to \Z$ with finitely generated kernel such that $\chi(\calg_e) = \Z$ for every edge $e \in E(X)$. Let $N_v \leq_f \calg_v$ be a subgroup of finite index for each $v \in V(X)$. Then there exists a positive integer $L$ such that for every positive integer $i$, there is a virtually free filling $\calg \to \overline{\calg}$ associated to $\{N_v\}_{v \in V(X)}$ such that the order of each edge group is $iL$.
\end{corollary}

\begin{proof}

For each vertex $v \in V(X)$, let $\chi_v \colon \calg_v \to \Z$ be the restriction of $\chi$ to $\calg_v$. Since $\chi$ is surjective on edge groups, each restriction $\chi_v$ is surjective. Let $L_v > 0$ be the constant from \cref{lemma:prescribed-orders-fillings} for the finite index subgroup $N_v \leq \calg_v$ and $k = 1$. Let $L = \lcm\{L_v\}_{v\in V(X)}$. Fix integer $i > 0$. Then, for every $v \in V(X)$, let $N'_v \trianglelefteq_f \calg_v$ be the finite index normal subgroup contained in $N_v$ such that for every $g \in \calg_v$ with $\chi_v(g) = 1$, we have that the order of $g$ in the quotient $\calg_v / N_v'$ is $iL$, which exists by \cref{lemma:prescribed-orders-fillings}. In particular, the generators of the images of the incident edge groups have order $iL$. Now we may construct a vertex filling of $\calg$ by replacing each vertex group $\calg_v$ with $\calg_{v} / N_v'$.
\end{proof}

We now specialise to the setting of splittings of \{finitely generated free\}-by-cyclic groups with unipotent and superlinearly growing monodromy.  

\begin{prop}\label{prop:properties-of-std-splitting}
 Let $G$ be a \{finitely generated free\}-by-cyclic group with unipotent superlinearly growing monodromy and let $G \cong \pi_1(\calg)$ be the standard splitting. Then the following properties are satisfied. 
 \begin{enumerate}
     \item The profinite topology on $\pi_1(\calg)$ is efficient.
     \item The edge groups are separable in the vertex groups. 
     \item The edge groups are root-closed in $\pi_1(\calg)$.
     \item $\calg$ admits arbitrarily deep virtually free vertex fillings.    
 \end{enumerate}
\end{prop}

\begin{proof}
    The group $G$ is residually finite by \cite{Baumslag1971} and every vertex and edge group is fully separable by \cite[Proposition~2.9]{HughesKudlinska2023}. By \cref{lemma:Reid} it follows that $G$ induces the full profinite topology on every edge and vertex group and thus the profinite topology on $\pi_1(\calg)$ is efficient, proving (1).

    All infinite cyclic subgroups are separable in \{finitely generated free\}-by-cyclic groups by \cite[Proposition~2.9]{HughesKudlinska2023} and thus the edge groups are separable in the vertex groups, proving (2). 

    The edge groups are root-closed since they are maximal cyclic subgroups of $G$. The latter follows from the fact that the epimorphism $\chi \colon G \to \Z$ is surjective on the edge groups.  This proves (3). 

    Finally, $\calg$ admits arbitrarily deep virtually free vertex fillings by \cref{cor:cyclic-graph-fillings}, proving (4).
\end{proof}

\subsection{Conjugacy separability}

In this subsection we will utilise virtually free vertex fillings to show that certain pairs of elements can be conjugacy distinguished in finite quotients. 

\begin{lemma}\label{lemma:translation-length-preserve}
    Let $\calg = (X, \calg_{\bullet}, \iota_{\bullet})$ be a graph of groups such that the edge groups are separable in the vertex groups.  Suppose that $\calg$ admits arbitrarily deep virtually free vertex fillings. Let $g_1, g_2 \in \pi_1(\calg, v)$. Then there exists a virtually free vertex filling $\calg \to \overline{\calg}$ inducing a morphism of the associated Bass--Serre trees $T \to \overline{T}$ such that $\ell_T(g_1^k) = \ell_{\overline{T}}(\bar{g}_1^k)$ and $\ell_T(g_2^k) = \ell_{\overline{T}}(\bar{g}_2^k)$ for all $k \in \Z$.  
\end{lemma}

\begin{proof}
Note that the image of an elliptic element is elliptic in any vertex filling. Thus we may assume that $g_1$ and $g_2$ are loxodromic. 

Let $G = \pi_1(\calg, v)$. Let $\alpha_i$ be the axis of $g_i$ in $T$ for $i = 1,2$. Let $D_i \subseteq \alpha_i$ be a minimal closed connected fundamental domain for the action of $G$ on $\alpha_i$ and note that $D_i$ is isometric to the closed interval $[1,n_i]$ for some $n_i$. Let $\widetilde{D}_i\subset \alpha_i$ be obtained from the subinterval $D_i$ by taking the union with an edge on either side of $D_i$. Then, for any pair of edges $f_1$ and $f_2$ on $\alpha_i$ with $i(f_1) = i(f_2)$, there exist edges $e_1$ and $e_2$ on $\widetilde{D}_i$ and some $g\in G$ such that $f_1 = g \cdot e_1$ and $f_2 = g \cdot e_2$. 

Fix a vertex $v$ in $D_1$ and let $e_1$ and $e_2$ be the distinct edges on $\alpha_1$ with $i(e_1) = v = i(e_2)$. If $e_1$ and $e_2$ are in the same $G$-orbit, let $x_v \in \calg_v$ be such that $x_v \cdot e_1 = e_2$. Since $\calg_{e_1} \leq \calg_v$ is separable, we may pick a finite index normal subgroup $N_v \trianglelefteq_f \calg_v$ such that $x_v \not \in N_v \cdot \calg_{e_1}$. If $e_1$ and $e_2$ are not in the same $G$-orbit, take $N_v = \calg_v$. Do the same for the vertices in the interior of $D_2$, obtaining a collection of finite index normal subgroups $\{M_v\}_{v \in V(D_2)}$. Finally, extend each collection  $\{N_v\}_{v \in V(D_1)}$ and $\{M_v\}_{v \in V(D_2)}$ to equivariant families $\{N_v\}_{v \in V(T)}$ and $\{M_v\}_{v \in V(T)}$, in the obvious way. That is, for any $w \in V(T)$ with $g\cdot v = w$ for some $v \in V(D_1)$, we set $N_w = g\cdot N_v \cdot g^{-1}$. If $w \in V(T)$ is not in the $G$-orbit of any vertex in $D_1$, then set $N_w = \calg_w$. Carry out the analogous procedure for the collection $\{M_v\}_{v \in V(T)}$.

Finally, set $L_v := N_v \cap M_v$ for each $v \in V(T)$. Then, $\{L_v\}_{v\in V(T)}$ is an equivariant collection of finite index normal subgroups $L_v \trianglelefteq_f \calg_v$. Construct a vertex filling $\calg \to \overline{\calg}$ corresponding to the system $\{L_v\}_{v \in V(X)}$ and let $T \to \overline{T}$ be the associated morphism of Bass--Serre trees. 

We claim that $\alpha_1$ and $\alpha_2$ are mapped injectively onto their image via the quotient $T \to \overline{T} = \langle L_v \rangle_{v\in V(T)} \backslash T$. Since $T \to \overline{T}$ maps edges to edges, and since $\overline{T}$ is a tree, it must be the case that if two edges of $\alpha_i$ are identified under the map $T \to \overline{T}$, then there exist two edges $f_1$ and $f_2$ in $\alpha_i$ with $i(f_1) = i(f_2)$ that are identified. By equivariance, we may assume such edges to be contained in $\widetilde{D}_i$.. However, our construction of the subgroups $L_v \leq N_v \cap M_v$ precludes this. 

Now let $n_i = \ell_T(g_i) \in \NN$ and let $\bar{\alpha}_i$ be the image of the axis $\alpha_i$ under the vertex filling morphism $T \to \overline{T}$. Then $\bar{g}_i$ acts as translation by $n_i$ on the line $\bar{\alpha}_i$. Hence $\bar{g}_i$ is loxodromic and $\ell_{\overline{T}}(\bar{g}_i) = n_i$. It follows that for all $k \in \Z$, 
\[\ell_{\overline{T}}(\bar{g}_i^{k}) = |k| \cdot \ell_{\overline{T}}(\bar{g}_i)  = |k| \cdot n_i = \ell_{T}(g_i^k).\qedhere\]\end{proof}

A direct consequence of \cref{lemma:translation-length-preserve} is the following:

\begin{lemma}\label{lemma:mixed-loxodromic-conjugacy}
Let $\calg = (X, \calg_{\bullet}, \iota_{\bullet})$ be a graph of groups such that the edge groups are separable in the vertex groups. Suppose that $\calg$ admits arbitrarily deep virtually free vertex fillings. Let $G = \pi_1(\calg, v)$ and let  $g_1, g_2 \in G$ be two elements with distinct translation lengths. Then there exists a finite quotient of $G$ such that the images of $g_1$ and $g_2$ are not conjugate. 
\end{lemma}

\begin{lemma}\label{lemma:elliptic-conjugacy}
Let $\calg = (X, \calg_{\bullet}, \iota_{\bullet})$ be an acylindrical graph of groups that satisfies the following. 
\begin{enumerate}
    \item Vertex groups are conjugacy separable.
    \item Edge groups are conjugacy distinguished in vertex groups. 
    \item $\calg$ admits arbitrarily deep virtually free vertex fillings. 
\end{enumerate}

 Let $G = \pi_1(\calg, v)$. Then for any pair $g_1, g_2 \in G$ of non-conjugate elements which are elliptic, there exists a virtually free quotient $G \onto \overline{G}$ such that the image of $g_1$ is not conjugate to the image of $g_2$. Hence, there exists a finite quotient of $G$ such that the image of $g_1$ is not conjugate to the image of $g_2$.
\end{lemma}

    \begin{proof}
    Let $T$ be the Bass--Serre tree corresponding to the splitting $G = \pi_1(\calg, v)$. 

   Using the fact that edge groups are conjugacy distinguished in the vertex groups, and that $\Fix(g_1)$ and $\Fix(g_2)$ have bounded diameter by acylindricity, we may construct a virtually free vertex filling $\pi \colon T \to \overline{T}$ such that $\pi(\Fix(g_1)) = \Fix(\bar{g}_1)$ and $\pi(\Fix(g_2)) = \Fix(\bar{g}_2)$. Note also that for any vertex fillings, two edges $e, e'$ incident at the same vertex are identified by $\pi$ if only if they are contained in the same $G$-orbit by \cref{lem:orbits-in-virtually free-fillings}.
    
    Suppose first that $\mathrm{Fix}(g_1)$ consists of a single vertex $v$. Then, if $\Fix(g_2)$ contains an edge, it follows that $\Fix(\bar{g}_1)$ and $\Fix(\bar{g}_2)$ are not isomorphic as graphs and thus $\bar{g}_1$ and $\bar{g}_2$ are not conjugate.  
    
    Hence, suppose that $\Fix(g_2)$ consists of a single vertex $w$. If $v$ and $w$ are not in the same $G$-orbit in $T$, then $\bar{v}$ and $\bar{w}$ are not in the same $\bar{G}$-orbit and thus $\bar{g}_1$ and $\bar{g}_2$ are not conjugate in $\bar{G}$. Hence, suppose that $v$ and $w$ are in the same orbit. Then, there exists $x \in G$ such that $g_2' := g_2^x$ fixes $v$. Now $g_2'$ is not conjugate to $g_1$ and $g_1, g_2' \in \calg_v$. Hence, by conjugacy separability of $\calg_v$, there is a finite quotient $q \colon \calg_v \to Q_v$ such that $q(g_1)$ and $q(g_2')$ are not conjugate in $Q_v$. Hence there is a virtually free vertex filling $\calg \to \bar{\calg}$ such that $\Fix(\bar{g}_1) = \Fix(\bar{g}'_2) = \{\bar{v}\}$ and $\bar{g}_1$ and $\bar{g}'_2$ are not conjugate in $\bar{\calg}_v = \mathrm{stab}_{\bar{G}}(\bar{v}).$ But if there exists some $\bar{y} \in \bar{G}$ such that $\bar{g}_2'^{\bar{y}} = \bar{g}_1$, then $\bar{y}\cdot \Fix(\bar{g}_1) = \Fix(\bar{g}'_2)$. Hence, $\bar{y} \in \mathrm{stab}_{\bar{G}}(\bar{v})$. It follows that $\bar{g}_1$ and $\bar{g}'_2$ are not conjugate in $\bar{G}$. Thus $\bar{g}_1$ and $\bar{g}_2$ are not conjugate.
    
    Assume now that $\Fix(g_1)$ and $\Fix(g_2)$ both contain an edge. If no pair of edges $e_1 \in \Fix(g_1)$ and $e_2 \in \Fix(g_2)$ are in the same $G$-orbit, then their images in a virtually free filling are not in the same $\bar{G}$-orbit and thus $\bar{g}_1$ and $\bar{g}_2$ are not conjugate. Hence, there exists a pair $e_1$ and $e_2$ that are in the same $G$-orbit. Replacing $g_2$ by a conjugate, we may assume that there exist some edge $e$ that is fixed by $g_1$ and $g_2$. If $\Fix(g_1) = \Fix(g_2)$ consists of a single edge, then as before, we may use conjugacy separability of the vertex groups to pass to a virtually free quotient where $\bar{g}_1$ and $\bar{g}_2$ are not conjugate.

    Suppose finally that $\Fix(g_1)$ contains at least two edges. Since $\calg$ is acylindrical, each $\Fix(g_i)$ is a subtree of $T$ of finite diameter. In particular, it contains finitely many non-leaf vertices. Let $u$ be an endpoint of an edge $e \in \Fix(g_1) \cap \Fix(g_2)$ that is not a leaf of $\Fix(g_1)$.

    Consider the set $X$ of all elements $x \in G$ such that $x\cdot \Fix(g_1) = \Fix(g_2)$. Since $G$ acts by isometries, each $x \in X$ maps $u$ to a non-leaf vertex of $\Fix(g_2)$. Let $\{v_1, \ldots, v_k\}$ be the non-leaf vertices of $\Fix(g_2)$. Let $X_i \subset X$ be the set of all $x \in X$ such that $x \cdot u = v_i$. Then the $X_i$ partition $X$ into finitely many subsets, and for any two $x, y \in X_i$, we must have that $y^{-1} x \in \calg_u$. Hence, $X$ consists of finitely many $\calg_u$-orbits. For each $i$, pick some $x_i \in X_i$. Then $x_i^{-1} g_2 x_i \in \calg_u$ for every $i$. Since $\calg_u$ is conjugacy separable, we may pass to a further virtually free filling such that $x_i^{-1} g_2 x_i$ and $g_1$ are not conjugate in $\bar{\calg}_u$ for every $i$. Then, by the same argument as above, we conclude that $\bar{g}_1$ and $\bar{g}_2$ are not conjugate.   

    The final claim follows from the fact that finitely generated virtually free groups are conjugacy separable by \cref{virtually free-conj-dist}.
    \end{proof}

\subsection{Conjugacy distinguished cyclic subgroups}

We mimic the proof of \cref{lemma:elliptic-conjugacy} to construct virtually free fillings that separate the conjugacy class of an element $g_2$ from the cyclic subgroup $\langle g_1 \rangle$ in the case that $g_1$ and $g_2$ are both elliptic. 

For the remainder of this article we will use the notation $\calg \to \overline{\calg}$ to denote a vertex filling, $T \to \overline{T}$ the corresponding morphism of the associated Bass--Serre trees, and $G \to \overline{G}, g \mapsto \bar{g}$ the quotient of the correpsoninng fundamental groups.

\begin{lemma}\label{lemma:conjugacy-distinguished-elliptic}
        Let $\calg = (X, \calg_{\bullet}, \iota_{\bullet})$ be an acylindrical graph of groups that satisfies the following. 
\begin{enumerate}
    \item Vertex groups are conjugacy separable.
    \item Vertex groups are $\Z$-conjugacy distinguished. 
    \item Edge groups are conjugacy distinguished in the vertex groups. 
    \item Edge groups are root-closed in $\pi_1(\calg)$. 
    \item  $\calg$ admits arbitrarily deep virtually free vertex fillings. 
\end{enumerate}

Let $G = \pi_1(\calg, v)$. Then, for any elements $g_1, g_2 \in G$ such that at least one of $g_1$ or $g_2$ is elliptic and $g_2$ is not conjugate into $\langle g_1 \rangle$, there exists a virtually free quotient of $G$ such that the image of $g_2$ is not conjugate into the image of $\langle g_1 \rangle$. Hence, there exists a finite quotient of $G$ such that the image of $g_2$ is not conjugate into the image of $\langle g_1 \rangle$. \end{lemma}

\begin{proof}
    Let $g_1, g_2 \in G$ be as in the statement of the lemma. Let $T$ be the Bass--Serre tree corresponding to $\calg$. 

    Suppose first that $g_1$ is loxodromic and $g_2$ is elliptic. By \cref{lemma:translation-length-preserve}, there exists a virtually free vertex filling $T \to \overline{T}$ such that $\bar{g}_1$ is loxodromic. Moreover, $g_2^k$ remains elliptic in any virtually free filling. Hence $g_1$ is not conjugate into $\langle g_2 \rangle$ in the image. The same argument holds if $g_1$ is elliptic and $g_2$ is loxodromic.

    Suppose now that both $g_1$ and $g_2$ are elliptic. We begin by noting that for any $k,l \in \Z \setminus 0$,\[ \Fix(g_1^{k}) \subseteq \Fix(g_1^{kl}),\] and since the action of $G$ on $T$ is acylindrical, the diameter of $\Fix(g_1^k)$ is uniformly bounded over all $k \in \Z \setminus 0$. Using the fact that edge groups are conjugacy distinguished in the vertex groups, and that $\Fix(g_1)$ and $\Fix(g_2)$ have bounded diameter by acylindricity, we may pass to a virtually free vertex filling $\pi \colon T \to \overline{T}$ such that $\pi(\Fix(g_1^k)) = \Fix(\bar{g}^k_1)$ for all $k \in \Z \setminus 0$, and $\pi(\Fix(g_2)) = \Fix(\bar{g}_2)$.
    
    Suppose now that $\mathrm{Fix}(g_2)$ consists of a single vertex $v$. If there is no power $g_1^k$ of $g_1$ such that $\Fix(g_1^k)$ consists of a single vertex, then $\Fix(\bar{g}_1^k)$ is not isomorphic to $\Fix(\bar{g}_2)$ for any $k \in \NN$ and thus $\bar{g}_2$ is not conjugate into $\langle \bar{g}_1 \rangle$.

    If there is a power $g_1^k$ of $g_1$ such that $\Fix(g_1^k)$ consists of a single vertex, then we must have that $\Fix(g_1)$ consists of a single vertex, since $\Fix(g_1) \subseteq \Fix(g_1^k)$ for all $k \in \Z$. Let $\Fix(g_1) = \{w\}$.  If $v$ and $w$ are not in the same $G$-orbit in $T$, then $\bar{v}$ and $\bar{w}$ are not in the same $\bar{G}$-orbit and thus $\bar{g}_2$ and $\bar{g}_1^k$ are not conjugate in $\bar{G}$ for all $k \in \Z$.  
    
    Hence, suppose that $v$ and $w$ are in the same orbit. Then, there exists $x \in G$ such that $g_2' = g_2^x$ fixes $v$. Now $g_2'$ is not conjugate to $g_1^k$ for all $k \in \NN$, and $g_1^k, g_2' \in \calg_v$. Then, since $\langle g_1 \rangle$ is conjugacy distinguished in $\calg_v$, there is a finite quotient $q \colon \calg_v \to Q_v$ such that $q(g_1^k)$ and $q(g_2')$ are not conjugate in $Q_v$, for all $k \in \Z$. Hence there is a virtually free vertex filling $\calg \to \bar{\calg}$ such that $\Fix(\bar{g}^k_1) = \Fix(\bar{g}'_2) = \{\bar{v}\}$ and $\bar{g}^k_1$ and $\bar{g}'_2$ are not conjugate in $\bar{\calg}_v = \mathrm{stab}_{\bar{G}}(\bar{v}).$ But if there exists some $\bar{y} \in \bar{G}$ such that $\bar{g}_2'^{\bar{y}} = \bar{g}_1^k$, then $\bar{y}\cdot \Fix(\bar{g}^k_1) = \Fix(\bar{g}'_2)$. Hence, $\bar{y} \in \mathrm{stab}_{\bar{G}}(\bar{v})$. It follows that $\bar{g}^k_1$ and $\bar{g}'_2$ are not conjugate in $\bar{G}$. Thus $\bar{g}^k_1$ and $\bar{g}_2$ are not conjugate. 
    
    Assume now that $\Fix(g_2)$ contains an edge and there exists $k \in \Z \setminus 0$ such that $\Fix(g_1^k)$ also contains an edge. If no pair of edges $e_1 \in \Fix(g_2)$ and $e_2 \in \Fix(g_1^k)$, for any $k \in \Z \setminus 0$, are in the same $G$-orbit, then their images in a virtually free filling are not in the same $\bar{G}$-orbit and thus $\bar{g}^k_1$ and $\bar{g}_2$ are not conjugate. Hence, there exists a pair $e_1$ and $e_2$ that are in the same $G$-orbit. Replacing $g_2$ by a conjugate, we may assume that there exist some edge $e$ that is fixed by $g_2$ and $g_1^k$ for some $k \in \Z \setminus 0$. 
    
    Suppose that $\Fix(g_2)$ only contains the edge $e$. If $\Fix(g_1)$ also consists of a single edge, then for every $k \in \Z \setminus 0$ such that $\Fix(g^k_2)$ consists of an edge, we must have that $\Fix(g_1^k) = \Fix(g_2) = \{e\}$. Let $u$ be an endpoint of $e$. Then $u$ is fixed by $g_2$ and $g_1^k$ for every $k \in \Z$. Hence, we may use the conjugacy distinguished property of the vertex group $\calg_u$ to pass to a virtually free quotient such that $\bar{g}_2$ is not conjugate into $\langle \bar{g}_1 \rangle$.

     Finally, suppose that $\Fix(g_2)$ contains at least two edges. Since $\calg$ is acylindrical, each $\Fix(g_2)$ is a subtree of $T$ of finite diameter. In particular, it contains finitely many non-leaf vertices. Let $u$ be an endpoint of an edge $e \in \Fix(g_1) \cap \Fix(g_2)$ that is not a leaf of $\Fix(g_1)$.

     We will now use the root-closed assumption. Since $G$ is root-closed, $\Fix(g_1^k) = \Fix(g_1)$ for all $k \in \Z \setminus 0$. Now we argue as in the last part of \cref{lemma:elliptic-conjugacy}. Consider the set $X$ of all elements $x \in G$ such that $x\cdot \Fix(g_1) = \Fix(g_2)$. Then each $x \in X$ maps $u$ to a non-leaf vertex of $\Fix(g_2)$. Let $\{v_1, \ldots, v_k\}$ be the non-leaf vertices of $\Fix(g_2)$. Let $X_i$ be the subset of $x \in X$ such that $x \cdot u = v_i$. Then the $X_i$ partition $X$ into finitely many subsets, and for any two $x, y \in X_i$, we must have that $y^{-1} x \in \widebar{\calg}_u$. Hence, $X$ consists of finitely many $\widebar{\calg}_u$ orbits. For each $i$, pick some $x_i \in X_i$. Now $x_i^{-1} g_2 x_i, g_1^k \in \calg_u$ for every $i$ and $k \in \Z \setminus 0$. Since $\calg_u$ is conjugacy distinguished, we may pass to a further virtually free filling such that the image of $x_i^{-1} g_2 x_i$ is not conjugate into the image of $\langle g_1 \rangle$.

     The final claim follows from the fact that cyclic subgroups are conjugacy distinguished in virtually free groups \cite{ChagasZalesskii2015}.
    \end{proof}

\begin{lemma}\label{lemma:conj-distinguished-loxodromics}
        Let $\calg = (X, \calg_{\bullet}, \iota_{\bullet})$ be an acylindrical graph of groups such that the edge groups are separable in the vertex groups. Suppose that $\calg$ admits arbitrarily deep virtually free vertex fillings. Suppose also that hyperbolic elements of $G = \pi_1(\calg, v)$ are conjugacy distinguished.
        
        Then, for any hyperbolic elements $g_1$ and $g_1$ such that $g_2$ is not conjugate into $\langle g_1 \rangle$, there exists a virtually free quotient of $G$ such that the image of $g_2$ is not conjugate into the image of $\langle g_1 \rangle$. Hence, there exists a finite quotient of $G$ such that the image of $g_2$ is not conjugate into the image of $\langle g_1 \rangle$. \end{lemma}

\begin{proof}
    By \cref{lemma:translation-length-preserve}, we may pass to a virtually free vertex filling $G \to \overline{G} = G/ N_1$ that preserves the translation lengths of $g_1$ and $g_2$. Hence, there exists at most one positive integer $k > 0$ such that \[\ell_T(g_2) = \ell_{\overline{T}}(\bar{g}_2) = \ell_{\overline{T}}(\bar{g}_1^k) = \ell_T(g_1^k).\] If no such integer exists, then $\bar{g}_2$ is not conjugate into $\langle \bar{g}_1 \rangle$ and we are done. 
    
    Suppose that there exists some $k > 0$ such that $\ell_{\overline{T}}(g_2) = \ell_{\overline{T}}(g_1^k)$. Since hyperbolic elements in $G$ are conjugacy distinguished, there exists a finite quotient $\pi \colon \pi_1(\calg) \to Q$ such that $\pi(g_2)$ is not conjugate to $\pi(g_1^{\pm k})$. Let $N_2$ be the kernel of $\pi$.

    Then, the quotient $G \to G / N_1 \cap N_2$ corresponds to a virtually free filling such that the image of $g_2$ is not conjugate into the image $\langle g_1 \rangle$. The final claim follows from the fact that cyclic subgroups are conjugacy distinguished in virtually free groups. 
\end{proof}

\section{Piecewise trivial suspensions}\label{sec:PTS}

The goal of this section is to prove that \{finitely generated free\}-by-cyclic groups with unipotent linear monodromy are conjugacy separable.  This is done in \Cref{thm:conjugacy-separable-linear-UPG}.

\subsection{Some background results}

Recall that a \{finitely generated free\}-by-cyclic group with unipotent and linearly growing monodromy admits a standard splitting as in \cref{thm:unipotent-linear-splitting}. Following the terminology established in \cite{dahmani_unipotent_2024}, we will call such a splitting a \emph{piecewise trivial suspension with free local fibres}, sometimes omitting reference to the local fibres if it is clear from the context.

A piecewise trivial suspension $\calk$ is said to be \emph{clean} if whenever $w$ is a white vertex of the underlying graph we have  $\calk_w \cong F_w\oplus\langle t_w \rangle$ and if $H \leq \calk_w$ is the image of any edge group in $\calk_w$, then $H=\langle c_H,t_w \rangle$ where $c_H$ is a primitive element of $F_w$. We say that a finite index subgroup $K\leq \pi_1(\calg,v)$ is a \emph{clean cover} if the induced splitting $\calk$ of $K$ is clean and furthermore if there is a power $k$ such that every black vertex group in $\calk$ is conjugate to a \emph{$k$-congruence subgroup}, i.e. the subgroup $k\ZZ\oplus k\ZZ\leq \ZZ^2$ of a black vertex group of $\calg$.

\begin{prop}\label{prop:piecewise-trivial-suspensions-properties}
    Let $\calg$ be a piecewise trivial suspension with free local fibres. Then the following properties are satisfied.
    \begin{enumerate}
        \item The profinite topology on $\pi_1(\calg,v)$ is efficient; in particular, the edge groups are separable in the vertex groups. 
        \item The vertex groups are conjugacy separable. 
        \item Cyclic subgroups are conjugacy distinguished in the vertex groups. 
        \item The edge groups are conjugacy distinguished in the vertex groups. 
        \item The splitting is 4-acylindrical.
    \end{enumerate}
    Moreover, $\pi_1(\calg,v)$ contains a finite index subgroup that admits a splitting as a clean piecewise trivial suspension with free fibres.
\end{prop}

\begin{proof}
   By definition of piecewise trivial suspension the splitting is 4-acylindrical, the vertex groups are of the form $F_v \times \Z$ where $F_v$ is a finitely generated free group and the edge groups are maximal $\Z^2$-subgroups. The group $G$ is residually finite by \cite{Baumslag1971}, and by \cite[Proposition 2.9]{HughesKudlinska2023} every edge and vertex group is fully separable. Hence the profinite topology on $\pi_1(\calg)$ is efficient.  
    
    Since finitely generated free groups are conjugacy separable, it follows that the vertex groups are conjugacy separable. Since finitely generated subgroups are conjugacy distinguished in finitely generated free groups, it follows that edge groups and cyclic subgroups are conjugacy distinguished in the vertex groups. Similarly, edge groups are separable in the vertex groups. 
    
By \cite[Proposition 6.14]{dahmani_unipotent_2024}, there is a finite index subgroup $G' \leq G$ that admits a splitting as a clean piecewise trivial suspension. 
\end{proof}
    
Clean piecewise trivial suspensions admit virtually free vertex fillings of a particularly nice form.

\begin{prop}[Virtually free vertex fillings ({\cite[Proposition 6.17]{dahmani_unipotent_2024})}]\label{prop:linear-vertex-fillings}
    Let $\calk$ be a clean piecewise trivial suspension. Suppose for every white vertex $w$ we are given a finite index subgroup $N_w\leq F_w$. Then there are finite index characteristic subgroups $D_w \leq N_w \leq F_W$ such that there exists a positive integer $N > 0$ and a graph of groups $\bar\calk$ and with the same underlying graph such that for every white vertex we have $\bar\calk_w \cong F_w/D_w\oplus\left( \ZZ/N\ZZ\right)$, every black vertex is isomorphic to $(\ZZ/N\ZZ)^2$ and there is a natural surjection $\pi_1(\calk,v)\onto\pi_1(\bar\calk,v)$ induced by a morphism of graphs of groups, where each vertex group morphism $\calk_v \onto \bar\calk_v$ is the quotient by $\langle D_w,t_w^N\rangle.$
\end{prop}

\begin{remark}\label{ramrk:linear-vertex-fillings}
    By strong omnipotence of the free groups $D_w$, there exists a positive integer $N' > 0$ such that for every $i \in \NN$, there is a virtually free vertex filling $\bar{\calk}$ of $\calk$ with images of white vertex groups of the form $\bar\calk_w \cong F_w/D_w\oplus\left( \ZZ/iN\ZZ\right)$ and black vertex groups $(\ZZ/iN\ZZ)^2$.
\end{remark}

\subsection{Short positions and vertex fillings} \label{sec:short-position}

Let $\calg = (X, \calg_{\bullet}, \iota_{\bullet}) $ be a piecewise trivial suspension. A \emph{turn} is a pair of edges $(e,f)\in E(X)^{\pm 1} \times E(X)^{\pm 1}$ such that $\tau(e) = \iota(f)$. Every turn defines a system of double coset representatives in $\calg_{\tau(e)}$ for the double cosets $I_e\backslash \calg_{\tau(e)} /O_f$ where $I_e$ is the image of $\tau_e \colon \calg_e \to \calg_{\tau(e)}$ and $O_f$ is the image of $\iota_f \colon \calg_f \to \calg_{\iota(f)}$. For a piecewise trivial suspension we can always assume that the double coset representatives are the trivial element in the black vertex groups and elements of the local fibre $F_v$ in the white vertex groups.

Recall that an element $g \in \bass(\calg)$ is a \emph{$\calg$-loop} if it is an element of $\pi_1(\calg,v)$ for some $v\in V(X)$. We say that a $\calg$-loop based at a white vertex is in \emph{normal form} when it is explicitly written out as
\[
    \ell_0^{l_0} \tilde{a_0} e_1e_2 \ell_2^{l_2}\tilde{a_2}e_3\cdots e_{n-1}e_n\ell_n^{l_n}\tilde{a_n} r_n^{m_n}t^k
\] where the $\tilde{a_i}$ are double coset representatives with $\tilde a_0$ and $\tilde a_n$ being representatives for the turn $(e_n,e_1)$; for each $i \geq 1$, $\ell_i \in I_{e_i}$, $\ell_0 \in I_{e_n} \cap F_{\tau(e_n)}$ and $r_n\in O_{e_1} \cap F_{\iota(e_1)}$; and $t$ is the generator of the center of $\calg_{\tau(e_n)}$.

We say a $\calg$-loop $g$ is in \emph{short position} within its $\calg$-conjugacy class if it has translation length 2 and its normal form looks like\[
e_1e_2\tilde{a_2}r_2^dt^k, 
\] or if if has translation length 4 or more and its normal form looks like\[
e_1e_2\tilde{a_2}e_3e_4 \ell_{e_4}^d\tilde{a_4}\cdots t^k
\] where $d\in \mathbb Z$ is minimal (w.r.t. some fixed well-order of $\ZZ$) among all conjugates of $g$. Note that by conjugating by the right element of the edge group corresponding to $e_1$ it is possible to fix the initial prefix $e_1e_2\tilde{a_2}$ while changing the exponent $d$. 

We say that $g$ is in \emph{almost short position} if it has a prefix of the form $e_1e_2$. Note that all conjugates in almost short position are conjugates of elements in short position by elements that lie in the edge group $\iota_{e_1}(\calg_{e_1})$. The following proposition follows by 4-acylindricity of the splitting.
\begin{prop}[{\cite[Proposition 5.6]{dahmani_unipotent_2024}}]
    An element of $\pi_1(\calg,v)$ with translation length $2m$ has at most $m$ $\bass(\calg)$-conjugates in short position.
\end{prop}

Now given a $\calg$-loop $g$, let $DCR(g)=(\tilde{a}_0, \ldots,\tilde{a}_n)$ denote the sequence of double coset representatives that occur in the normal form for $g$ if $g$ is in short position, and if $g'$ is a conjugate of $g$ by an element of an edge group that is in almost short position then $DCR(g')=DCR(g)$. As a consequence of double coset separability of free groups we have:

\begin{lemma}\label{lemma:separating-DCR}
    If for any $\bass(\calg)$ conjugates $k_1,k_2$ of $g_1,g_2$ in short position with the same underlying edge part we have $DCR(k_1)\neq DCR(k_2)$ then there is a virtually free quotient of $G$ where $g_1,g_2$ have non-conjugate image.   
\end{lemma}
\begin{proof}
    By hypothesis, we can find sufficiently deep finite index subgroups of the vertex groups so that the $k^{}\mathrm{th}$ coset representative $\tilde{b}_{2k}$ occurring in $DCR(k_2)$ is separated from the double coset $I_{e_{2k}}\tilde{a}_{2k}O_{e_{2k+1}}$. Let $\bar\calg$ be a sufficiently deep virtually free vertex filling. Then all almost short conjugates of the image of $g_1$ will be different from the image of all almost short conjugates of the image of $g_2$ in $\bass(\bar\calg)$. Since any element has almost short conjugates, it follows that $g_1$ and $g_2$ have non-conjugate images.
    \end{proof}

\subsection{Separating elements with the same double coset sequence}
    Suppose now that we are given two non-conjugate elements in short position with normal forms
    \begin{eqnarray}
        g &=& e_1e_2a_2e_3e_4c_4^{n_4}a_4e_5\cdots e_lc_l^{n_l}a_l\bar c_1^{n_1}t_1^r \label{eq:g}\\
        h &=& e_1e_2a_2e_3e_4c_4^{m_4}a_4e_5\cdots e_lc_l^{m_l}a_l\bar c_1^{m_1}t_1^r
        \label{eq:h}
    \end{eqnarray}
    where we allow $e_i=e_j^{\pm 1}$ even if $i\neq j$ and where the $c_j$ are elements that generate the intersection of the fibre and the image of the edge group. We have the following migration relations:
    \begin{eqnarray}
        \bar c_i e_ie_{i+1} &=& e_ie_{i+1}c_{i+1}\\
        t_i  e_ie_{i+1} &=& e_ie_{i+1} c_{i+1}^{\epsilon_{i+1}}t_{i+1}, \label{eqn:def-twist-parameter}
    \end{eqnarray}
    where $t_i$ is the generator of the center of either $\calg_{\tau(e_i)}$ or $\calg_{\iota(e_{i})}$; which ever happens to be a non-abelian vertex group. The exponents $\epsilon_j$ that occur in \eqref{eqn:def-twist-parameter} are called \emph{twisting numbers} since these are the exponents that occur in the Dehn multitwists.

    Let us now look at the effect of conjugating $g$ from \eqref{eq:g} by $t_1$. The sequence of equalities is obtained by migrating the $t_i$ symbols to the right.
    \begin{eqnarray*}
        t_1 g t_1^{-1} & = &  t_1e_1e_2a_2e_3e_4c_4^{n_4}a_4e_5\cdots e_lc_l^{n_l}a_l\bar c_1^{n_1}t_1^{r-1}\\
        & = & e_1e_2 c_2^{\epsilon_2}t_2a_2e_3e_4c_4^{n_4}a_4e_5\cdots e_lc_l^{n_l}a_l\bar c_1^{n_1}t_1^{r-1}\\
        & = & e_1e_2 c_2^{\epsilon_2}a_2e_3e_4c_4^{n_4+\epsilon_4}t_4a_4e_5\cdots e_lc_l^{n_l}a_l\bar c_1^{n_1}t_1^{r-1}\\
        & = & \cdots
    \end{eqnarray*}
    \begin{eqnarray}
          t_1 g t_1^{-1} &= & e_1e_2 c_2^{\epsilon_2}a_2e_3e_4c_4^{n_4+\epsilon_4}a_4e_5\cdots e_lc_l^{n_l+\epsilon_l}a_l\bar c_1^{n_1}t_1^{r}\label{eqn:t-conj}
    \end{eqnarray}
    From this we see that the subgroup that preserves the prefix $e_1e_2a_2e_3e_4$ of $g$ under conjugation is $\langle t \bar c_1^{-\epsilon_2}\rangle$. We also see that \begin{equation}\label{eq:almost_shor_conj}
        (t_1\bar c_1^{-\epsilon_2}) g (t_1\bar c_1^{-\epsilon_2})^{-1} = e_1e_2a_2e_3e_4 c_4^{n_4+\epsilon_4} \ldots e_l c_l^{n_l+\epsilon_l}a_l \bar c_1^{n_1+\epsilon_2}t_1^r.
    \end{equation} 

    Up to this point we did not specify the ordering on $\ZZ$ we wanted for short position. For our purposes we will now require our well-ordering on $\ZZ$ to be any well-order that extends\[
    0 < 1 <\cdots < |\epsilon_4|.
    \]
    \begin{lemma}\label{lem:short-position-small-exp}
        If $g,h$ as given in \eqref{eq:g}, \eqref{eq:h}, respectively, are in short position then we have\[
            0\leq n_4,m_4 <|\epsilon_4|.
            \]
    \end{lemma}
 
    \begin{proof}
        Suppose towards a contradiction that $g$ is in short position but either $n_4<0$ or $n_4>\epsilon_4$. This means that by the division algorithm there exists $q\in \ZZ$ such that \[
            n_4 = q|\epsilon_4|+r
        \] with $0\leq r <|\epsilon_4|.$ By equation \ref{eq:almost_shor_conj}, by repeatedly conjugating $g$ by $(t_1\bar c_1^{\epsilon_2})^{\pm 1}$, we can add integer multiples of $\epsilon_4$ to he exponent in the $c_4$ position while preserving the $e_1e_2a_2e_3e_4$-prefix. By the division algorithm we can therefore make this exponent the remainder of division, contradicting that $g$ was in short position. The same argument works for $h$
    \end{proof}
    
    We can now make sense of and then justify the following definition.
    
    \begin{defn}\label{def:persistent_vector}
        If $(g,h)$ is the pair of elements given by \eqref{eq:g} and \eqref{eq:h} then we define the \emph{persistent vector} $\vec\epsilon$ and the \emph{difference vector} $\vec \delta$ to be\[
        \vec \epsilon = \begin{bmatrix}
            \epsilon_4\\ \epsilon_6 \\ \vdots \\\epsilon_l \\ \epsilon_2
        \end{bmatrix}
         \textrm{~and~}  \vec \delta = \begin{bmatrix}
             m_4-n_4\\m_6-n_6 \\\vdots \\ m_l-n_l\\m_1-n_1
         \end{bmatrix}.
        \] We also define the \emph{persistent equation} to be\[
            \vec\delta = \lambda \vec\epsilon.
        \]
    \end{defn}
    The reason why these vectors are important is that firstly, if the persistent equation does in fact have a solution $\lambda_0\in \ZZ$ then by \eqref{eq:almost_shor_conj} we would have $(t_1\bar c_1^{-\epsilon_2})^{\lambda_0} g (t_1\bar c_1^{-\epsilon_2})^{-\lambda_0} = h$. That said, we also have the following.
            
    On the one hand, 4-acylindricity implies that all the $\epsilon_i$ are nonzero. Lemma \ref{lem:short-position-small-exp} implies $|m_4-n_4|<|\epsilon_4|$, so if $|m_4-n_4|\neq 0$ then there is no integer solution to the equation and if $|m_4-n_4|=0$ then the only solution is $\lambda=0$, which implies $\vec\delta=0$ and therefore that $g=h$. 

    We now wish to pass to a virtually free vertex filling $\bar\calg$ in which the images of $g,h$ remain non-conjugate. To simplify notation we will continue to denote the images of $g,h$ in $\bass(\bar\calg)$ as \eqref{eq:g} and \eqref{eq:h}, but assuming that the non-$E(X)$ symbols denote elements in the finite, filled, vertex groups. Every abelian vertex group in $\bar\calg$ is isomorphic to $\ZZ_N\times \ZZ_N$ so we'll call $N$ the \emph{abelian exponent} of $\bar\calg$. The image of a double coset $\langle c_n\rangle a_n \langle \bar c_{n+1}\rangle$ is said to be \emph{permeable} if there are $p,q \in \ZZ$, at least one of which is not zero, such that\[
        c_n^p a_n =_{\bar\calg} a_n (\bar c_{n-1})^q.
    \] Otherwise the image of the double coset is said to be \emph{impermeable}. It is easy to see that the image of double coset is impermeable if and only it has cardinality $N^2$. 
    
    Suppose all the double cosets from $DCR(g)$ have impermeable images in a virtually free vertex filling. We will now describe potential conjugators that bring the image of $g$ to the image of $h$. We note that \eqref{eq:almost_shor_conj} still holds in $\bar\calg$ only here the exponents can be taken $\mod N$. In particular, if there is a solution $\lambda$ to the persistent equation \[
        \vec \delta = \lambda \vec \epsilon \mod N    
    \] 
    then the images of $g$ and $h$ will be conjugate. The persistent equation comes from the determining which conjugators preserve the prefix $e_1e_2a_2e_3e_4$ of $g$ or $h$ for that matter. It is worth noting that in $\calg$, the group of such elements coincides with the stabilizer of a segment of length 4 in the dual Bass--Serre tree. Equation \eqref{eqn:t-conj} also holds with exponents modulo $N$ and we immediately get that conjugation by $t_1^{N/\epsilon_2}$ also preserves the prefix $e_1e_2a_2e_3e_4$. Looking at the effect of this conjugation by repeatedly applying \eqref{eq:almost_shor_conj} motivates the following.

    \begin{defn}
        If $(g,h)$ is the pair of elements given by \eqref{eq:g} and \eqref{eq:h} then for a given $N$ the we have the \emph{content vector} and \emph{variable vector} \[
       \vec\kappa=\begin{bmatrix}
            \epsilon_4\\ \epsilon_6 \\\vdots \\\epsilon_l\\ 0
        \end{bmatrix}
        \textrm{~and~}
             \vec\nu_N = \frac{N}{\epsilon_2}\cdot\vec\kappa        
    \] respectively. We further define the \emph{modulo $N$ conjugacy equation} to be 
      \begin{equation}\label{eq:mod_n_conj}
        \vec \delta =\lambda_1 \vec \epsilon + \lambda_2 \vec \nu_N \mod N.
    \end{equation}
    \end{defn}
        
    Obviously, a solution $(\lambda_1,\lambda_2)$ to the modulo $N$ conjugacy equation would imply that the images of $g$ and $h$ are conjugate in $\bar\calg$. The following shows that we can arrange for this not to happen.

    \begin{prop}\label{prop:no_sol_const}
        If $(g,h)$ is the pair of elements given by \eqref{eq:g} and \eqref{eq:h} then there is some $M$ such that, if $M|N$, the modulo $N$ conjugacy equation \eqref{eq:mod_n_conj} has no solution.
    \end{prop}
    \begin{proof}
        The equation\begin{equation}\label{eqn:relaxed}
            \vec\delta = \lambda_1\vec\epsilon + \lambda_2\vec\kappa
        \end{equation} has a solution if and only if $\vec \delta$ lies in the subgroup $\langle \vec\epsilon, \vec\kappa \rangle$. The hypothesis that $g,h$ are non conjugate implies that $\delta \neq \vec 0$, so $(\lambda_1,\lambda_2)=(0,0)$ is not a solution. Lemma \ref{lem:short-position-small-exp} implies that there are also no integer solutions with $\lambda_2=0.$

        Recall that all the twisting numbers $\epsilon_i$ are non-zero. It follows that the persistent and content vectors, $\vec\epsilon$ and $\vec\kappa$, are linearly independent. So if there is a solution to \eqref{eqn:relaxed}, then this solution must be unique. This means that we can take $L$ to be divisible by $\epsilon_2$ and so large that \begin{equation}\label{eqn:relaxed2}
            \vec\delta = \lambda_1\vec\epsilon + \lambda_2 \frac{L}{\epsilon_2}\vec\kappa
        \end{equation}
        does not have a solution in $\mathbb Z^2$. It follows that we can always find some sufficiently large $L$ that is divisible by $\epsilon_2$ such that\[
        \vec \delta \not \in \left\langle \vec\epsilon,\frac{L}{\epsilon_2} \vec\kappa\right\rangle.       \] Since $\mathbb Z^{l/2}$ is subgroup separable then we can find some $M$ that is divisible by $L$ such that the image of $\vec\delta$ does not lie in the image of $\langle \vec\epsilon,\frac{L}{\epsilon_2} \vec\kappa\rangle$ in $(\mathbb Z/M\mathbb Z)^{l/2}$ and therefore neither in any $(\mathbb Z/N\mathbb Z)^{l/2}$  where $M|N \Rightarrow N = qL$ for some $q\in \mathbb Z$.

        Suppose finally towards a contradiction that the modulo $N$ conjugacy equation \eqref{eq:mod_n_conj} had a solution, then we have
        \begin{eqnarray*}
         & \vec \delta &=\lambda_1 \vec \epsilon + \lambda_2 \vec \nu_N \mod N\\
         \Rightarrow & \vec\delta &= \lambda_1 \vec \epsilon + \lambda_2 \frac{N}{\epsilon_2}\vec\kappa \mod N\\
         \Rightarrow & \vec\delta &= \lambda_1 \vec \epsilon + (\lambda_2q) \frac{L}{\epsilon_2}\vec\kappa \mod N\\
        \end{eqnarray*}
        which implies that the image of $\vec\delta$ lies in the image of $\langle \vec\epsilon,\frac{L}{\epsilon_2} \vec\kappa\rangle$ in $(\mathbb Z/N\mathbb Z)^{l/2}$, which is a contradiction and the result follows.
    \end{proof}

    \begin{prop}\label{prop:big_N_separate}
        Let $g,h$ be as in \eqref{eq:g},\eqref{eq:h} and suppose they are non-conjugate in $\calg$. Let $g_1,\ldots,g_c$ be all the conjugates of $g$ that are in short position and where $DCR(g_i)=DCR(g), i = 1,\ldots,c$. If we can find a vertex filling $\bar\calg$ such that all double coset images are impermeable, distinct double cosets have distinct images in their respective vertex groups, and where the abelian exponent $N$ is such that the modulo $N$ conjugacy equation\[
            \vec \delta_i = \lambda_1 \vec\epsilon + \lambda_2 \vec\nu_N \mod N
            \] has no solution, where $\delta_i$ is the difference vector for $(g_i,h)$, then the images of $g,h$ are not conjugate in $\bar\calg$.
    \end{prop}
    \begin{proof}
        Consider the image of $g$ as given in \eqref{eq:g} in $\bar\calg$. By impermeability distinct vectors of exponents $[n_4,\ldots,n_l,n_1]^\top$ $\mod N$ will give distinct elements. Suppose towards a contradiction that the images of $g$ and $h$ were conjugate in $\bar\calg$. Let $k$ be the conjugator that gives $kgk^{-1} = _{\bar\calg} h$. If $k$ is hyperbolic then, it can be decomposed as $k = ep$ where $p$ cyclically permutes syllables and $e$ is elliptic. It follows that there will be some elliptic $k$ such that $(k')g_i(k')^{-1}=_{\bar\calg}h$ for some $1\leq i \leq c$.

        Equation \eqref{eqn:t-conj} implies that $k'$ must lie in the image of the subgroup $\langle t\bar c_1^{-\epsilon_2}, t^{N/\epsilon_2}\rangle
        $ and in particular that \eqref{eq:mod_n_conj}, the modulo $N$ conjugacy equation, must hold for $\delta_i$ on the left hand side, which is impossible by hypothesis. The result follows.
    \end{proof}

\subsection{Designer vertex fillings}
Let $g_1,g_2$ be two hyperbolic elements in $G$. Our goal is to find a finite quotient of $G$ where the images $\bar{g}_1$ and $\bar{g}_2$ are not conjugate. We will do this by finding a virtually free quotient of a finite index subgroup of $G$ in which powers of images of so-called elevations of $g_1$ and $g_2$ are non-conjugate. The finite index subgroup and virtually free quotients we will construct will work specifically for the pair $g_1,g_2$. 

We write $H \findex G$ to indicate that $H$ is a finite index subgroup of $G$.  If $H \findex G$ then there is some power $e(H)$ such that $g^{e(H)} \in H$ for all $g \in G$.  The following result is a consequence of the proof of \cite[Lemma 3.1]{cotton-barratt_conjugacy_2012}, we reproduce the relevant part here for completeness.

\begin{lemma}\label{lem:designer-subgroup}
Let $g_1,g_2$ be elements of a group $G$. Suppose there exists $\gamma \in \widehat{G}$ such that $g_1^\gamma = g_2$, i.e. $g_1,g_2$ are conjugate in the profinite completion of $G$. Let $H \findex G$. Then there exists some $d\in G$ and some $\gamma' \in \hat H$ such that $(g_1^{e(H)})^{\gamma'} = (g_2^{e(H)})^d$.
\end{lemma}
\begin{proof}
    Since $\widehat G=\widehat H G$ we may write $\gamma=\gamma'd^{-1}$ with $\gamma'\in\widehat H$ and $d\in G$.  We have $g_2^{e(H)}=\left(g_1^{e(H)}\right)^\gamma=\left(g_1^{e(H)}\right)^{\gamma'd^{-1}}$, that is $\left(g_1^{e(H)}\right)^d=\left(g_2^{e(H)}\right)^{\gamma'}$.
\end{proof}

Our goal will therefore be, given nonconjugate elements $g_1,g_2 \in G$, to construct some $H\findex G$ such that for any $G$-conjugate of $g_2^{e(H)}$ lying in $H$ there will be some finite quotient of $H$ such that its image will be nonconjugate to the image of $g_1^{e(H)}$. Lemma \ref{lem:designer-subgroup} will then imply that $g_1$ and $g_2$ have nonconjugates image in some finite quotient of $G$.

An important tool to achieve our goal is Shepherd's strong commanding property, which is a generalization of Wise's omnipotence.

\begin{defn}[{Commanding group elements \cite[Definition 1.1]{shepherd_imitator_2023} (c.f. \cite[Definition 3.2]{wise_subgroup_2000})}]\label{defn:command}
A group $G$ \emph{commands} a set of elements $\{g_1,\ldots, g_n\} \subset G$ if there exists an integer $N > 0$ such that for any integers $r_1,\ldots, r_n > 0$ there exists a homomorphism to a finite group $G \to \bar G,  g \mapsto \bar g$ such that the order of  $\bar g_i$ is $N r_i$. If this can always be done with $\langle \bar g_i\rangle \cap \langle \bar g_j\rangle = \{1\}$ for all $i \neq j$ then we say that $G$ \emph{strongly commands} $\{g_1,\ldots, g_n\}$.
\end{defn}

Typically, command (or omnipotence) has been used to construct finite degree covering spaces of graphs of spaces - we will be using command in this way and we will also use command in a somewhat novel way to construct vertex fillings. The following lemma indicates how the notion of \emph{strong command} will be useful to construct impermeable double cosets.

\begin{lemma}\label{lem:strong_command_impermeable}
  Let $H=\langle h \rangle$ and $K=\langle k \rangle$ be non-trivial cyclic subgroups of a group $F$. If $H \cap gKg^{-1} = \{1\}$ then the double coset $Hg K$ is impermeable.
\end{lemma}

\begin{proof}
  We will prove the contrapositive, to this end suppose $HgK$ is not impermeable. Then there are distinct pairs of integers $(n_1,m_1)\neq (n_2,m_2)$ such that\[
    h^{n_1}gk^{m_1} =     h^{n_2}gk^{m_2}.
  \] This immediately implies \[
    h^{n_1-n_2}=gk^{m_2-m_1}g^{-1}
  \] and since either $n_1=n_2$ or $m_2-m_1$ is non-zero, both must be non-zero so $H \cap gKg^{-1} \neq \{1\}$ and the result follows.
\end{proof}

Let us now first focus on constructing a finite index subgroup that will enable us to do the necessary vertex fillings. This finite index subgroup will be constructed from a finite degree covering space.

Let $X$ be a topological space and let $\ell: S^1 \imm X$ be a loop (i.e. an continuous immersion from a circle to the space). Let $\rho: Y \onto X$ be a finite degree covering space of $X$. An \emph{elevation} $\hat \ell: S^1 \imm Y$ of $\ell$ to $Y$ is an immersion such that the following diagram commutes:
\[
\begin{tikzcd}
    S^1 \ar[r, loop-math to, "\hat\ell"] \ar[d, -math onto, "c"] & Y \ar[d, -math onto,"\rho"]\\
    S^1 \ar[r, loop-math to, "\ell"] & X
\end{tikzcd}
\]
for some covering map $c$. The degree of $c$ is the \emph{degree of the elevation}. 
On the group theoretic level, if an orientation is given to $S^1$ then the loop $\ell$ represents a conjugacy class $[g]$ in $G=\pi_1(X,x_0)$ where $x_0$ is any basepoint. The covering space $Y$ corresponds to a conjugacy class of a finite index subgroup $H\findex G$, if $Y$ is a regular cover, i.e. the group of deck transformations of $\rho: Y \onto X$ acts transitively on point preimages, then $H$ will in fact be a normal subgroup.

A fixed loop $\ell$ will admit only finitely many elevations to a finite cover $Y$. If $Y$ is regular, then if non-empty, the set of 
elevations $\hat\ell_1,\ldots,\hat\ell_m$ correspond to the set of $H$-conjugacy classes $[(g^{d_1})^{n}]_H,\ldots, [(g^{d_m})^{n}]_H$ of generators of the intersection of conjugates of $\langle g \rangle$ and $H$. The power $n$ is the common degree of the elevations. We call the $H$-conjugacy classes $[(g^{d_i})^n]_H$ the \emph{algebraic elevations} of $g$ to $H$. A covering space argument gives $n\cdot m = [G:H]$. 

From the graph of groups $\calg$ we have a corresponding graph of spaces $\calx_\calg$ constructed from the canonical splitting $\calg$ of a piecewise trivial suspension as follows: the vertex space associated to a white vertex $w$ is of the form $X_w\times S^1$ where $X_w$ is a bouquet of circles with vertex $x_w$ and $F_w \cong \pi_1(X_w,x_w)$, the vertex space associated to a black vertex $b$ is a 2-torus $S^1 \times S^1$, edge spaces are also 2-tori. Any non-elliptic $\calg$-loop in short position will give rise to an immersed loop $\ell: S^1 \to \calx_\calg$.

Recall that to every $\calg$-loop we can associate a sequence of turns. We call a turn of the form $(e,e^{-1})$ a \emph{sharp turn}. We will need to eliminate sharp turns to enable impermeability. Sharp turns occur precisely when an immersed loop $\ell$ enters and then exits a vertex space via the same edge space. A finite covering space of $\calx_\calg$ is also naturally a graph of spaces. Our goal is to construct a covering space that ensures that elevations of $\ell$ no longer have sharp turns. We will achieve this using subgroup separability of free groups.

A finite cover of $\calx_\calg$ corresponds to a (conjugacy class) of a finite index subgroup of $G$. In particular, if $G$ is a piecewise trivial suspension, then so are all its finite index subgroups. A sharp turn will occur along the loop $\ell$ if in the $\calg$-loop representing its $\pi_1$-image we have a subword $$e a e^{-1}.$$ Note that for any white vertex $w$, $\calg_w = F_{w} \oplus \langle t_w \rangle$ and $t_w$ is contained in the image of every incident edge group. Thus, we may always assume that $a_i$ is contained in the fibre.

Let $H \leq G$ be a finite index subgroup. Suppose that for every white vertex $w \in V(X)$, the subgroups in the induced graph-of-groups $\calg^H$ intersect $F_w$ as a normal subgroup $H_{w} \nfindex F_w$. Suppose also that for any sharp turn $e a e^{-1}$ with $\tau(e)=w$, we have that the intersections of subgroups $I_{e}$ and $aI_{e}a^{-1}$ with $H_{w}$ are non-conjugate in $H_{w}$. Then, the elevations of $\ell$ will always enter and exit a vertex space from different edge spaces.

\begin{lemma}\label{lem:sep_sharp}
    For every $F_w$, there is a normal finite index subgroup $H_w \nfindex F_w$ such that for any sharp turn $e a e^{-1}$ that occurs in any conjugate of  $g_1$ or $g_2$ in short form with $a \in F_w$, conjugates of $I_{e}\cap H_w$ and $aI_{e}a^{-1}\cap H_w$ have trivial intersection in $H_w$. 
\end{lemma}
\begin{proof}
  Let $eae^{-1}$ with $a\in F_w$ be a sharp turn. By hypothesis, the sharp turn occurs in a reduced word so $a \not\in I_e \cap F_w = \langle c \rangle$. Also note that by hypothesis $\langle a\rangle$ is a maximal cyclic subgroup of $F_w$. Suppose that for some $H \nfindex F_w$ there exists $h \in H$ such that
  $$
  \langle c \rangle \cap h\langle a c a^{-1} \rangle h^{-1} \neq \{1\}.
  $$
Then this implies that $[ha,c]=1$, which in a free group, by maximality of the cyclic group $\langle c \rangle$ implies that  $ha \in \langle c \rangle$. In particular the image of $a$ in $F_w/H$ lies in the image of $\langle c \rangle.$
    
Now $F_w$ is subgroup separable so there exists $H_{a,c}\nfindex F_w$ that separates $a$ from $\langle c \rangle$. In particular, the subgroups $\langle c \rangle\cap H_{a,w}$ and $\langle a c a^{-1} \rangle\cap H_{a,w}$  are not conjugate in $H_{a,w}$.

Since any subgroup that is deeper than $H_{a,c}$ will also separate $a$ from $\langle c \rangle$, we may take such a normal finite index subgroup for every sharp turn that occurs in $F_w$ and then take their intersection to obtain $H_w$ with the desired properties. Note that there are finitely many sharp turns to consider since the short form representatives of $g_1$ and $g_2$ have finite length.
\end{proof}

\begin{lemma}
    Let $G = \pi_1(\calg, v)$ be a piecewise trivial suspension and let $g_1,g_2 \in G$. There exists a finite index subgroup $H \findex G$ with an induced graph of groups $\calg^H$ such that the following hold:
    \begin{enumerate}
        \item Elevations of the conjugates in short form of $g_1$ and $g_2$ to $H$ have no sharp turns with respect to $\calg^H$.
        \item There is some exponent $k_H$ such that every black vertex group of $\calg^H$ is conjugate to some subgroup of $k_H$ powers $k_H \cdot \calg_b$ where $b \in V(X)$ a black vertex.
        \item Every white vertex group $\calg^H_{\hat{w}}$ is conjugate to a subgroup $\langle H_w, t_w^{k_H} \rangle \leq F_w \oplus \langle t_w \rangle = \calg_w$ and the image of every edge group in $\calg^H_{\hat{w}}$ is conjugate to a subgroup of the form $\langle c^{k_H}, t_w^{k_H} \rangle \cong \ZZ^2$ where $c$ is a fibrewise generator of an edge group.
    \end{enumerate}
\end{lemma}
\begin{proof}
 
  By \cref{lem:sep_sharp}, for every white vertex $w \in V(X)$, there is a finite index normal subgroup $J_w \nfindex F_w$ such that for every sharp turn $eae^{-1}$ occurring in any conjugate of $g_1$ or $g_2$ in short position, we have that the $I_e \cap J_w$ and $aI_ea^{-1} \cap J_w$ have no conjugates in $J_w$ with non-trivial intersection.

Fix a white vertex $w \in V(X)$. For each incident edge $e$ at $w$, let $h_e$ be the generator of $\calg_e \cap J_w$ and let $\mathcal{H}_w$ be the set of all the elements $h_e$. The hypotheses imply that $\mathcal{H}_w$ is an independent set and thus $J_w$ commands the elements in $\mathcal{H}_w$ by \cite[Theorem~3.5]{wise_subgroup_2000}. Each $h_e$ is an algebraic elevation of a fibrewise generator $c_{e_i} \in F_w$ of the image of an incident edge group.

  Now, using command, we may find some positive integer $D$ and for each white vertex $w \in V(X)$ a finite index subgroups $H_w \trianglelefteq_f F_w$, such that all algebraic elevations of the fibrewise generators of edge groups in $F_w$ have a common degree $D$. We note that for every fibrewise generator $c_{e_i}\in F_w$ the sum of the algebraic degrees of all its algebraic elevations is the index $[F_w:H_w]$. 

  We can now construct finite index subgroups
  \[
    \langle H_w,t_w^D \rangle  \leq \calg_w \cong F_w \oplus \langle t_w \rangle 
  \] of the white vertex groups and take the $D$-congruence subgroups $D\cdot \mathbb Z^2$ of the black vertex groups. By following the construction of \cite[Proposition 6.14]{dahmani_unipotent_2024} it is possible to construct a finite degree covering graph of spaces, and therefore a finite index subgroup $H$ with the desired properties with $k_H = D$.
\end{proof}

We call the finite index subgroup $H\leq \pi_1(\calg)$ constructed in Lemma \ref{lem:sep_sharp} \emph{unsharpened relative to $g_1,g_2$.} We now construct the vertex filling of $H = \pi_1(\calg^H)$ that witnesses the non-conjugacy of $g_1$ and $g_2$.

\begin{prop}\label{prop:conjugacy-separability-same-DCR}
  Let $g_1,g_2$ be non-conjugate elements of $G=\pi_1(\calg)$ with $DCR(g_1)=DCR(g_2)$. Then it is possible to find a finite index subgroup $H = \pi_1(\calg^H)$ that admits a virtually free vertex filling $\calg^H \to \overline{\calg^H}$ in which all algebraic elevations of $g_1$ and $g_2$ have non-conjugate images.
\end{prop}
\begin{proof}
  We first apply Lemma \ref{lem:sep_sharp} in order to find the finite index subgroup $H = \pi_1(\calg^H)$ that is unsharpened relative to $g_1,g_2$. Let $P$ be the product of the positive integers $M$ given in \cref{prop:no_sol_const} coming from all pairs of non-conjugate algebraic elevations of $g_1,g_2$ in $H$.

  Consider now a white vertex group $\calg^H_w \cong H_w\oplus \langle t^{k_H}\rangle\leq \calg_{[w]}$, where $\calg_{[w]}$ is the vertex group of $\calg$ that naturally contains $\calg^H_w$.  The elements \[E_w=\{c_e \in H_{w}: \textrm{$e$ is incident to $w$}\}\] are independent. Let $N_w$ be the integer given in Definition \ref{defn:command} given for $H_w$ and the set $E_w$, and let \[
    N_H = M\cdot \prod_{w}N_w
  \] where $w$ ranges over the white vertices of $\calg^H$. For each turn $\tau=e'ae^{-1}$ that occurs in a normal form of an algebraic elevation of $g_1,g_2$ we want to find a finite quotient of $H_w\to \overline{H_w}$ such that the images of $\langle c_{e'}\rangle$ and $\langle a c_e a^{-1} \rangle$ have trivial intersection. By \cite[Theorem 4.3]{BridsonWilton2015}, $H_w$ strongly commands $E_w$ and replacing some $c_e$ with a conjugate $ac_ea^{-1}$ we can ensure by strong command that there is a finite quotient $H_w^\tau$ where the images of all the elements $c_e$ have order $N_H$ and where the images of $\langle c_{e'}\rangle$ and $\langle a c_e a^{-1} \rangle$ have trivial intersection. Let $K^\tau_w$ be the kernel of this homomorphism and let $K_w = \bigcap_{\tau} K^\tau_w$ where $\tau$ ranges over the turns that occur in $w$. Then the normal subgroup $\calk_w=K_w \oplus \langle t^{k_HN_H}\rangle \nfindex H_w\oplus \langle t^{k_H}\rangle$ is such that for every turn $\tau = e'ae^{-1}$, the element $a$ is in the  fibre, the corresponding double coset images are impermeable and the image of every edge group is a $N_H$-congruence quotient of the image of the edge group in $\calg^H_w$. It follows that that the $\calk_w$ give a system of vertex fillings and the result follows from Proposition \ref{prop:big_N_separate}.
\end{proof}

\subsection{Conjugacy separability in the unipotent linear case}

\begin{thm}\label{thm:conjugacy-separable-linear-UPG}
    Let $G$ be a \{finitely generated free\}-by-cyclic group with unipotent and linearly growing monodromy. Then $G$ is conjugacy separable and every cyclic subgroup of $G$ is conjugacy distinguished.
\end{thm}

\begin{proof}
    We begin by showing that $G$ is conjugacy separable. Since $G$ has the unique roots property by \cref{lemma:unique-roots-UPG}, by \cref{lemma:Cotton-Barrett--Wilton-lemma} it suffices to show that $G$ contains a conjugacy separable subgroup of finite index. 
    
   Up to replacing $G$ by a finite index subgroup we may assume that $G$ admits a splitting $G = \pi_1(\calg)$ where $\calg$ is a clean piecewise trivial suspension satisfying the properties in \cref{prop:piecewise-trivial-suspensions-properties}. Let $g_1, g_2 \in G$ be non-conjugate elements.

    If both $g_1$ and $g_2$ are elliptic, then by \cref{lemma:elliptic-conjugacy} there exists a finite quotient of $G$ such that the image of $g_1$ is not conjugate to the image of $g_2$. Moreover, if one of the $g_i$ is loxodromic, or if both $g_1$ and $g_2$ are loxodromic with different translation lengths, then by \cref{lemma:mixed-loxodromic-conjugacy} there exists a finite quotient where the images are not conjugate. 

   If no conjugates of $g_1$ and $g_2$ in short position have the same sets of double coset representatives then their conjugacy classes can be separated in a finite quotient by \cref{lemma:separating-DCR}. Otherwise, we combine \cref{prop:big_N_separate} with \cref{prop:conjugacy-separability-same-DCR}, noting that we may pass to a finite index subgroup by \cref{lem:designer-subgroup}.

We also claim that cyclic subgroups are conjugacy distinguished in $G$. We will again use the piecewise trivial splitting of $G$, noting that the edge groups are root-closed by \cref{lemma:UPG-root-closed-subgroups}. Since we have already established that $G$ is conjugacy separable, we can apply \cref{lemma:conjugacy-distinguished-elliptic} and \cref{lemma:conj-distinguished-loxodromics} to conclude the claim.
\end{proof}

\section{Double-\texorpdfstring{$\Z$}{Z}-coset separability}\label{sec:dbleZsep}

The aim of this section is to prove that double cosets of cyclic subgroups are separable in \{finitely generated free\}-by-cyclic groups with polynomially growing monodromy (see \Cref{prop:double-coset-unipotent}).
Note that for any element $g \in G$, the double coset $HgK \subset G$ is separable if and only if $HgKg^{-1} \subset G$ is separable. Hence we only need to consider separability of double cosets of the form $HK' \subset G$ where $H$ and $K'$ are cyclic.

\begin{lemma}\label{lem:aerial}
    Let $\calg$ be a $\kappa$-acylindrical graph of groups 
    and let $h,k$ be two hyperbolic elements of $\pi_1(\calg, v_0)$ with distinct respective axes $\alpha_h,\alpha_k$ in the dual Bass--Serre tree $T$. Let $H=\langle h \rangle$ and $K = \langle k \rangle$.  Let $v$ be a vertex in $\alpha_k$ and let $T_{HK}\subset T$ be the convex hull of $HK\cdot v$. Then the vertices of $T_{HK}$ have valence at most 4. Furthermore, for any vertex $w \in T_{HK}$ there are at most finitely many elements $x \in HK$ such that $x\cdot v = w$.
\end{lemma}
\begin{proof} Since $h,k$ have distinct axes, acylindricity forces the intersection of the axes to be contained in a finite interval. The tree $T_{HK}$ must therefore have one of the following configurations shown in Figure \ref{fig:aerials}, depending on whether the axes $\alpha_h,\alpha_k$ are disjoint, intersect in a point, or intersect in an arc. The advertised bound on vertex degrees follows immediately.
\begin{figure}[htb]
    \centering
    \includegraphics[width=\linewidth]{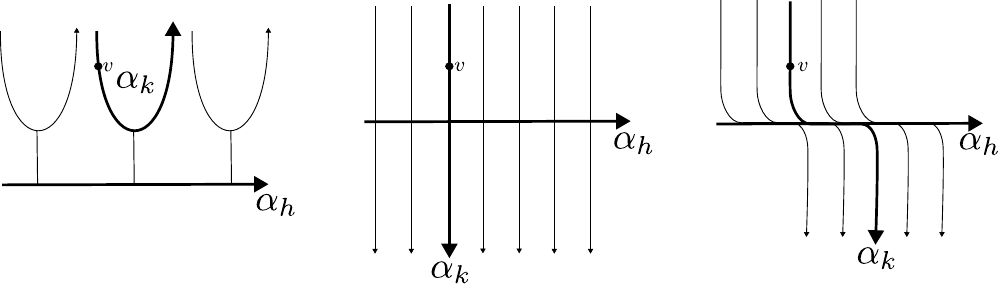}
    \caption{The three configurations for $T_{HK}$. The axes $\alpha_h,\alpha_k$ are drawn thicker.}
    \label{fig:aerials}
\end{figure}
Noting that $h$ and $k$ translate points along the axes $\alpha_h$ and $\alpha_k$, respectively, we immediately see that in the first two configurations the set $HK\cdot v = \{h^mk^n \cdot v: m,n \in \mathbb Z\}$ is in bijective correspondence with $HK$. In the third case, where the axes have an intersection that contains a non-trivial arc $I$, the only way for it to be possible that $h^{m_1}k^{n_1}\cdot v=h^{m_2}k^{n_2}\cdot v$, is if both $k^{n_1}v,k^{n_2}v \in I$ and the different powers of $h$ make these two points coincide. Since there are only finitely many powers $k^n$ of $k$ such that $k^n\cdot v \in I$ the finiteness result follows.
\end{proof}

\begin{lemma}\label{lemma:hyp-hyp-separability}
    Let $\calg$ be a $\kappa$-acylindrical graph of groups that admits arbitrarily deep edge group separating virtually free vertex fillings. Suppose that the vertex groups are residually finite and assume that $G = \pi_1(\calg, v_0)$ does not contain elements of order 2. Let $H,K \leq G$ be cyclic non-elliptic subgroups. Then the double coset $HK$ is separable in $G$. 
\end{lemma}
\begin{proof}
    
  Let $g \in G \setminus HK$. We begin by constructing a virtually free quotient of $G$ such that the image of $g$ is not mapped to the image of $HK$.
  
    We will split into two cases depending on whether the axes of the cyclic subgroups $H$ and $K$ are distinct or coincide.  

    \noindent \textbf{Case 1:} \emph{$H$ and $K$ have distinct axes.} 
    
    Let $T_{HK} \subset T$ be the tree given in the statement of Lemma \ref{lem:aerial}. By construction, there is a finite connected subgraph $\cald \subset T_{HK}$ such that \[T_{HK} \subset \bigcup_{x\in HK} x\cdot \cald.\] In particular, if $w$ is any vertex in $\cald$ and $e$ is an edge of $\cald$ that is adjacent to $w$, then if $w' \in HK\cdot w$ is another vertex in $T_{HK}$, it follows that there are finitely many elements $x \in HK$ such that $x\cdot w = w'$ and $x\cdot e$ is an edge of $T_{HK}$ adjacent to $w'$. Then, by the edge group separating hypothesis, there exists a vertex filling $\calg \to \overline\calg$ that doesn't fold any pair of edges of $T_{HK}$ incident to any vertex $w \in \cald$. It follows that for any sufficiently deep virtually free vertex filling $T \to \overline{T}$, $T_{HK}$ is mapped injectively onto its image. 
    
    Let us now fix a vertex $v$ in the axis of $k$. We split the argument into two subcases depending on the image $g\cdot v$.

    \noindent \textbf{Subcase i:}  \emph{$g\cdot v \in  T_{HK}$.}
 
    Let $\mathcal{A} =  \calg_v \cap g^{-1}HK$. Since $g \not\in HK$, it must be the case that all the elements of $\mathcal{A}$ are non-trivial and by \cref{lem:aerial}, $\mathcal{A}$ is finite.

    Since $\mathcal{G}_v$ is residually finite, there exists a finite index normal subgroup $N_v \trianglelefteq_f \mathcal{G}_v$ such that $N_v \cap \mathcal{A} = \emptyset$. Let $G \to \overline{G} = G / N$ be a virtually free vertex filling associated to $N_v$ and such that $T_{HK}$ is mapped injectively onto its image. In particular, for any vertex $v$ we have that $\overline{\calg}_v \cap \overline{HK} = \overline{\calg_v \cap HK}$.
    
    Suppose for contradiction that $\bar{g} \in \overline{HK}$. Then $\bar{g} \in \bar{g}  \overline{\calg}_v \cap \overline{HK} = \overline{g \calg_v \cap HK}$. Hence there exists $y \in g \calg_v \cap HK$ such that $\bar{y} = \bar{g}$. Hence, \[g^{-1}y \in N \cap \calg_v \cap g^{-1}HK = N_v \cap \mathcal{A} = \emptyset.\]
    This is a contradiction.

\noindent\textbf{Subcase ii:} $g\cdot v \not\in T_{HK}$. 

Let $\rho$ be the path of minimal length joining $g\cdot v$ to $T_{HK}$. Note that by Lemma \ref{lem:aerial} the maximal possible valency of any vertex in the tree $T_{HK}\cup\rho$ is 5. Thus, arguing as before we may take a sufficiently deep virtually free vertex filling so that $T_{HK}\cup\rho$ is mapped injectively onto its image. It now follows that \[\bar g \cdot \bar v \not\in \overline{T_{HK}} \supset \overline{HK}\bar v\] which means that $\bar g\not \in \overline{HK}$ as required. \hfill$\blackdiamond$

    We have exhausted all possibilities for $g\cdot v$ completing the proof of Case 1. \hfill$\blacksquare$

    \noindent \textbf{Case 2:} \emph{The axes of $H$ and $K$ coincide.}
    
    Since the action of $G$ on $T$ is acylindrical, it must be the case that $H$ and $K$ are contained in a maximal cyclic subgroup.  Here we are using acylindricity and that there are no elements of order $2$ (otherwise they could be contained in $D_\infty$). Thus, we see that $HK$ coincides with a cyclic subgroup $C$. The result in this case follows by the same arguments as those given above, the only difference is that $T_{HK}$ is the minimal $C$-invariant tree $T_C$.

    Finally, since $g \in G \setminus HK$ was arbitrary, and since finitely generated virtually free groups are double coset separable, we conclude that $HK$ is separable in $G$. 
\end{proof}

\begin{lemma}\label{lemma:double-coset-mixed}
  Let $\calg$ be a $\kappa$-acylindrical graph of groups that admits arbitrarily deep edge separating virtually free vertex fillings and let $G = \pi_1(\calg, v_0)$.  Let $H,K \leq G$ be cyclic subgroups and suppose that $H$ is generated by a loxodromic element and $K$ is generated by an elliptic element. Then the double cosets $HK$ and $KH$ are separable in $G$.
\end{lemma}

\begin{proof}
 Let $H = \langle h \rangle$ and let $\alpha_h$ denote the axis of $h$. Let $v \in \Fix(K)$ be the element of $\Fix(K)$ closest to $\alpha_h$. Let $T_{HK}$ be the convex hull of the set $HK\cdot v$ in the Bass--Serre tree $T$ associated to $\calg$. If $v$ is contained in $\alpha_h$ then $T_{HK} = \alpha_h$. Otherwise, $T_{HK}$ is as in \cref{fig:T_HK-hyperbolic-elliptic}. In particular, $T_{HK}$ is locally finite and for any vertex $w \in T_{HK}$, there is at most one $m \in \Z$ such that $h^mk \cdot v = w$ for some $k \in K$. Now the argument is the same as in \cref{lemma:hyp-hyp-separability}. To show that $KH$ is separable, we note that the map $G \to G, g \mapsto g^{-1}$ is continuous with respect to the profinite topology on $G$. Hence, if $HK \subseteq G$ is closed then so is $K^{-1}H^{-1} = KH$. \end{proof}

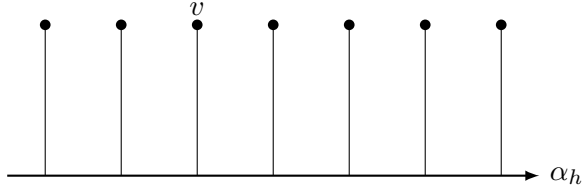
\begin{figure}
    \centering
  
\begin{tikzpicture}
  \draw[->, thick, >=latex, fill=black] (-0.5,0) -- (6.5,0) node[right] {$\alpha_h$};

  \def\n{6}       
  \def\dx{1}      
  \def\height{2}  

  \foreach \i in {0,...,\n} {
    \draw({\i*\dx},0) -- ({\i*\dx},\height);  
    \fill ({\i*\dx},\height) circle (2pt);  
  }

  \node at (2,\height) [above] {$v$};

\end{tikzpicture}
    \caption{The configuration for $T_{HK}$ when $H$ is hyperbolic and $K$ elliptic.}
    \label{fig:T_HK-hyperbolic-elliptic}
\end{figure}

\begin{lemma}\label{lemma:double-cosets-elliptic-case1}
    Let $\calg$ be a $\kappa$-acylindrical graph of groups such that double cosets of cyclic subgroups are separable in the vertex groups. Let $G = \pi_1(\calg, v_0)$. Suppose that all the vertex groups are fully separable in $G$. Let $H, K \leq G$ be cyclic subgroups generated by elliptic elements such that $\Fix(H) \cap \Fix(K) \neq \emptyset$. Then $HK$ is separable in $G$.  
\end{lemma}

\begin{proof}
    Pick a vertex $v \in \Fix(H) \cap \Fix(K)$. Hence, $H, K \leq \calg_v$. Then, since double cosets of cyclic subgroups are separable in $\calg_v$, it follows that $HK$ is closed in the profinite topology on $\calg_v$. By assumption, $G$ induces the full profinite topology on $\calg_v$ and thus $HK \subseteq G$ is separable.
\end{proof}

\begin{lemma}\label{lemma:double-coset-separability-pts}
Let $\mathcal{G}$ be a clean piecewise trivial suspension with free local fibres and let $G = \pi_1(\calg, v_0)$. Let $H,K \leq G$ be cyclic subgroups generated by elliptic elements. Suppose that $\Fix(H) \cap \Fix(K) = \emptyset$. Then $HK$ is separable in $G$.
\end{lemma}

We postpone the proof of \cref{lemma:double-coset-separability-pts} until the end of the section. 

\begin{prop}\label{prop:double-coset-separability-linear}
Let $G$ be \{finitely generated free\}-by-cyclic with unipotent and linearly growing monodromy. Then double cosets of cyclic subgroups are separable. 
\end{prop}

\begin{proof}

By \cref{lemma:double-coset-separability-finite index}, it suffices to show that $G$ admits a finite index subgroup with separable double cosets of cyclic subgroups. Hence, assume that $G$ admits a splitting $G \cong\pi_1(\calg)$ where $\calg$ is a clean piecewise trivial suspension with free local fibres. Then, all vertex groups are residually finite and all edge groups are separable in the vertex groups. Hence, $\calg$ admits edge separating virtually free vertex fillings by \cref{prop:linear-vertex-fillings} and \cref{lemma:edge separating-fillings}. 
    
If at least one of the subgroups is generated by a hyperbolic element, then the result follows by \cref{lemma:hyp-hyp-separability} and \cref{lemma:double-coset-mixed}. Otherwise, the result holds by \cref{lemma:double-cosets-elliptic-case1} and  \cref{lemma:double-coset-separability-pts}.
\end{proof}

We are now ready to prove the main result of this section.

\begin{duplicate}[\Cref{prop:double-coset-unipotent}]
    Let $G$ be a \{finitely generated free\}-by-cyclic group with polynomially growing monodromy. Then double cosets of cyclic subgroups are separable in $G$.
\end{duplicate}
\begin{proof}

By \cref{lemma:double-coset-separability-finite index}, it suffices to show that $G$ admits a finite index subgroup with separable double cosets of cyclic subgroups. Hence, we may assume without loss of generality that $G$ has unipotent and polynomially growing monodromy. 

We will argue by induction on the degree $d$ of growth. If $d = 0$ then $G \cong F \times \Z$ where $F$ is a free group of finite rank. Double cosets of finitely generated subgroups of free groups are separable by \cref{thm:free-double-coset}. It follows that double cosets of cyclic subgroups are separable in $G$. 

The case $d = 1$ is proved in \cref{prop:double-coset-separability-linear}.

Now suppose that $d \geq 2$. Let $G = \pi_1(\calg)$ be the standard splitting and let $T$ be the associated Bass--Serre tree. Let $H$ and $K$ be finitely generated cyclic subgroups of $G$. 

By \cref{lem:abelian-subgroups-finite index}, we may find a finite index subgroup $G' = F \rtimes_{\phi} \langle t \rangle$ where $\phi \in \Aut(F)$ is unipotent and polynomially growing of degree $d$, and such that $H' := H \cap G' = \langle ut \rangle$ or $H' = \langle v \rangle$, and $K' := K \cap G' = \langle u't \rangle$ or $K' = \langle v' \rangle$ for some $u, u',v, v' \in F$ and $v$ and $v'$ not proper powers. By \cref{lemma:double-coset-separability-finite index}, it suffices to show that the double coset $H' K' \subset G'$ is separable. Hence, without loss of generality, we may assume that $H$ and $K$ are of this form. 

If at least one of $H$ or $K$ is generated by a hyperbolic element, then $HK$ is separable by \cref{lemma:hyp-hyp-separability} and \cref{lemma:double-coset-mixed}. If both $H$ and $K$ are elliptic and $\Fix(H) \cap \Fix(K) \neq \emptyset$ then the result follows by \cref{lemma:double-cosets-elliptic-case1}. Hence, let us assume that $H$ and $K$ are elliptic and $\Fix(H) \cap \Fix(K) = \emptyset$. Let $[a,b]\subset T$ be the segment that realises the distance between $\Fix(H)$ and $\Fix(K)$ where $a$ is fixed by $K$ and $b$ is fixed by $H$. We define \[T_{HK} = \bigcup_{h\in H} h\cdot[a,b] \] and let $g \in G \setminus HK$. 

Note that by the root-closed property of the edge groups, if there is an edge on $[a,b]$ fixed by a proper power of the generator of $H$, then it is fixed by all of $H$. Hence, no edge on $[a,b]$ is fixed by a non-trivial element of $H$. Let $e$ be the edge on $[a,b]$ adjacent to $b$.

    \noindent \textbf{Case 1:}\label{case1-general-double-cosets} $g\cdot a \in HK \cdot a$. 
    
    Let $h \in H$ be such that $g \cdot a = h \cdot a$, note that such $h$ exists because $K$ fixes $a$ by assumption.  Since $g  \not\in HK$, we have that $h^{-1}g\in \calg_a\setminus K$.  

    \begin{claim} \label{claim1}
        For a sufficiently high prime $p$, there exists a $p$-periodic quotient $\calg_a \to \overline{\calg}_a$ such that $\overline{h^{-1}g}  \not\in \overline{K}$ and the image of every incident edge group has order $p$.
    \end{claim} 

    \begin{proof}

        We will show that $\langle k \rangle \cap \langle h^{-1}g \rangle = 1$ or $[k, h^{-1}g] \neq 1$. 
        
        Suppose first that $k$ and $h^{-1}g$ are elements of the fibre $F$ and $\langle k \rangle \cap \langle h^{-1}g \rangle \neq 1$. Then $k$ and $h^{-1}g$ must be powers of a common element. Since $k \in F$ is not a proper power (by our initial assumptions on $H$ and $K$), it follows that $h^{-1}g \in K$, a contradiction. If exactly one of $\{k, h^{-1}g\}$ is contained in the fibre, then it is clear that $\langle k \rangle \cap \langle h^{-1}g \rangle = 1$. Finally, if both are not contained in the fibre then in particular $k \in F \cdot t$ (by our initial assumptions on $H$ and $K$), and thus $C_G(k) = (C_G(k) \cap F) \oplus \langle k \rangle$. Hence, if $k$ and $h^{-1}g$ commute and $\langle k \rangle \cap \langle h^{-1}g \rangle \neq 1$ then it must be the case that $h^{-1}g \in \langle k \rangle$, another contradiction.

        Now, by \Cref{prop} there exists a positive integer $L_a > 0$ such that for any prime $p > L_a$, there is a $p$-periodic quotient $\calg_a \to \overline{\calg_a}$ such that the images of the generators of incident edge groups and the elements $h^{-1}g$ and $k$ are non-trivial and such that $\overline{\langle h^{-1}g \rangle} \cap \overline{\langle k \rangle} = 1$.  In particular, $\overline{h^{-1}g} \not \in \overline{K}$.

\end{proof}

\begin{claim}\label{claim2} For a sufficiently high prime $p$, there exists a $p$-periodic quotient $\calg_b \to \overline{\calg}_b$ such that the images of $H$ and $\calg_e$ have trivial intersection and each incident edge group has order $p$.
\end{claim}

\begin{proof}

This is essentially the same argument as \cref{claim1}.

\end{proof}

\begin{claim}\label{claim3}

There is a virtually free filling $\calg \to \overline{\calg}$ such that the images of $\calg_e$ and $H$ have trivial intersection, and the image of $h^{-1}g $ is not contained in the image of $K$.
\end{claim}

\begin{proof}

For each vertex $v \in V(X) \setminus \{a,b\}$, we use \cref{cor:periodic-quotients} to construct a finite quotient such that each incident edge group is isomorphic to $\Z / p\Z$ for a sufficiently high prime $p$.  By \cref{claim1} and \cref{claim2}, there is a sufficiently high prime $p$ such that we may find a $p$-periodic quotient of $\calg_a$ and $\calg_b$ such that all incident edge groups are isomorphic to $\Z / p\Z$, the image of $\calg_e$ and $H$ have trivial intersection, and the image of $h^{-1}g $ is not contained in the image of $K$. Then, we may assemble all such quotients to construct a virtually free vertex filling $\calg \to \overline{\calg}$.
\end{proof}

 We now complete the proof of Case 1.  To this end, consider the filling $\calg \to \overline{\calg}$ constructed in \cref{claim3} and suppose for contradiction that $\bar{g} \in \overline{H}\overline{K}$. Then $\bar{g} = \bar{h'} \bar{k}$ for some $h' \in H$ and $k \in K$. Thus $\bar{h'}^{-1} \bar{g} \in \overline{\calg}_a$. Hence $(\bar{h'}^{-1}  \bar{g})( \bar{g}^{-1} \bar{h}) \in \overline{\calg}_a$. It follows that $\bar{h'}^{-1} \bar{h} \in \overline{\calg}_a \cap \overline{\calg}_b \leq \overline{\calg}_{e}$. Hence,  $\bar{h'}^{-1} \bar{h} \in \overline{\calg_{e}} \cap \overline{H}$. However, by \Cref{claim3} we have constructed a virtually free vertex filling such that the images of $\calg_{e}$ and $H$ have trivial intersection. Hence, we must have that $\bar{h'}^{-1} \bar{h} = 1$. Thus, $\bar{h}^{-1} \bar{g} \in \overline{K}$. Again by \Cref{claim3}, in the constructed virtually free vertex filling the image of $h^{-1}g$ is not contained in the image of $K$ and thus we have arrived at a contradiction. It follows that $\bar{g} \not \in \overline{H} \overline{K}$ as required. \hfill$\blacksquare$

 \smallskip

  \noindent \textbf{Case 2:}\label{case2-general-double-cosets} $g \cdot a \not\in HK \cdot a$. 
  
  Let $e'$ be the edge on $[g\cdot a, b]$ adjacent to $b$.

  \noindent \textbf{Subcase i:} \emph{There is some $h\in H$ such that $h \cdot e = e'$.} 
  
  By assumption, we have that $h^{-1}g \cdot a \neq a$ and $e$ is an edge in the intersection $h^{-1}[g\cdot a, b] \cap [a,b]$. Let $\eta$ be the geodesic in $T$ that joins $a$ to $h^{-1}g \cdot a$. Note that the path $\eta$ does not contain $b$. Using \cref{cor:periodic-quotients} there exists some $L > 0$ such that for every prime $p > L$ there is a virtually free vertex filling of $T$ such that $\eta$ is mapped injectively to $\overline{T}$ and the order of every edge group is $p$. Moreover, we may find some further $L' > L$ such that for any prime $p > L'$, there is a finite quotient $Q_b$ of $\calg_{b'}$ so that the images of $H$ and $\calg_{e}$ have trivial intersection and every incident edge has order $p$. As before, we may replace the image of $\calg_{b}$ in the original virtually free vertex filling with the quotient $Q_b$, to obtain a virtually free vertex filling $\calg \to \overline{\calg}$ so that the images of $H$ and $\calg_e$ have trivial intersection and the segment $\eta$ is mapped injectively via the folding map $T \to \overline{T}$. 
    
   Suppose now that $\bar{g} = \bar{h'}\cdot \bar{k}$ for $h' \in H$ and $k \in K$. Then $\bar{h'} \cdot \bar{e}= \bar{h} \cdot \bar{e}$ and thus $\bar{h'}^{-1} \bar{h} \in \overline{\calg_{e}} \cap \overline{H}$. But by construction of the virtually free vertex filling, we have that $\overline{H} \cap \overline{\calg_{e}} = 1$ and thus $\bar{h'} = \bar{h}$. It follows that $\bar{h}^{-1} \bar{g} \cdot \bar{a} = \bar{a}$. However, the fact that the geodesic $\eta$ is mapped injectively via the vertex filling $T \to \overline{T}$ implies that $\bar{h}^{-1}\bar{g} \cdot \bar{a} \neq \bar{a}$, which gives a contradiction. Hence in this case also we have that $\bar{g} \not \in \overline{HK}$.   \hfill $\blackdiamond$

    \noindent \textbf{Subcase ii:} \emph{There is no $h \in H$ such that $h \cdot e = e'$.}
    
    Let $\calg \to \overline{\calg}$ be a virtually free vertex filling such that no edges in $[a,b]$ are folded in the image and the same is true for $[g\cdot a, b]$. If $e$ and $e'$ are not in the same $G$-orbit then it cannot be the case that $\bar{e}$ and $\bar{e}'$ are in the same $\overline{G}$-orbit by \cref{lem:orbits-in-virtually free-fillings}. Then, we must have that $\bar{g} \not \in \overline{HK}$. 
    
    Let us then assume then that there exists some $x \in G$ such that $x \cdot e = e'$. It must then be the case that $x \not \in H \calg_{e}$. By induction, using double coset separability of abelian subgroups in the vertex groups, we may construct a deeper virtually free vertex filling such that the image of $x$ is not in the image of the double coset $ H \calg_{e}$. Hence, again we have that $\bar{g}\cdot \bar{a}$ will not be in the image of $\overline{HK}\cdot \bar{a}$. \hfill $\blackdiamond$

    The two subcases together complete the proof of Case 2 \hfill $\blacksquare$
    
    Cases 1 and 2 exhaust all possibilities and the result follows.
\end{proof}

We end the section with a sketch of the proof of \cref{lemma:double-coset-separability-pts}. The argument follows the same outline as the proof of \cref{prop:double-coset-unipotent} and thus we give a detailed account of only one of the cases, leaving the others as an exercise. 

\begin{proof}[Proof of \cref{lemma:double-coset-separability-pts}]

As before, after possibly passing to a further finite index subgroup, we may assume that $G = F \rtimes \langle t \rangle$ and $H = \langle ut \rangle$ or $H = \langle v \rangle$, and $K=  \langle u't \rangle$ or $K = \langle v' \rangle$ for some $u, u',v, v' \in F$ and $v$ and $v'$ not proper powers. 

   Let $[a,b] \subset T$ be the segment realising the distance between $\Fix(K)$ and $\Fix(H)$, with $a \in \Fix(K)$ and $b \in \Fix(H)$. Note that since the edge groups map surjectively onto the black vertex groups, it must be the case that both $a$ and $b$ are white vertices and thus are of the form $\calg_a = F_a \oplus \langle t_a \rangle$ and $\calg_b = F_b \oplus \langle t_b \rangle$ where $F_a$ and $F_b$ are finite rank non-abelian subgroup of the fibre $F$. 
   
   Let $T_{HK} = \bigcup_{h \in H} h \cdot [a,b]$. Let $e$ be the edge in $[a,b]$ adjacent to $b$. Let $g \in G \setminus HK$. 

    We begin by assuming that $g \cdot a \in HK \cdot a$
    . Thus, there is some $h \in H$ such that $h^{-1}g \in \calg_a$ and thus $h^{-1}g \in \calg_a \setminus K$. Now since $\calg_a$ is cyclic subgroup separable, there is a finite index normal subgroup $N_a \trianglelefteq \calg_a$ such that the image of $h^{-1}g$ is separated from the image of $K$ in of $\calg_a / N_a$. Let $t_a \in \calg_a$ be the central element. 
    By \cref{prop:linear-vertex-fillings} and \cref{ramrk:linear-vertex-fillings}, we may find some positive integer $N >0$ such that for every $i \in \NN$, there is a virtually free vertex filling of $\calg$ associated to $\{N_a\}$ such that the image of each edge group is of the form $ (\Z / iN \Z)^2$ and quotients of white vertices are of the form $F_w \oplus \langle t_w \rangle \twoheadrightarrow F_w/ D_w \oplus \langle \overline{t}_w \rangle$ where $\overline{t}_w$ has order $iN$.
    
    The image of $\calg_{e}$ in $\calg_b$ is $\langle c_e, t_b \rangle$ for some $c_e \in F \cap \calg_b =: F_b$ and $h = ut_b^k$ for some $k \in \Z$ and $u \in F_b$. Note that no conjugate of a power of $h$ fixes the edge $e$. Hence, $u$ and $c_e$ are independent in $F_b$. Now for every other incident edge $e'$ at $b$, let $c_{e'}$ be a generator of the cyclic subgroup $\calg_{e'} \cap F$. By strong omnipotence of free groups, there exists some positive integer $L > 0$ such that for any positive integer $j > 0$, there is a finite index normal subgroup $(F_b')_j \trianglelefteq_f F_b$ such that the images of the subgroups generated by $u$ and $c_e$ have trivial intersection and the order of the image of $c_{e'}$ for every edge $e'$ incident at $b$ is $jL$. In particular, the quotient $ \calg_b \to \calg_b /  \langle (F_b')_{LN} , t_b^{LN} \rangle$ is such that the image of each incident edge group is of the form 
    \[ \langle \bar{c}_e\rangle \oplus \langle \overline{t}_b \rangle \cong \Z / LN \Z \times \Z / LN \Z,\]
    and the intersection of the image of the cyclic subgroup generated by $h$ with the image of the incident edge group $\calg_e$ is trivial. 
    
    Now, construct a vertex filling $\calg \to \overline{\calg}$ as described two paragraphs above, setting $i = L$. Replace the image of the vertex group corresponding to the vertex $b$ in the filling by the quotient constructed in the previous paragraph, 
    \[\overline{\calg}_b = \calg_b /  \langle (F_b')_{LN}, t_b^{LN} \rangle.\] As a result, we obtain a filling of $\calg$ such that $\langle \bar{h} \rangle \cap \overline{\calg_e} = 1$ and $\overline{h^{-1}g} \not \in \overline{K}.$ Now we may apply the same argument as in the proof of \cref{case1-general-double-cosets} to conclude that $\bar{g} \not \in \overline{H}\overline{K}$. 

     The other case, that is when $g\cdot a \not\in HK\cdot a$, follows the same outline as \cref{case1-general-double-cosets}, again using the strong omnipotence of free groups to construct finite quotients of the white vertex groups with the required intersection properties of subgroups. 
\end{proof}

\section{Conjugacy separability for polynomial growth}\label{sec:ConjSep}

\subsection{The polynomial unipotent case}

In order to prove that all \{finitely generated free\}-by-cyclic groups with unipotent monodromies are conjugacy separable, we will induct on the degree of growth and use the following combination theorem of Wilton--Zalesskii:

\begin{thm}[Wilton--Zalesskii \cite{WiltonZalesskii2010}]\label{thm:combination-Wilon--Zalesskii}
Let $\mathcal{G}$ be a graph of groups with conjugacy separable vertex groups and suppose that the profinite topology on $G = \pi_1(\mathcal{G})$ is efficient.  Suppose that the following conditions are satisfied for all vertices $v \in V(X)$ and edges $e,f \in E(X)$ that are adjacent to $v$. 
\begin{enumerate}
    \item\label{1} The edge group $\calg_e$ is conjugacy distinguished in $\calg_v$.
    \item\label{2} For any $g \in \calg_v$, the double coset $\calg_e g \calg_f$ is separable in $\calg_v$.
    \item\label{3} The intersection $\overline{\calg_e} \cap \overline{\calg_f}$ of the closures of $\calg_e$ and $\calg_f$ in the profinite completion of $\calg_v$, is equal to the profinite completion $\widehat{\calg_e \cap \calg_v}$.
    \item The graph of groups $\mathcal{G}$ is profinitely 2-acylindrical.
\end{enumerate}
Then $G$ is conjugacy separable.
\end{thm}

We are now ready to prove the main result for unipotent and polynomially growing automorphisms. 

\begin{thm}\label{thm:unipotent-conj-sep}
    Let $G$ be a \{finitely generated free\}-by-cyclic group with unipotent and polynomially growing monodromy. Then $G$ is conjugacy separable and every cyclic subgroup is conjugacy distinguished.
\end{thm}

\begin{proof}
    We will argue by induction on the degree $d$ of growth. If $d = 0$ then $G \cong F \times \Z$ where $F$ is a finitely generated free group. It is not hard to see that $G$ is conjugacy separable and that every infinite cyclic subgroup is conjugacy distinguished, since those properties hold for the fibre $F$.

   When $d = 1$, both results follow by \cref{thm:conjugacy-separable-linear-UPG}.
   
 Suppose now that $d \geq 2$. Let $G \cong \pi_1(\calg, v)$ be the standard splitting (as in \cref{prop:std-splitting-higher-growth}). We will show that all the hypotheses of \cref{thm:combination-Wilon--Zalesskii} are satisfied and thus $G$ is conjugacy separable. 
 
We begin by noting that the vertex groups are of the form $F_w \rtimes_{\phi} \Z$ where $\phi \in \Aut(F_w)$ is a representative of a unipotent and polynomially growing outer automorphism with polynomial growth of degree $d_w \leq d-1$, and thus each vertex group is conjugacy separable by induction. Moreover, the edge groups are cyclic and thus conjugacy distinguished in the vertex groups by induction.

The profinite topology on $G \cong \pi_1(\calg, v)$ is efficient by \cref{prop:properties-of-std-splitting}. 

    The vertex groups are double-$\Z$-coset separable by \cref{prop:double-coset-unipotent}, and thus for any vertex $v$ and incident edge groups $\calg_e$ and $\calg_f$, and for any finite index subgroup $\calg_f' \leq_f \calg_f$, the double coset $\calg_e \calg_f' \subseteq G$ is separable. It follows by \cref{thm-Minasyan} that $\overline{\calg}_e \cap \overline{\calg}_f = \overline{\calg_e \cap \calg_f} = \widehat{\calg_e \cap \calg_f}$. 

    Finally, the group $G$ admits an epimorphism $\phi \colon G \to \Z$ that is surjective on edge groups. Hence, for any two edge groups $\calg_e$ and $\calg_f$, we must have that $\calg_e \cap \calg_f$ is either trivial or $\calg_e = \calg_f$. It follows by \cref{lemma:profinitely-acylindrical-splitting} that $\widehat{\calg}$ is profinitely 2-acylindrical.

    Hence, the conditions of \cref{thm:combination-Wilon--Zalesskii} are satisfied and thus $G$ is conjugacy separable.

    To show that cyclic subgroups are conjugacy distinguished, we apply \cref{lemma:conjugacy-distinguished-elliptic} and \cref{lemma:conj-distinguished-loxodromics}, noting that the hypotheses are satisfied by \cref{prop:properties-of-std-splitting}.
\end{proof}

\subsection{The bootstrap}

We will show that all polynomially growing \{finitely generated free\}-by-cyclic groups are conjugacy separable using the following:

\begin{thm}[{Chagas--Zalesskii \cite[Theorem~2.4]{ChagasZalesskii2010}}]\label{thm:Chagas--Zalesskii}
    Let $G$ be a finitely generated torsion-free group that admits a conjugacy separable normal subgroup of finite index. Suppose also that for every $g\in G \setminus 1$, the centraliser $C_{G}(g)$ is conjugacy separable and $\overline{C_G(g)} = \widehat{C_G(g)}$. Then $G$ is conjugacy separable.
\end{thm}

\begin{lemma}\label{lemma:centralisers}
    Let $G$ be \{finitely generated free\}-by-cyclic. Then for every $g \in G \setminus 1$, the centraliser $C_G(g)$ is either infinite cyclic or isomorphic to a \{finitely generated free\}-by-cyclic subgroup with finite order monodromy. 
\end{lemma}

\begin{proof}
    Let $G = F \rtimes_{\phi} \langle t \rangle$ where $F$ is finite rank free. Let $\chi \colon G \to \Z$ be the map that sends $F \mapsto 0$ and $t \mapsto 1$. Let $g \in G \setminus 1$. 
    
    Suppose first that $g \in F$. Then, 
    \[C_G(g) \cap F = C_F(g) \cong \Z.\]
    Hence, either $\chi(C_G(g)) = 0$, in which case 
    \[ C_G(g) = C_G(g) \cap F \cong \Z,\]
    or $\chi$ is non-trivial on $C_G(g)$, and 
    \[ C_G(g) \cong \Z \rtimes \Z.\]
    
    Suppose now that $\chi(g) \neq 0$. Then, 
    \[C_G(g) =  ( C_G(g) \cap F ) \rtimes \langle s \rangle,\]
    for some $s \in G \setminus F$. We must have that $g = y s^{l}$ for some $y \in C_G(g)\cap F$ and $l \in \Z \setminus 0$. Then $C_G(g)$ admits a finite index subgroup generated by 
    \[ \langle C_G(g) \cap F, ys^l \rangle  = (C_G(g) \cap F) \oplus \langle g \rangle.\]
    Let $\psi \in \Aut(F)$ be the automorphism induced by conjugation action by $g$. Then, $C_G(g) \cap F = \mathrm{Fix}_F(\psi)$. By the work of Gersten, $\Fix_F(\psi)$ is finitely generated  \cite{Gersten1987}. It follows that $\langle C_G(g) \cap F, ys^l \rangle \cong F' \times \Z$ where $F'$ is finite rank free. Hence, $C_G(g)$ is a \{finitely generated free\}-by-cyclic group with finite order monodromy. 
\end{proof}

We will need the following result due to G. Bartlett:

\begin{thm}[{Bartlett \cite[Proposition~2.7]{Bartlett2025}}]\label{prop:finite-order-conj-sep}
    Let $G$ be \{finitely generated free\}-by-cyclic with finite order monodromy. Then $G$ is conjugacy separable.
\end{thm}

We are now prove conjugacy separability of polynomially growing \{finitely generated free\}-by-cyclic groups.

\begin{thm}\label{thm:poly_conj_sep}
    Let $G$ be a \{finitely generated free\}-by-cyclic group with polynomially growing monodromy. Then $G$ is conjugacy separable. 
\end{thm}
\begin{proof}
    Every \{finitely generated free\}-by-cyclic group with polynomially growing monodromy admits a finite index normal subgroup that is \{finitely generated free\}-by-cyclic with UPG monodromy \cite[Corollary~5.7.6]{BestvinaFeighnHandel2005}. Thus, $G$ admits a finite index normal subgroup that is conjugacy separable by \cref{thm:unipotent-conj-sep}. 

    The centraliser of every non-trivial element is cyclic or \{finitely generated free\}-by-cyclic with finite order monodromy by \cref{lemma:centralisers} and thus conjugacy separable by \cref{prop:finite-order-conj-sep}. Moreover, every \{finitely generated free\}-by-cyclic or abelian subgroup $H$ of $G$ is fully separable by \cite[Proposition 2.9]{HughesKudlinska2023}, and thus we have that $\overline{C_G(g)} = \widehat{C_G(g)}$. 

    The result now follows from \cref{thm:Chagas--Zalesskii}.
\end{proof}

\section{The general case}\label{sec:general-case}

The aim of this section is to prove 
\Cref{main}.  We will first introduce the Dehn filling machinery that we will need to promote \Cref{thm:poly_conj_sep} to the general case.  We then prove that all \{finitely generated free\}-by-cyclic groups are conjugacy separable. We then recall work of Linton \cite{Linton2025,Linton_2026} describing the geometric structure of general free-by-cyclic groups.  Finally, we prove that all finitely generated free-by-cyclic groups are conjugacy separable.

\subsection{On relatively hyperbolic groups and Dehn fillings}

We begin by recalling several useful facts and definitions. 

A (non-necessarily finitely generated) group $\Gamma$  is said to be \emph{hyperbolic relative to} a finite collection of subgroups $\{H_i\}_{i\in I}$ if $\Gamma$ acts on a \emph{fine} (in the sense of Bowditch \cite{Bowditch00}) hyperbolic graph with finitely many orbits of edges, finite edge stabilisers, and such that infinite vertex stabilisers are exactly conjugates of the groups in $\{H_i\}_{i\in I}$ (see \cite[Definitions 3.4 and 3.5]{hruska_relative_2010}). 

Equivalently, $\Gamma$ is hyperbolic relative to $\{H_i\}_{i\in I}$ if there is a finite subset $S \subseteq \Gamma$ such that $\Gamma$ is generated by $ S \cup \bigcup_{i\in I} H_i$ and the coned-off Cayley graph of $\Gamma$ with respect to $S$ and cosets of the subgroups in $\{H_i\}_{i\in I}$, is hyperbolic and fine (see \cite[Definitions 3.4, 3.5]{hruska_relative_2010}).

If $\Gamma$ is hyperbolic relative to $\{H_i\}_{i\in I}$ then the family $\{H_i\}_{i\in I}$ is \emph{almost malnormal}, meaning that if $| H_i \cap H_j^g | = \infty$ for some element $g \in \Gamma$, then $i=j$ and $g\in H_i$.  
If, moreover, the groups $H_i$ are torsion free then the family is \emph{strictly malnormal}; that is, if $| H_i \cap H_j^g | \neq \{1\}$ for some element $g \in \Gamma$, then $i=j$ and $g\in H_i$. The conjugates of the $H_i$ are called \emph{peripheral subgroups}, and we say that an element of $\Gamma$ is \emph{loxodromic} if it has infinite order and is not in a peripheral subgroup.  

It is well known that if $\Gamma$ is relatively hyperbolic and $g \in \Gamma$ is loxodromic, then there exists a unique maximal virtually cyclic subgroup $E_{\Gamma}(g) \leq \Gamma$ containing $g$. If $\Gamma$ is moreover torsion free, then the group $E_{\Gamma}(g)$ is given by the centraliser of $g$ and it is the maximal cyclic subgroup containing $g$. 


We recall that if $\Gamma$ is relatively hyperbolic with respect to $\{H_i\}_{i\in I}$, and $g$ is loxodromic in this relatively hyperbolic structure, then  $\Gamma$ is relatively hyperbolic with respect to $\{ E_{\Gamma}(g)\}\cup\{ H_i\}_{i\in I}$ (see for instance \cite{Osin2006} or \cite[Lemma 4.4]{Dahmani2003} in the case that $\Gamma$ is finitely generated and torsion free).


\smallskip


For the remainder of this section, we take $\Gamma= F\rtimes \langle t \rangle$ to be finitely generated with $F$ a (non-necessarily finitely generated) free group. Moreover, we fix $\{ H_i \}_{i \in I}$ to be a finite collection of finitely generated subgroups of $\Gamma$, where $H_i = L_i\rtimes \langle t_i \rangle$, each $L_i$ is a subgroup of $F$, and $t_i \not \in F$. We assume that $(\Gamma, \{ H_i \}_{i \in I})$ is relatively hyperbolic.

We begin by observing that each $L_i$ has finitely many $F$-conjugates in its $\Gamma$-conjugacy class. Indeed, let $m_i \in \mathbb{Z}$ be the image of $t_i$ under the map $\Gamma \to \mathbb{Z}$ that sends $F \mapsto 0$ and $t \mapsto 1$. Then, for a fixed $i$ and for $J_i=\{0, \ldots, |m_i|-1\} $, the subgroups $L_{i,j} =  t^{-j}L_i t^{j}$, for $j\in J_i$, are the $F$-conjugacy representatives of the $\Gamma$-conjugacy class of $L_i$. 
Since $\{H_i\}_{i\in I}$ is a malnormal collection in $\Gamma$, the collection $\{L_{i,j}\}_{i\in I, j\in J_i}$ forms  a malnormal collection in the group $F$.

If $F$ and the subgroups $\{L_i\}_{i\in I}$ are finitely generated, then each $L_i \leq F$ is quasiconvex, and thus $F$ is relatively hyperbolic with respect to $\{L_{i,j} \}_{i\in I, j\in J_i}$, see \cite[Theorem 7.11]{Bowditch00}. If there exists a finite subset $S \subseteq F$ such that $F$ is generated by $S \cup \bigcup_{i\in I} L_i$, and if the $L_i$ are simultaneous free factors in a free product decomposition of $F$, then it is also the case that $F$ is relatively hyperbolic with respect to $\{L_{i,j} \}_{i\in I, j\in J_i}$.



The Dehn filling Theorem \cite{Osin2007, GrovesManning2008}  states that, for a collection of  subgroups $K_i\triangleleft L_i$ that are normal 
in $H_i$, if they are \emph{deep enough}, that is avoid a certain finite collection of non-trivial elements in the $L_i$, then the following statements hold.

\begin{enumerate} 
\item[(i)]  The quotient $\Gamma/ \llangle \bigcup_i K_i \rrangle_\Gamma $ is hyperbolic relative to the images of the $H_i$, which are isomorphic to $(L_i/K_i) \rtimes \langle t_i \rangle$. 
\item[(ii)] 
 The quotient $F/ \llangle \bigcup_{i,j}  t^{-j} K_i t^j \rrangle_F $  is hyperbolic relative to the images of the $L_{i, j}$. The images $\overline L_{i,j}$ are isomorphic to $(L_{i,j}/K_{i,j})$, for $K_{i,j} = t^{-j} K_i t^j$ (for this case, even though we will not use it, we mention that Osin's version can be applied even if $F$ is non finitely generated, see the comment after Theorem 1.1 in \cite{Osin2007}).
\item[(iii)] The family $\{H_i/K_i\}_{i\in I}$ of subgroups of  $\Gamma/ \llangle \bigcup_i K_i \rrangle_\Gamma $ is malnormal \cite[Proposition~2.7]{GrovesManning2021}.
\end{enumerate} 

The quotients $\Gamma/ \llangle \bigcup_i K_i \rrangle_\Gamma $  and $F/ \llangle \bigcup_{i,j}  t^{-j} K_i t^j \rrangle_F $ are respectively the Dehn fillings of $\Gamma$ and of $F$ by the families $\{K_i, i \in I\}$, and $ \{K_{i,j}, i \in I, j\in J_i\}$.


Observe that  $\llangle \bigcup_{i,j} t^{-j} K_i t^j  \rrangle_F = \llangle \bigcup_i K_i \rrangle_\Gamma$   and that     $\Gamma/ \llangle \bigcup_i K_i \rrangle_\Gamma $  is also canonically isomorphic to $(F/ \llangle \bigcup_{i,j} t^{-j} K_i t^j  \rrangle_F) \rtimes \mathbb{Z}$. 

\begin{prop} \label{prop;DF-separates-pairs}
Let $\Gamma$ be a torsion-free finitely generated group, $\mathcal{P} = \{ H_i= L_i\rtimes \langle t_i \rangle, i\in I \} $ a finite collection of finitely generated subgroups of $\Gamma$,  and assume that $(\Gamma, \mathcal{P})$ is  relatively hyperbolic. We also assume  that 
\begin{itemize} \item each $L_i$ admits arbitrarily deep finite-index normal subgroups that are invariant under conjugation by $t_i$, and
\item each $L_i\rtimes \langle t_i \rangle$ for $i\in I$ is conjugacy separable. 
\end{itemize}  
Let $g,h \in \Gamma$ be non-conjugate elements.  Then, for all sufficiently deep Dehn fillings of $\Gamma$ by subgroups $K_i \triangleleft L_i$ of finite index in $L_i$ invariant under conjugation by $t_i$, the images $\bar g, \bar h$ of $g,h$ in $\bar\Gamma = \Gamma /\llangle \bigcup_i K_i\rrangle_\Gamma$ are not conjugate. 
\end{prop}

\begin{proof} 

  We split into cases depending on whether $g$ and $h$ are loxodromic or elliptic with respect to the peripheral structure $(\Gamma, \mathcal{P})$. 

Suppose first that at least one of $g$ or $h$ is loxodromic, and let us assume, without loss of generality, that $g$ is loxodromic. Then, $\Gamma$ is hyperbolic relative to the collection $\{ E_{\Gamma}(g)\}\cup\{ H_i\}_{i\in I}$, and we will consider Dehn fillings of this structure for which the normal subgroup of $E_{\Gamma}(g)$ is trivial. By the Dehn filling Theorem, the group $E_{\Gamma}(g)$ embeds injectively in any such Dehn filling that is sufficiently deep.

Suppose that $h$ is loxodromic with respect to the collection $\{H_i \}_{i\in I}$ and not conjugate to an element in $E_{\Gamma}(g)$.  We consider $\Gamma$ as hyperbolic relative to the collection $\{ E_{\Gamma}(g)\}\cup\{ H_i\}_{i\in I}$. Then any Dehn filling of this structure that is deep enough  preserves the property that the image of $h$ is loxodromic in the quotient and hence is not conjugate into the image of $E_{\Gamma}(g)$. In particular the image of $g$ is not conjugate to the image of $h$. 

Suppose now that $h$ is elliptic with respect to the collection $\{H_i \}_{i\in I}$. Hence, $h$ is conjugate into some $H_i$. For any sufficiently Dehn filling of $\{ E_{\Gamma}(g)\}\cup\{ H_i\}_{i\in I}$, the image of the family $\{ E_{\Gamma}(g)\}\cup\{ H_i\}_{i\in I}$ is almost malnormal. Choose a Dehn filling using a collection $\{K_i\}_{i\in I}$ where each $K_i$ is a sufficiently deep subgroup of $L_i<H_i$ (it is a Dehn filling as in Proposition \ref{prop;DF_setting}). The image of $g$ has infinite order and is contained in the image of $E_{\Gamma}(g)$. In particular, the image of $g$ is not conjugate into any of the images of the $H_i$, and thus it is not conjugate to the image of $h$. 

Consider the case that $h$ is conjugate to an element in $E_{\Gamma}(g)$. After replacing $h$ by a conjugate, may assume that $h$ is contained in $E_{\Gamma}(g)$ and not conjugate to $g$. Once again, consider $\Gamma$ as hyperbolic relative to the collection $\{ E_{\Gamma}(g)\}\cup\{ H_i\}_{i\in I}$.  In a deep enough Dehn filling quotient (again using the trivial subgroups for $E_{\Gamma}(g)$), the group $E_{\Gamma}(g)$ has almost malnormal image, and embeds injectively. Since the image $\bar h$ of $h$  has infinite order, any element conjugating it to $\bar g$ conjugates the image $\overline{E_{\Gamma}(g)}$ of $E_{\Gamma}(g)$ to a group with infinite intersection with $\overline{E_{\Gamma}(g)}$. Therefore this element must be in $\overline{E_{\Gamma}(g)}$, and thus $\bar g$ and $\bar h $ are conjugate in $\overline{E_{\Gamma}(g)}$. However, $\overline{E_{\Gamma}(g)}$ is isomorphic to ${E_{\Gamma}(g)}$, contradicting that $g$ and $h$ are not conjugate in ${E_{\Gamma}(g)}$. Therefore the images of $g$ and $h$ are not conjugate in the Dehn filling quotient. 

We now consider the case where neither $g$ nor $h$ are loxodromic in $(\Gamma, \mathcal{P})$.  We may assume, after conjugating each of them separately, that they each belong to one of the finitely many subgroups in $\{ H_i\}_{i\in I}$.

If $g$ and $h$ both belong to different (hence non-conjugate) peripheral subgroups, we use the preservation of strict malnormality through Dehn fillings provided by \cite[Proposition~2.7]{GrovesManning2021} to conclude that their images are not conjugate.  

Assume now that $g$ and $h$ belong to the same peripheral subgroup $H_i$. By  assumption, $H_i$ is conjugacy separable. In particular, there is a finite index subgroup $M_i$ of $H_i$ such that $g$ and $h$ have non-conjugate images in $H_i/M_i$. Hence, if $K_i=M_i\cap L_i$, the images of $g$ and $h$ in the quotient $H_i/K_i$ are non-conjugate, and this remains true for the quotient of $H_i$ by any normal subgroup contained in $K_i$ (and hence for all sufficiently deep Dehn fillings).  Since for sufficiently deep Dehn fillings, the quotient map $\Gamma \to  \Gamma/\llangle \bigcup K_i \rrangle_\Gamma $ restricted to $H_i$ is the quotient map $H_i \to H_i/K_i$, if the images of $g$ and $h$ are conjugate in  $\Gamma/\llangle \bigcup K_i \rrangle_\Gamma $, then the conjugator is not contained in $H_i/K_i$, and thus the group $H_i/K_i$ is not malnormal in $\Gamma/\llangle \bigcup K_i \rrangle_\Gamma $. This contradicts the fact that the peripheral system remains malnormal in sufficiently deep Dehn fillings by \cite[Proposition~2.7]{GrovesManning2021}.
\end{proof}

\subsection{The case of \{finitely generated free\}-by-cyclic groups}

Recall  that if $F$ is a finitely generated free group, and if $\Gamma$ is a group of the form $F\rtimes \langle t \rangle$, then there exists a finite collection of finitely generated subgroups $L_i <F$, and elements $t_i\in \Gamma \setminus \{H_i\}$ normalising $L_i$ in $\Gamma$, such that  $\Gamma$ is hyperbolic relative to the collection $\{H_i = L_i \rtimes \langle t_i \rangle\}_{i\in I}$ and each $L_i \rtimes \langle t_i\rangle$ has polynomially growing monodromy \cite{Ghosh2023, DahmaniLi2022}.  We call the $H_i$ the \emph{maximal polynomially growing sub-mapping-tori}. 

As recalled above, we may define the finite collection $\{L_{i,j} =  t^{-j}L_i t^{j}\}_{i\in I,   j\in J_i}$ consisting of the $F$-conjugacy representatives of the $\Gamma$-conjugacy classes of the $L_i$. The subgroups $L_{i,j}$, for $ i\in I, j\in J_i$, are $\Gamma$-conjugacy representatives of the maximal polynomially growing subgroups of $F$ under the monodromy of the semidirect product.

\begin{prop}\label{prop;DF_setting} Let $F$ be a finitely generated free group, and  $\Gamma$ a group of the form $F\rtimes \mathbb{Z}$, with a collection of conjugacy representatives of maximal polynomially growing sub-mapping-tori  $\{L_i \rtimes \langle t_i\rangle \}_{i\in I}$.  For all sufficiently deep choices of characteristic finite index subgroups $K_i \triangleleft L_i$, the Dehn filling $\Gamma/ \llangle \bigcup_i K_i \rrangle_\Gamma$ is conjugacy separable. 
\end{prop}

\begin{proof}
 We consider characteristic finite index subgroups $K_i$ of $L_i$ that are deep enough in $L_i$ to ensure that the Dehn filling Theorem yields the consequences recalled above. 

 Now, the subquotients $L_i/K_i$ are finite, and therefore the group $F/ \llangle \bigcup_{i,j} t^{-j} K_i t^j  \rrangle_F$ is hyperbolic relative to finite groups.  Hence it is a hyperbolic group.  Moreover, the subgroups   $ \left(L_i/K_i\right)   \rtimes \langle t_i \rangle$  of  $\Gamma/ \llangle \bigcup_i K_i \rrangle_\Gamma $  are virtually cyclic, and therefore $\Gamma/ \llangle \bigcup_i K_i \rrangle_\Gamma $ is a hyperbolic group as well.

We have that  $\Gamma/ \llangle \bigcup_i K_i \rrangle_\Gamma   \cong \left(F/ \llangle \bigcup_{i,j} t^{-j} K_i t^j  \rrangle_F \right) \rtimes \mathbb{Z}$.  Thus, the quotient group $\Gamma/ \llangle \bigcup_i K_i \rrangle_\Gamma $ is a hyperbolic hyperbolic-by-cyclic group. By \cite[Corollary~5.4]{DahmaniKrishnaMutanguha2025}, it is then virtually compact special. By  \cite[Theorem~1.1]{MinasyanZalesskii2016}, any hyperbolic virtually compact special group is conjugacy separable.  Hence  $\Gamma/ \llangle \bigcup_i K_i  \rrangle_\Gamma$ is conjugacy separable. 
\end{proof}

\begin{corollary}\label{coro;CS_fgf_bc}
    Any \{finitely generated free\}-by-cyclic group is conjugacy separable. 
\end{corollary}

\begin{proof}
Let $\Gamma = F \rtimes \langle t \rangle$ be such that $F$ is finitely generated free and $ \langle t \rangle \cong \mathbb{Z}$. We have that $\Gamma$ is relatively hyperbolic with respect to the maximal polynomially growing sub-mapping-tori $L_i \rtimes \langle t_i\rangle, i\in I$. Each $L_i$ is finitely generated free, and therefore admits arbitrarily deep finite index normal subgroups invariant under conjugation by $t_i$. By \Cref{thm:poly_conj_sep}, each $L_i \rtimes \langle t_i\rangle$, $i\in I$ is conjugacy separable. By \Cref{prop;DF-separates-pairs}, for any two non-conjugate elements $g,h\in\Gamma$, there are arbitrarily deep  Dehn fillings of $\Gamma$ in which their images are not conjugate. By \Cref{prop;DF_setting}, at least one of these Dehn fillings is conjugacy separable and therefore admits a further finite quotient in which $g$ and $h$ have non-conjugate images. It follows that $\Gamma$ is conjugacy separable.
\end{proof}

\subsection{The case of \{non-finitely-generated free\}-by-cyclic groups} \label{sec;final}

Let $F$ be a free group of infinite rank and suppose that $\Gamma = F \rtimes_{\phi} \Z$ is finitely generated. As before, $t \in \Gamma$ is a generator of the $\Z$-factor which induces the automorphism $\phi \in \Aut(F)$.

By the work of Linton \cite[Theorem 1.1]{Linton_2026}, the free group $F$ admits a free splitting of the form $A \ast \left( \aster_{j \in \mathbb{Z}} C_j \right)$ where $A$ and $C_0$ are finitely generated, and for all  $j \in \mathbb{Z}$, we have  $\phi^j(C_0) = C_{j}$.  We call the subgroup $\langle C_0, t\rangle$, the \emph{mapping telescope of $C_0$} and we denote it by $T$. 

By \cite[Theorems 1.1 and 5.1]{Linton2025}, $\Gamma = F\rtimes_\phi \langle t\rangle$ is hyperbolic relative to a finite collection $\mathcal{P}_\Gamma$ of subgroups that are mapping tori $H_i= A_i \rtimes \langle t_i\rangle$  of finitely generated free groups $A_i \leq F$, and the groups $A_i$ can be chosen as a collection of simultaneous free factors of $A$.  Moreover, $\Gamma$ is locally relatively quasiconvex for this peripheral structure. 

\begin{lemma}\label{lem;ridges_on_T}
    The mapping telescope $T= \langle C_0, t\rangle$ is malnormal in $\Gamma$. 
    Moreover, the intersections of $T$ with conjugates of the subgroups $H_i\in \mathcal{P}_\Gamma$ fall in finitely many $T$-conjugacy classes, and they are all cyclic subgroups intersecting $F$ trivially.



\end{lemma}

\begin{proof}
    If an element $ft^k$ conjugates $g\in \langle C_0, t\rangle$ to $g'\in \langle C_0, t\rangle$, then there are $c, c' \in  \aster_j C_j $ and $r$ such that $fct^r = c't^r f$.  This means $fc = c'\phi^r(f)$. If $f$ is not in $\aster_j C_j $, neither is  $\phi^r(f)$ and the free product decomposition is violated. Therefore we have the malnormality of $\langle C_0, t\rangle$. 

    Assume that for  $\gamma\in \Gamma$, we have $f \in \left( \aster_{j \in \mathbb{Z}} C_j \right) \cap \gamma^{-1}A_{i_0}\gamma$. The conjugation action of $\Gamma$ permutes the $F$-conjugacy classes of the $A_i$ and thus 
    the subgroup $\gamma^{-1}A_{i_0}\gamma$ is the $F$-conjugate of some $A_i$. The retraction of $F$ on the free factor $A_i$ then forces $f=1$. Therefore, if  $\langle C_0, t\rangle$  intersects a conjugate of some $H_i$ non-trivially, then the intersection lies outside $F$.

    Linton's local relative quasiconvexity ensures that $T$ is relatively quasiconvex in $\Gamma$.  Hence, $T$ is itself hyperbolic relative to the intersection with the peripheral structure of $\Gamma$. This forces there to be only finitely many  $T$-conjugacy classes of intersections with conjugates of the subgroups $H_i\in \mathcal{P}_\Gamma$.
\end{proof}

We recall another result of Linton \cite[Theorem 1.2]{Linton_2026}: there exists a finitely generated free group $F_0$ with an embedding $\iota : F \hookrightarrow F_0$, and an automorphism $\theta : F_0 \to F_0$ such that $\theta \circ \iota = \iota \circ \phi$. It will be convenient to recall, with a slight modification, Linton's construction from \cite[Section 4.4]{Linton_2026}.

Recall that $\Gamma = F \rtimes_{\phi} \mathbb{Z}$ is finitely generated and $F$ admits a free splitting $A \ast \left( \aster_{j \in \mathbb{Z}} C_j \right)$ where $A$ and $C_0$ are finitely generated, and $\phi^j(C_0) = C_j$ for all $j \in \Z$. 

The free group $F_0$ is defined as the free product $A \ast D$, where $D$ is a finite-rank free group chosen to satisfy the following conditions:
\begin{enumerate}
    \item $D$ admits an atoroidal, fully irreducible automorphism $\theta_0$, 
    \item $D$ contains a finitely generated subgroup $D_0 \leq D$ such that $\rank(D_0) = \rank(D)$, and
    \item $D_0$ is quasiconvex and malnormal in the hyperbolic group $D\rtimes_{\theta_0} \mathbb{Z}$.
\end{enumerate}

Note that Linton's construction does not ensure malnormality, but a choice of a proper free factor of $D$, with an appropriate automorphism with displacement larger than $3$ in the free factor complex allows this.   

Consequently, the subgroup of $D$ generated by the images of $\theta_0^i(D_0)$ for $i\in \mathbb{Z}$ naturally splits as the free product $\ast_{j\in \mathbb{Z} } \theta_0^j(D_0)$. The monomorphism $\iota \colon F \hookrightarrow F_0$ is defined on the free factors as follows. The restriction $\iota|_A$ is the identity on $A$. The restriction $\iota|_{C_0}$ maps $C_0$ isomorphically to $D_0$ via a choice of a basis. The restriction $\iota|_{C_j}$ is defined to be $\iota|_{C_j} = \theta_0^j \circ \iota|_{C_0} \circ \phi^{-j}$, for all $j \in \Z$.

Finally, we wish to extend the automorphism $\theta_0$, thus far only defined only on $D$, to all of $F_0$. We do so by setting $\theta(a) = \iota (\phi(a))$ for $a \in A.$ To check that this yields a well-defined automorphism of $F_0$, it suffices to check that $\theta$ is surjective on $A$. This follows by commutativity of the morphisms $\iota \circ \phi=  \theta\circ \iota $ and the surjectivity of $\iota$ on $A$. 

By taking appropriate powers, we can ensure that the exponential growth rate of $\theta_0$ is strictly larger than that of the restriction of $\phi$ to the subgroup $A$. 
Furthermore, may choose $\theta_0$ to be ageometric, in particular admitting a train track representative  without any periodic Nielsen paths, by \cite[Coro. B, \& \S 3.4]{KapovichMaherPfaffTaylor}. 

It is now an easy check that $\Gamma = F \rtimes_\phi \langle t \rangle $ embeds into $\Gamma_0 = F_0\rtimes_\theta  \langle t_0\rangle$, extending the map $\iota \colon F \to F_0$ by setting $\iota(t) = t_0$. Moreover, we have the decomposition  \[\Gamma_0 = F_0\rtimes_\theta \mathbb{Z} \cong \left(F\rtimes_\phi \mathbb{Z}\right) \ast_{\langle C_0, t\rangle}  \left( D\rtimes_{\theta_0} \mathbb{Z} \right). \] 
 By malnormality of the edge group in the first factor (\Cref{lem;ridges_on_T}), it follows that the splitting is acylindrical. 

\begin{lemma}\label{lem:Gamma0_hyp_rel}
    The group $\Gamma_0$ is hyperbolic relative to the images $\iota (A_i \rtimes \langle t_i \rangle)$. 
\end{lemma}

\begin{proof}
	The free factor $A$ of $F$ splits as $A_1\ast A_2 \ast \dots A_r \ast B$.  The group $\Gamma_0$ is obtained as the mapping torus of $A_1\ast A_2 \ast \dots A_r \ast B\ast D$ by the automorphism $\theta$. First we argue that $\theta$ has no twinned subgroups (in the sense of \cite{DahmaniLi2022}) in the free factor system $\{ A_1, \dots, A_r\}$.  Indeed, a pair of \emph{twinned subgroups} is a pair $L, L'$, conjugated to the $A_i, i \in I$, such that there exists $f\in F_0$ and $r>0$ where $\ad_f\circ \phi^r $ preserves $\langle L, L'\rangle$. Such a pair would contradict Linton's maximality of the $A_i$ \cite[Theorem 1.1 (1)]{Linton2025}.   
    
    We argue that $\theta$ is atoroidal with respect to the free factor system $\{ A_1, \dots, A_r\}$. Assume that it preserves a conjugacy class $[w]$, with $w$ cyclically reduced. For the same reason used for twin subgroups, either $w$ is conjugated into one of the $A_i$, or the word $w$ contains letters in $D$. By \cite[Theorem 6.0.1]{bestvina_tits_2000}, in a train track representative of $\theta$ in which $D$ is supported by an exponentially growing stratum, since $\theta_0$ is ageometric and so has no periodic Nielsen path, $w$ actually labels a path in the complement of this stratum.  In particular, $w$ does not contain a letter in $D$ and hence is conjugated into one of the $A_i$. 

    We may then apply \cite[Theorem 2]{DahmaniLi2022} in order to obtain the relative hyperbolicity of $\Gamma_0$. 
 \end{proof}

\begin{lemma}\label{lem;T_hatT_and_E} 
The subgroup $\iota(T)$ of $D\rtimes_{\theta_0} \mathbb{Z}$ is quasiconvex in the hyperbolic group $D\rtimes_{\theta_0} \mathbb{Z}$.  
The group $T$ is relatively quasiconvex in the relatively hyperbolic group $(\Gamma, \mathcal{P}_\Gamma)$. 
\end{lemma}
\begin{proof}
For the first claim, note that by construction $\iota(T) =  \left( \aster_{j\in \mathbb{Z} } \theta_0^j(D_0) \right)\rtimes \langle t\rangle$. Hence, this is a statement that does not depend on $\phi$, only on the choice of $D_0$ and $\theta_0$. The quasiconvexity of $\iota(T)$ in  $D\rtimes_{\theta_0} \mathbb{Z}$ was then proven at the end of Section 4.4 in \cite{Linton_2026}. The argument can be adapted as follows. 

The subgroup $D_0$ is quasiconvex in $D\rtimes_{\theta_0} \mathbb{Z}$, since it is a finitely generated subgroup of infinite index in the fiber $D$ and $\theta_0$ is atoroidal and fully irreducible \cite[Theorem 3.5]{Mitra_1999}. Then the group generated by $D_0$ and a large power of $t$ is quasiconvex and isomorphic to the free product $D_0\ast \langle t^m\rangle$ by \cite[Theorem 1]{Arz}. Up to passing to a finite index subgroup by replacing the stable letter $t$ by a power $t^m$, we may assume that $m=1$ and the result follows. 

The second claim follows from the local relative quasiconvexity in Linton's theorem \cite[Theorem 1.1(3)]{Linton2025}. 
\end{proof}

\begin{prop}\label{prop;RQCembedding}
    The subgroup $\iota(\Gamma) $ is fully relatively quasiconvex in $(\Gamma_0, \mathcal{P}_{\Gamma_0})$, for the peripheral structures of the maximal polynomially growing mapping tori of $\Gamma_0$.
\end{prop}

\begin{proof}

We will use a result of Pal \cite[Theorem 1.2]{Pal_2021}: \emph{Suppose we are given a relatively hyperbolic group that is an amalgamation of two finitely generated relatively hyperbolic groups over a finitely generated group.  Suppose further that both inclusions of the amalgamated subgroup quasi-isometric embeddings and the images are relatively quasiconvex.  If each vertex group has finite height, then the vertex groups are relatively quasiconvex.}

In our case we have $\Gamma_0 = \left(F\rtimes_\phi \mathbb{Z}\right) \ast_{\langle C_0, t\rangle}  \left( D\rtimes_{\theta_0} \mathbb{Z} \right)$.  The relative hyperbolicity and relative quasiconvexity are provided by \Cref{lem:Gamma0_hyp_rel} and \Cref{lem;T_hatT_and_E}, respectively. The embedding of the edge group in both vertex groups is a quasi-isometric embedding since the peripheral part of the edge group consists of undistorted cyclic groups.  The vertex groups have finite height because the edge group itself has finite height in the vertex groups \cite{HruskaWise_2009}, and the splitting is $2$-acylindrical (by malnormality of the edge group in the vertex group  $F\rtimes_\phi \mathbb{Z}$). It therefore follows that the subgroup $\iota(\Gamma) $ is fully relatively quasiconvex in $(\Gamma_0, \mathcal{P}_{\Gamma_0})$.
 \end{proof}

\begin{prop}\label{prop;stillVCS}
    Any sufficiently deep Dehn filling on $\Gamma_0$ induces a quotient on $\iota(\Gamma)$ which is a Dehn filling. More precisely, for any collection of non-trivial elements $\mathscr{X}$ to avoid in $\Gamma$, there is a collection of non-trivial elements  $\mathscr{X}_0$ to avoid in $\Gamma_0$ such that a Dehn filling of $\Gamma_0$ avoiding  $\mathscr{X}_0$  induces a Dehn filling of $\iota(\Gamma)$ avoiding  $\iota(\mathscr{X})$, and such that the image of $\iota(\Gamma)$ in $\overline{\Gamma_0}$ is relatively quasiconvex.   

    In particular if the Dehn filling of $\Gamma_0$ is hyperbolic, and hyperbolic-by-cyclic, then the image of $\iota(\Gamma)$ is quasiconvex in a virtually compact special group   $\overline{\Gamma_0}$, and itself virtually compact special.
\end{prop}
\begin{proof}
The first part of the statement is a consequence of \cite[Proposition 4.4]{AGM_2009}, using the full relative quasiconvexivity of $\iota(\Gamma)$ in $\Gamma_0$ from \Cref{prop;RQCembedding}.

For the second part of the statement, observe that by \cite[Corollary~5.4]{DahmaniKrishnaMutanguha2025}, the Dehn filling of $\Gamma_0$ is hyperbolic and virtually compact special and the image of $\iota(\Gamma)$ is quasiconvex. It follows that $\iota(\Gamma)$ is virtually compact special by \cite[Corollary~7.8]{HaglundWise_2008}. 
\end{proof}

\begin{duplicate}[\Cref{main}]
Every finitely generated free-by-cyclic group is conjugacy separable. 
\end{duplicate}

We now can present a  proof that is similar to that of Corollary \ref{coro;CS_fgf_bc}, except for the references used. 

\begin{proof}
Let $\Gamma$ be  a finitely generated free-by-cyclic group, written as $\Gamma = F\rtimes \langle t\rangle$, with the relatively hyperbolic structure $\mathcal{P}_\Gamma$ from Linton's theorem \cite[Theorem 1.1]{Linton2025}. Each peripheral subgroup is conjugate to some $A_i\rtimes \langle t_i\rangle$, with $A_i \leq F$ finitely generated free. Hence, each $A_i$ admits  arbitrarily deep finite index normal subgroup invariant under conjugation by $t_i$. By \Cref{coro;CS_fgf_bc},  each such $A_i \rtimes \langle t_i\rangle$, for $i\in I$, is conjugacy separable. Given two non-conjugate elements $g$ and  $h$,  \Cref{prop;DF-separates-pairs} ensures the existence of arbitrarily deep Dehn fillings in which their images are non-conjugate.  By \Cref{prop;stillVCS}, at least one of these Dehn fillings is virtually compact special, hence conjugacy separable by \cite{MinasyanZalesskii2016}, and therefore admits a further finite quotient in which $g$ and $h$ have non-conjugate images. 
\end{proof}

We now prove \Cref{cor:rf-outer-automorphism} from the introduction. 

\begin{duplicate}[\Cref{cor:rf-outer-automorphism}]
    If $G$ is finitely generated free-by-cyclic then $\Out(G)$ is residually finite. 
\end{duplicate}
\begin{proof}
  By the work of Grossman \cite{Grossman1974}, if $G$ is a finitely generated conjugacy separable group such that every pointwise inner automorphism is inner, then $\Out(G)$ is residually finite. The latter condition holds for all torsion-free acylindrically hyperbolic groups by \cite[Corollary~1.5]{AntolinMinasyanSisto2016}. 
  
  Any finitely generated free-by-cyclic groups with non-periodic monodromy is acylindrically hyperbolic by the work of Genevois--Horbez \cite[Corollary 1.5]{GenevoisHorbez2021} and Linton \cite{Linton2025}.  This establishes the result in the cases when the monodromy is non-periodic.  If the monodromy is periodic, then since $G$ is finitely generated it must be the case that the free fiber is finitely generated. We apply \cite[Theorem 4.4]{Levitt2007} noting \cite[Corollary 2.2]{KrsticVogtmann1993}.  \end{proof}

\bibliographystyle{halpha}
\bibliography{refs.bib}

\end{document}